\documentclass[leqno]{amsart}
\usepackage{amsmath} 
\usepackage{amssymb,mathtools,stmaryrd}
\usepackage{mathrsfs,euscript}
\usepackage[table,dvipsnames]{xcolor}
\usepackage{hyperref,tikz-cd,enumitem}
\usepackage[utf8]{inputenc}
\hypersetup{
 colorlinks=true,
 linkcolor=DarkOrchid,
 filecolor=blue,
 citecolor=olive,
 urlcolor=orange,
 pdftitle={Pask\={u}nas' theory},
}

\calclayout

\newtheorem{thm}{Theorem}[section]
\newtheorem{lem}[thm]{Lemma}
\newtheorem{prop}[thm]{Proposition}
\newtheorem{cor}[thm]{Corollary}
\newtheorem{conj}[thm]{Conjecture}

\theoremstyle{definition}
\newtheorem{defn}[thm]{Definition}

\theoremstyle{remark}
\newtheorem{rem}[thm]{Remark}

\newcommand{\smat}[1]{\left(\begin{smallmatrix} #1 \end{smallmatrix}\right)}
\newcommand{\id}{\mathbf{1}}
\newcommand{\oo}{\mathcal{O}} 
\newcommand{\eo}{\EuScript{O}}
\newcommand{\fF}{\mathbb{F}} 

\newcommand{\Q}{{\mathbf{Q}}}
\newcommand{\Z}{{\mathbf{Z}}}
\newcommand{\Qp}{\mathbf{Q}_p}
\newcommand{\Zp}{\mathbf{Z}_p}

\newcommand{\C}{\mathbf C}
\newcommand{\A}{\mathbf A}
\newcommand{\dd}{\mathfrak{d}} 
\newcommand{\arch}{\mathbf{a}}
\newcommand{\finite}{\mathbf{h}}
\DeclareMathOperator{\Nr}{N}
\DeclareMathOperator{\Tr}{Tr}

\DeclareMathOperator{\End}{End}

\DeclareMathOperator{\Hom}{Hom}
\DeclareMathOperator{\Ext}{Ext}

\DeclareMathOperator{\Ind}{Ind}
\DeclareMathOperator{\cInd}{c-Ind}

\DeclareMathOperator{\Res}{Res}
\DeclareMathOperator{\Cor}{Cor}
\DeclareMathOperator{\Image}{Im}
\DeclareMathOperator{\coker}{coker}
\DeclareMathOperator{\rank}{rank}

\DeclareMathOperator{\Sym}{Sym}

\DeclareMathOperator{\GL}{GL}

\DeclareMathOperator{\UU}{U}

\DeclareMathOperator{\mtr}{tr}

\DeclareMathOperator{\val}{val} 
\DeclareMathOperator{\ind}{ind} 

\DeclareMathOperator{\Spec}{Spec}

\DeclareMathOperator{\length}{length}

\DeclareMathOperator{\Gal}{\mathcal{G}}
\DeclareMathOperator{\WD}{WD}
\DeclareMathOperator{\Rec}{Rec}

\DeclareMathOperator{\Art}{Art}
\newcommand{\Fr}{\textnormal{Fr}} 
\newcommand{\dR}{\textnormal{dR}}
\newcommand{\pst}{\textnormal{pst}}

\newcommand{\cts}{\textnormal{cts}}

\newcommand{\fa}{\mathfrak{a}}
\newcommand{\fc}{\mathfrak{c}}
\newcommand{\ff}{\mathfrak{f}}

\newcommand{\fl}{\mathfrak{l}}
\newcommand{\fm}{\mathfrak{m}}
\newcommand{\fn}{\mathfrak{n}}
\newcommand{\fp}{\mathfrak{p}}
\newcommand{\fq}{\mathfrak{q}}
\newcommand{\fs}{\mathfrak{s}}

\DeclareMathOperator{\Mod}{\textnormal{Mod}}
\DeclareMathOperator{\fC}{\mathfrak{C}} 
\DeclareMathOperator{\Ban}{\textnormal{Ban}_{G,\zeta}^{\adm}}
\DeclareMathOperator{\Rep}{Rep}
\DeclareMathOperator{\V}{\check{\mathbf{V}}} 
\DeclareMathOperator{\Ord}{Ord} 
\DeclareMathOperator{\Irr}{Irr}
\DeclareMathOperator{\soc}{soc}

\newcommand{\Gp}{\mathcal{G}_{\Qp}} 
\newcommand{\B}{\mathfrak B} 

\newcommand{\sm}{\textnormal{sm}}
\newcommand{\adm}{\textnormal{adm}}
\newcommand{\ladm}{\textnormal{ladm}}
\newcommand{\lfin}{\textnormal{lfin}}

\newcommand{\red}{\textnormal{red}}

\newcommand{\xx}{x_\textnormal{red}}

\newcommand{\F}{{\mathcal{F}}} 
\newcommand{\K}{{\mathcal{K}}} 
\newcommand{\qch}{\epsilon} 
\newcommand{\bw}{{\overline{w}}}
\newcommand{\bl}{{\bar{\fl}}}

\newcommand{\cG}{\mathcal{G}}
\newcommand{\fG}{\mathfrak{G}}
\newcommand{\fX}{\mathfrak{X}}

\newcommand{\wt}[1]{\underline{ #1 }}
\newcommand{\Iw}{\textnormal{Iw}} 
\newcommand{\TT}{\mathbb{T}} 
\newcommand{\euF}{\EuScript{F}} 
\newcommand{\I}{\mathcal{I}} 
\newcommand{\ord}{\textnormal{ord}} 

\DeclareMathOperator{\loc}{loc} 
\DeclareMathOperator{\Sel}{Sel} 
\DeclareMathOperator{\car}{char} 
\newcommand{\lin}[1]{\mathcal{L}^{(#1)}} 
\newcommand{\Lda}[1]{\Lambda^{(#1)}} 

\begin{document}
\title{Anticyclotomic Euler systems for CM fields}

\begin{abstract}
    Let $\K/\F$ be a CM extension satisfying the ordinary assumption for an odd prime $p$
    and let $\psi$ be a finite order anticyclotomic Hecke character
    of $\K$.
    When $\K$ has a place above $p$ of degree one,
    we apply the technique from \cite{urban}
    and the results from \cite{lee} to construct 
    an anticyclotomic Euler system for $\psi$
    under minor assumptions
    and prove one side o the divisibility of
    the anticyclotomic Iwasawa main conjecture for $\psi$
    when $p$ is inverted.
\end{abstract}

\author[Y-S.~Lee]{Yu-Sheng Lee}
\address{Department of Mathematics, University  of Michigan, Ann Arbor, MI 48109, USA}
\email{yushglee@umich.edu}
\date{\today}

\maketitle
\setcounter{tocdepth}{1}
\tableofcontents

\section*{Introduction}

Let $\K$ be a CM field with maximal totally real subfield $\F$.
Fix an odd prime $p>0$ which is unramified in $\F$
such that every prime of $\F$ above $p$ splits in $\K$.
Then there exists a subset $\Sigma\subset \Hom(\K,\C)$
such that if $S_p^\K$ denote the set of primes of $\K$ above $p$
and $\Sigma_p\subset S_p^\K$ is the subset of primes
induced by embeddings in $\Sigma$,
through a fixed choice of embeddings
$\iota_\infty\colon \K\to \C$ and
$\iota_p\colon \K\to \C_p$, then we have
\[
    \Hom(\K,\C)=\Sigma\sqcup \Sigma^c,\quad
    S_p^\K=\Sigma_p\sqcup \Sigma_p^c
\]
where $\Sigma^c=c\circ \Sigma$ and $\Sigma_p^c=c\circ \Sigma_p^c$
for the complex multiplication $c\in Gal(\K/\F)$.

Given a finite order character $\psi$
of the absolute Galois group $\Gal_\K$ of $\K$.
The Iwasawa main conjecture for $\K$,
under Greenberg's formulation,
relates the following two objects associated to $\psi$.
On the one hand, 
by class field theory the Galois group of
the maximal pro-$p$ abelian extension which is unramified 
away $p$ has a free part $W$ of $\Zp$-rank $d+1-\delta$,
where $d=[\F:\Q]$ and $\delta$ is the Leopoldt defect.
We let $\K_\infty/\K$ denote the extension such that 
$Gal(\K_\infty/\K)\cong W$.
Define $\Lambda_W=\eo\llbracket W\rrbracket$
for the ring of integers $\eo$ of a sufficiently large
field extension $E$ over $\Qp$
and let $\Gal_\K$ acts on 
the Pontryagin dual $\Lambda_W^*$ of which through
$\psi\langle*\rangle$, where 
$\langle*\rangle\colon \Gal_\K\to W\to \Lambda^\times$
is the tautological character. 
Consider the Selmer group
\[
    \Sel(\psi,\Sigma_p)=\ker
    \big\{
    H^1(\K, \Lambda_W^*)\to \prod_{w\notin \Sigma_p}
    H^1(I_w, \Lambda_W^*)
    \big\}
\]
where $w$ ranges through all finite primes of $\K$
outside $\Sigma_p$
and $I_w$ the inertia group at $w$.
Then the Pontryagin dual 
$X(\Lambda_W)\coloneqq \Sel(\psi,\Sigma_p)^\vee$
is a finitely generated torsion $\Lambda$-module.
And to such modules we can consider the characteristic ideal
$F(\Lambda_W)\coloneqq \car_{\Lambda_W}(X(\Lambda_W))$.

On the other hand, fix a CM type $\Sigma$ as above.
By the work Katz \cite{Katz1978} 
and Hida-Tilouine \cite{HT93} there exists 
a $p$-adic L-function
$L_p(\psi,\Sigma_p)\in \Lambda_W$ 
that interpolates the algebraic part
of the $L$-value $L(0,\psi\alpha)$,
when $\alpha$ are among characters of $W$
are critical in the sense of Deligne's conjecture.

\begin{conj}
The characteristic ideal $F(\Lambda_W)$
is generated by $L(\psi,\Sigma_p)$ in $\Lambda_W$.
\end{conj}

The same conjecture can also be formulated if we replace
$W$ by the Galois group over $\K$
of subextensions in $\K_\infty/\K$.
But unless the subextension contains the 
cyclotomic $\Zp$-extension,
the Pontryagin dual of the Selmer group defined above
could fail to be torsion.

In particular, the case when 
$\psi$ is anticyclotmic and
$W=\Gal(\K_\infty^a/\K)$ is the Galois group of the maximal 
anticyclotomic subextension $\K_\infty^a$ is considered
in \cite{HT93} and \cite{HT94}.
Here $\K_\infty^a$ is anticyclotomic in the sense that 
$c\in\Gal(\K/\F)$ acts by the inversion on $W$.
In this case $W$ has the $\Zp$-rank $d=[\F:\Q]$.
And we say $\psi$ is anticyclotomic if 
$\psi(\gamma^c)=\psi^{-1}(\gamma)$ for $\gamma\in\Gal_\K$.
Using the congruences between modulars forms
with CM by $\K$ and those without,
it is shown in \textit{loc.cit.}
that the characteristic ideal $F(\Lambda_W)$
belongs to the ideal generated by $L_p(\psi,\Sigma_p)$,
the projection of the Katz-Hida-Tilouine $p$-adic $L$-function
to the anticyclotomic space $\eo\llbracket W\rrbracket$.
Furthermore, under the following  assumptions on $\psi$:
\begin{itemize}
    \item The order of $\psi$ is coprime to $p$.
    \item The prime-to-$p$ part of the conductor of $\psi$
    is a product of primes split in $\K$.
    \item The restriction $\psi\vert_{D_w}$ is nontrivial
    at the decomposition group $D_w$ of each $w\in\Sigma_p$.
    \item The restriction of $\psi$ to 
    the absolute Galois group of $\K(\sqrt{(-1)^{(p-1)/2}p})$
    is nontrivial.
\end{itemize}
the crucial torsion property is obtained in
\cite[Thm.5.33]{Hida06b} by relating the Selmer group to 
Galois deformation problems
and the full main conjecture 
is proved in \cite{Hida06}.

In the current article, we consider the same
anticyclotomic main conjecture under
a different set of assumptions on $\psi$.
One of our main result is as follows.

\begin{thm}

Assume there exists a place $w_0\in\Sigma_p$ of degree one
and the following conditions
\begin{enumerate}[label=($\K$\arabic*)]
\item The extension $\K/\F$ is generic 
in the sense of \cite{Rohrlich}.
\label{cond:K2in}
\item Every prime of $\F$ above $2$ splits in $\K$.
\label{cond:K3in}
\end{enumerate}
Identify the decomposition group $D_{w_0}$
with the absolute Galois group $\Gp$ of $\Qp$
and let $\omega\colon \Gp\to (\Z/p\Z)^\times$ be the Teichmuller character.
Then $X(\Lambda)$ is torsion and the $p$-adic $L$-function belongs to 
the characteristic ideal $F(\Lambda_W)$ in 
$E\llbracket W\rrbracket$ under the following conditions.
\begin{enumerate}[label=($\psi$\arabic*)]
\item  The restriction $\psi\vert_{\Gp}$ is not congruent
to the trivial character $\id$ or $\omega^{\pm1}$.
\label{cond:psi1in}
\item There exists a finite order character $\mu$ of $\Gal_\K$
which is unramified at all inert places and 
$\mu\vert_{\oo_w^\times}$ is the unique nontrivial
quadratic character at all ramified places such that
$\psi(\gamma)=\mu^{c-1}(\gamma)\coloneqq\mu(\gamma^c)/\mu(\gamma)$.
\label{cond:psi2in}
\end{enumerate}
\end{thm}

We remark that our proof for the torsion property
does not rely on global Galois deformation techniques.
Instead, we construct 
an anticyclotomic Euler system associated to $\psi$.
Then the general machinery from \cite{Rubin}
allows us to bound the size
of the Selmer group from above and obtain the theorem.
This is the reason that the inclusion relation
we obtained is opposite to that of \cite{HT94},
which directly constructs elements in the Selmer 
group and gives a lower bound to the Selmer group
in terms of the $p$-adic $L$-function.

To state the precise result 
we introduce the following notations.
Let $\oo_\K$ be the ring of integers of $\K$
and let $\ff\subset \oo_\K$ be a square-free ideal
consists only of split primes
with $\ff+\ff^c=\oo_\K$. 
We write $\fs=\ff\ff^c$
and let $\fG_\fs^a$ be the Galois group of the maximal 
pro-$p$ abelian extension over $\K$ that is unramified 
away primes dividing $\fs p$.
Let $\I_\fs=\eo\llbracket \fG_\fs^a\rrbracket$
and consider the character $\Psi_\fs=\psi^{-1}\langle*\rangle_\fs$,
where $\langle*\rangle_\fs\colon \Gal_\K\to \fG_\fs\to \I_\fs^\times$
is the tautological character.

\begin{thm}

Under the same assumptions on $\K$ and $\psi$ as above,
there exists a collection of cohomology classes
$\{\mathcal{Z}_\fs\in H^1(\K, \I_\fs(\Psi_\fs^{-1}))\}_{\fs\in\mathcal{R}}$
that has the following properties
for ideals $\fs$ as above varying 
in a certain infinite set $\mathcal{R}$.
\begin{enumerate}
    \item There exists a finite set $S$ of primes of $\K$
    that contains all the places above $p$ and all the 
    places where $\psi$ is ramified, such that each
    $\mathcal{Z}_{\fs}$ is unramified at places $w\notin S$
    that are coprime to $\fs$.
    \item For $w'\in \Sigma_p\setminus\{w_0\}$
    and $\sigma\in D_{w'}$, the restriction
    $\loc_{w'}(\mathcal{Z}_\fs)\in H^1(\K_{w'}, \I_\fs(\Psi_\fs^{-1}))$
    satisfies 
    \[
        (\Psi_{\fs}^{-1}(\sigma)-1)\cdot \loc_{w'}(\mathcal{Z}_\fs)=0
    \]
    \item For $w=w_0$, the restriction 
    $\loc_{w}(\mathcal{Z}_\fs)\in H^1(\Qp, \I_\fs(\Psi_\fs^{-1}))$
    is equal to $L_\fs\cdot \mathcal{Z}_{p,\fs}$ for
    a certain $p$-adic $L$-function $L_\fs\in \I_\fs$ 
    and a special class
    $\mathcal{Z}_{p,\fs}\in H^1(\Qp, \I_\fs(\Psi_\fs^{-1}))$.
    \item If $\fl$ is a split prime 
    and $\ell\fs,\fs\in \mathcal{R}$ for $\ell=\fl\fl^c$, then
    \[
        \phi^{\ell\fs}_\fs(\mathcal{Z}_{\ell\fs})=
        (1-\epsilon\Psi_{\fs}(\varpi_{\bar{\fl}}))
        (1-\Psi_\fs(\varpi_{\bar{\fl}}))\cdot 
        \mathcal{Z}_{\fs}
    \]
    where $\phi^{\ell\fs}_\fs\colon \I_{\ell\fs}\to \I_\fs$
    is associated to the quotient $\fG_{\ell\fs}^a\to \fG_\fs^a$
    and $\epsilon$ is the $p$-th cyclotomic character.
\end{enumerate}
\end{thm}

\subsection*{Congruences and Euler systems}

Our construction of the classes $\{\mathcal{Z}_\fs\}_{\fs\in\mathcal{R}}$
follows from Urban's approach in \cite{urban},
which use the Eisenstein congruences 
to reproduce the Euler system of cyclotomic units.
The idea of using congruences relations to construct cohomology classes
dates back to the proof of the converse of
Herbrand-Ribet theorem in \cite{Ribet1976}
and is also the main ingredient of \cite{HT93}.

To be specific,
given $\psi=\mu^{c-1}$ we may 
consider the Hida family of Hilbert modular forms
with CM in $\K$ associated to $\mu$,
which induces a Hecke eigensystem
$\lambda\coloneqq \TT^{\ord}\to \I$
of the big ordinary Hecke algebra $\TT^{\ord}$
to a certain finite extension $\I$ over $\Lambda_W$.
Let $T\colon \Gal_\K\to \I$ be the associated 
Galois pseudo-representatiaon.
A generalization of Ribet's lemma in \cite{Ribet1976}
then produces cohomology classes
with coefficients in quotients of $\I$ instead of $\I$.
The insight of Urban's method is to consider instead
the big Hecke algebra of a larger space of $p$-adic modular forms
on which the reducible ideal of 
the Galois pseudo-representation is better controlled.

Before further explanation,
we note that in the current article 
we have chosen to work with 
modular forms on a definite unitary group $G=\UU(2)$
rather than the classical Hilbert modular forms.
Using the author's previous work \cite{lee}
we can construct $\I_\fs$-adic Hida families $\euF_\fs$
of CM forms on $G$ associated to $\mu$ for $\fs\in\mathcal{R}$.
Then $L_\fs\in \I_\fs$ is essentially the inner product 
of $\euF_\fs$ with itself.
One can think of $\euF_\fs$ as interpolating the Jacquet-Langlands
transfers of the Hilbert modular forms above.
We remark that our construction in \cite{lee}
relies on the existence of an auxiliary self-dual 
Hecke character with non-vanishing central $L$-value.
To the knowledge of the author such existence is generally hard
and the available results mostly uses the canonical characters
from \cite{Rohrlich}, which is the main reason
we impose the condition \ref{cond:K2in} and \ref{cond:K3in}.
On the other hand, the condition \ref{cond:psi2in} 
can be relieved if the local computations in \cite{lee}
at the non-split places can be improved,
which the author would hope to come back to in the future.

If $U^p\subset G(\A_{\F,f}^{(p)})$ is an open compact subgroup,
for any open compact subgroup $U_p\subset G(\F\otimes_\Q\Qp)$
we let $S(U^pU_p,E/\eo)$ denote the space
of trivial weight $p$-adic modular forms 
of level $U^pU_p$ with coefficients in $E/\eo$.
We can then consider the inverse limit
\[
    S(U^p,E/\eo)\coloneqq \varinjlim_{U_p}
    S(U^pU_p, E/\eo)
\]
and the Pontryagin dual of which,
where $U_p$ ranges through all the compact open subgroup
in $G_p\coloneqq G(\F\otimes_\Q\Qp)$.
Then $S(U^p,E/\eo)$ is a smooth $G_p$-representation
in the sense of \cite{emeI} and we can apply
Emerton's functor of ordinary parts $\Ord_P$ on which.
Since $G_p\cong \prod_{w\in\Sigma_p}\GL_2(\K_w)$,
if we pick 
\[
    P=\GL_2(\K_{w_0})\times\prod_{w'\neq w_0}B(\K_{w'}),\quad
    B=\prod_{w\in \Sigma_p}B(\K_{w'}),
\]
where $B\subset \GL_2$ is the subgroup of upper-triangular matrices,
then $\Ord_P$ and $\Ord_B$
gives the space of modular forms 
that are ordinary at all $w'\in \Sigma_p\setminus\{w_0\}$ 
and at all $w\in \Sigma_p$ respectively.
We write $M(U^p)$ and $M^{\ord}(U^p)$
for the respective Pontryagin duals,
which we refer to as the spaces of (ordinary)-completed homology
and let $\TT(U^p,\eo)$ and $\TT^{\ord}(U^p,\eo)$
be the Hecke algebras acting on which. 
It can be shown that the families $\euF_\fs$
we constructed give rise to eigensystems
$\lambda_\fs\colon \TT(U^p_\fs,\eo)\to \I_\fs$
and elements $F_\fs\in M^{\ord}(U^p_\fs)\otimes \I_\fs$,
where $\{U^p_\fs\}_{\fs\in\mathcal{R}}$
is a collection of open compact subgroups
that varies systematically when $\fs\in\mathcal{R}$.

\subsection*{Completed homology and $p$-adic local Langlands}

The spaces $M(U^p_\fs)$ is naturally 
a representation of $\GL_2(\K_{w_0})=\GL_2(\Qp)$
since $w_0$ has degree one.
Let $\fm\subset \TT(U^p_\fs)$ be the maximal ideal
associated to the eigensystems $\lambda_\fs$.
If $T(U^p_\fs)\colon \Gal_\K\to \TT(U^p_\fs)$
is the big Galois pseudo-representation,
then $T(U^p_\fs)\vert_{\Gp}\bmod \fm$
is independent of $\fs$ 
and the universal deformation ring $R$ of which
is isomorphic to a power series ring
$\eo\llbracket x_1,x_2,x_3\rrbracket$ of three variables over $\eo$
under \ref{cond:psi1in}.
Furthermore we may take one of the variable, say $\xx\coloneqq x_1$,
to be the generator of the reducibility ideal.

Let $\Lambda_T$ be a the Iwasawa-algebra
associated to the maximal torus $T$ of $G$.
By the control theorem in \cite{ger}
the spaces $M^{\ord}(U^p_\fs)$ are finite free over $\Lambda_T$.
And since $B$-ordinary forms are automatically $P$-ordinary,
we have the natural homomorphisms 
$\TT(U^p_\fs,\eo)\to \TT^{\ord}(U^p_\fs,\eo)$ and
$M(U^p_\fs)\to M^{\ord}(U^p_\fs)$.
Let $R\to \TT(U^p_\fs,\eo)$ be the homomorphism
induced by $T(U^p_\fs)$, then 
\[
    \xx\in \ker\big(R\to \TT(U^p_\fs,\eo)\to \TT^{\ord}(U^p_\fs,\eo)\big)
\]
since Galois representations associated to ordinary modular forms
are reducible at $w_0$.
We then prove the theorem below following \cite[\S 5]{urban},
which combines Pask\={u}nas' theory of blocks from \cite{pask} and
Pan's result on the local-global compatibility from \cite{pan}.

\begin{thm}
    Let $\tilde{P}_{1,\fm}$ be the projective envelope
    of a member in the block associated to $\fm$.
    \begin{enumerate}
    \item There are only two maximal ideals $\fm_1$ and $\fm_2$
    in the localization $\TT^{\ord}(U^p_\fs,\eo)_\fm$ 
    as a $\TT(U^p_\fs,\eo)$-algebra.
    \item Both $\Hom(\tilde{P}_{1,\fm},M(U^p_\fs)_\fm')$
    and $\TT(U^p_\fs,\eo)_\fm$ 
    are finite free over 
    $\Lambda_T\llbracket\xx\rrbracket$.
    \item Define
    $\Hom(\tilde{P}_{1,\fm},M(U^p)_{\fm}')^\red\coloneqq
	\Hom(\tilde{P}_{1,\fm},M(U^p)_{\fm}')/\xx
	\Hom(\tilde{P}_{1,\fm},M(U^p)_{\fm}')$,
    then there exists an exact sequence 
    \[
    0\to M^{\ord}(U^p)_{\fm_2}\to
    \Hom(\tilde{P}_{1,\fm},M(U^p)_{\fm}')^\red
    \to M^{\ord}(U^p)_{\fm_1}'\to0 
    \]
    \end{enumerate}
\end{thm}
We refer the readers to
Corollary \ref{cor:fil_by_ord},
Proposition \ref{prop:Hecke_finite}, and
Corollary \ref{cor:Hecke_ff} for the precise meaning
of the notations. 
Here we only note the following two important 
consequences of the theorem.
\begin{itemize}
\item Write $\TT_\fs=\TT(U^p_\fs,\eo)$.
The pseudo-representation $T_\fs=T(U^p_\fs)$
can be represented by a generalized matrix algebra
\[
    \rho_\fs(\gamma)=
    \begin{pmatrix}
        A(\gamma) & B(\gamma)\\
        C(\gamma) & D(\gamma)
    \end{pmatrix}\in \GL_2(\TT_\fs[1/\xx]),\quad
    T_\fs(\gamma)=A(\gamma)+D(\gamma).
\]
\item There exists a fundamental exact sequence 
which is essentially the extension of the congruence module
of $\euF_\fs$ by $\I_\fs$.
\end{itemize}

Let $B_\fs\subset \TT_\fs[1/\xx]$
be the submodule generated by all $B(\gamma)$.
The generalization of Ribet's lemma asserts that
each homomorphism in $\Hom_{\TT_\fs}(B_\fs,\I_\fs)$
would product a cohomology class 
in $H^1(\K,\I_\fs(\Psi_\fs^{-1}))$.
By our construction of the fundamental exact sequence
there exists natural maps from $B_\fs$
to the middle term of which that are compatible
among $\fs\in\mathcal{R}$.
The desired classes $\mathcal{Z}_\fs$
are then obtained by ``mutiplying'' the natural maps
by $L_\fs$, which annihilates the congruence modules
and gives a map valued in $\I_\fs$.
And the extra Euler factors 
in $\phi^{\ell\fs}_\fs(L_{\ell\fs})$
compared to $L_\fs$ contribute to the 
norm relation between classes of different levels.

\subsection*{Iwasawa main conjectures}

We make a final remark on the deduction  of the main conjecture.
Since the classes we constructed does not fit the general
formulation in \cite{Rubin},
we need to modify the argument in which 
based on the theory of Jetchev-Nekovar-Skinner
and combine it with the specialization principle from \cite{Och05}
and the algebraic functional equation from \cite{Hsieh2010}.
The classes we constructed assuming \ref{cond:psi1in}
is sufficient to show that the Selmer group
$\Sel(\psi,\Sigma_p)$ is torsion.
But in order to show the main conjecture 
we have to introduce \ref{cond:psi2in}
so that $\loc_{w'}(\mathcal{Z}_\fs)=0$
at all $w'\in \Sigma\setminus\{w_0\}$.
However, given that 
$\loc_{w'}(\mathcal{Z}_\fs)$
is annihilated by $\Psi_\fs^{-1}(\sigma)-1$
for any $\sigma\in D_{w'}$, 
which generated an ideal in $\I_\fs$ of height at least two,
it seems reasonable to expect that the current
obstruction only contribute to a pseudo-null module.

On the other hand, it is unlikely that 
our method can be improved to gives
the inclusion in $\eo\llbracket W\rrbracket$
since the constructions of $L_\fs$ in \cite{lee}
would always carry some fudge factors from local volumes
of $G(\A_{\F,f})$ that are hard to control. 

\subsection*{Contents of the paper}
After introducing the notations, we recall the theory
of $p$-adic modular forms and their associated
Galois representation following \cite{ger}.
We incorporate the language of 
smooth representations in the discussion of which
to prepare for the application to $p$-adic local Langlands.

In \S 3 we first review Pask\={u}nas'theory
of blocks on the category of locally-finite representaitons
of $\GL_2(\Q_p)$ following \cite{pask} and \cite{urban}.
We then apply the theory
to the completed homology and 
obtain the local-global compatibility result following \cite{pan}.
We then deduce the fundamental exact sequence
that allows the construction of our Euler systems.

In \S 4,
after reviewing the construction from \cite{lee}
of the Hida families $\euF_\fs$
and the $p$-adic 
we apply the fundamental exact sequence obtained
in previous section and show that 
the produced cohomology classes
satisfied the desired properties
and form an Euler system.
We leave the deduction of the main conjecture
from the Euler systems to \S 5.

\subsection*{Acknowledgement}

The results of this work is the extension of parts of the the author's Ph.D. thesis in Columbia University.
The author would like to thank his advisor, Eric Urban, for introducing him to the subject and his constant encouragement.
The author would also like to thank Professor Ming-Lun Hsieh for his 
interest in the work and very helpful discussions during the preparation 
of the article.
The results here are first announced in the 2025
RIMS conference on
"Arithmetic aspects of automorphic forms and automorphic representations",
and the author would like to thank the organizer and the Institute
for their hospitality and the financial support.

\section{Notations}

Throughout the article, $\F$ is a totally real field 
with $[\F,\Q]=d$
and $\K$ is a totally imaginary quadratic extension over $\F$.
Let $\arch=\Hom(\F, \C)$ 
denote the set of archimedean places of $\F$,
and $\finite$ the set of finite places of $\F$.
We fix an odd prime $p$ throughout the article
that satisfies the following conditions.
\begin{align}\label{cond:ord}\tag{ord}
\text{Every finite place of $\F$ above $p$ is split in $\K$}.\\
\label{cond:deg}\tag{deg}
\text{There exists a place $w_0$ of $\K$ above $p$ of degree one}.
\end{align}
We fix an embedding $\iota_\infty:\bar{\Q}\to \C$
and an isomorphism $\iota:\C\cong \C_p$,
and write $\iota_p=\iota\circ\iota_\infty:\bar{\Q}\to \C_p$.

Given a place $v$ of $\F$, archimedean or finite,
let $w\mid v$ denote a place $w$ of $\K$ above $v$.
Then $\K_w$ and $\F_v$ are respectively
the completions of the fields $\K$ and $\F$ at $w$ and $v$.
When $v\in \finite$ we denote by $\oo_w$ and $\oo_v$ 
the rings of integers of $\K_w$ and $\F_v$, and
$\varpi_w$ and $\varpi_v$
are uniformizers in $\oo_w$ and $\oo_v$ respectively.
We normalize the norm $|\cdot|_v$ on $\F_v$,
so that it is the usual absolute value when $v\in \arch$
and $q_v=|\varpi_v|_v^{-1}$
is the cardinality of the residue field $\oo_v/(\varpi_v)$
when $v\in \finite$.
For $w\mid v$, define $|a|_w=|\Nr_{\K_w/\F_v}(a)|_v$.
Then $|\cdot|_w$ is the square of the usual absolute value
when $v\in\arch$ and
$q_w=|\varpi_w|_w^{-1}$
is the cardinality of the residue field $\oo_w/(\varpi_w)$
when $v\in\finite$.

Denote by $\A=\A_{\F}$ the ring of adeles over $\F$,
by $\A_{\infty}$ and $\A_{f}$ respectively
the archimedean and the finite components of $\A$.
Let $\qch_{\K/\F}$ denote 
the quadratic character on $\A_\F^\times/\F^\times$
associated to $\K/\F$ by the global class field theory,
$\qch_v$ denote the component on $\F_v^\times$ 
when $v\in \finite$.
Let $c\in \Gal(\K/\F)$ be the unique complex conjugation.
When it should be clear from the context
we write $cz$, $z^c$ and $\bar{z}$ interchangeably 
for the action of $c$ on $z\in \K\otimes_\F R$, where $R$ is an $\F$-algebra,
induced by the action on the first factor.
We then define 
$\A_\K^1=\{z\in \A^\times_\K=\K\otimes_\F \A_\F \mid z\bar{z}=\Nr_{\K/\F}z=1\}$ and
$\K_v^1=\{z\in \K_v\coloneqq \K\otimes_\F\F_v\mid z\bar{z}=1\}$.
If $\eta$ is a character of $\A_\K^1/\K^1$, 
we denote
by $\tilde{\eta}(\alpha)\coloneqq \eta(\alpha/\alpha^c)$
the Hecke character which is the base change of $\eta$ 
to $\A_\K^\times/\K^\times$.

\subsection{CM types}

Denote respectively by $S_p$ and $S_p^\K$ the set of places above $p$
of $\F$ and $\K$.
Identify $I_\K=\Hom(\K,\bar{\Q})$ with
$\Hom(\K,\C)$ and $\Hom(\K,\C_p)$ by compositions with $\iota_\infty$ and $\iota_p$.
Given $\sigma\in I_\K$,
let $w_\sigma\in S_p^\K$ be the place induced by
$\sigma_p\coloneqq \iota_p\circ \sigma\in\Hom(\K,\C_p)$.
For $w\in S_p^\K$, define
\[
    I_w=\{\sigma\in I_\K\mid w=w_\sigma \}=\Hom(\K_w,\C_p)
\]
and decompose $I_\K=\sqcup_{w\mid p}I_w$.
For a subset $\Sigma\subset I_\K$
define $\Sigma_p=\{w_\sigma\mid \sigma\in \Sigma\}$.
We write
$\Sigma^c=\{\sigma c\mid \sigma\in \Sigma\}$ and 
$\Sigma_p^c=\{cw\mid w\in \Sigma_p\}$.
We fix throughout the article a $p$-ordinary CM type,
which is a subset $\Sigma\subset I_\K$ such that
\[
    \Sigma\sqcup \Sigma^c=I_\K,\quad
    \Sigma_p\sqcup \Sigma_p^c=S_p^\K.
\]
The $p$-ordinary CM type $\Sigma$
always exists by the assumption \eqref{cond:ord},
and is identified with $\arch=\Hom(\F,\C)$ by restrictions.
When $v\in S_p$ decomposes into $v=w\bw$,
we understand always that $w\in \Sigma_p$.

\subsection{Artin reciprocity}

Let $\K^{ab}$ be the maximal abelian extension.
We normalize the Artin reciprocity map
$\Art\colon \A_{\K,f}^\times \to Gal(\K^{ab}/\K)$
so that $\Art_w(\varpi_w)=\Fr_w$
is the geometric Frobenius for each finite place $w$.
More generally,
let $U\subset \A_{\K,f}^\times$ be 
a closed subgroup that is stabilized 
by the complex conjugation
and let $L/\K$ be the extension such that 
$Gal(L/\K)\cong \K^\times\backslash \A_{\K,f}^\times/U$
under the reciprocity map.
Then $Gal(\K/\F)$ acts on $Gal(L/\K)$
and we define
\[
    Gal(L/\K)^{\pm}=\{\gamma\in \Gal(L/\K)\mid 
    \gamma^c=\gamma^{\pm1}\}.
\]
We call $Gal(L/\K)^a$, the quotient of 
$Gal(L/\K)$ by $Gal(L/\K)^+$,
the anticyclotomic quotient
of $Gal(L/\K)$.
And we say a character $\psi$ of $Gal(L/\K)$
is anticyclotomic if it factors through $Gal(L/\K)^a$.

Write $Gal(L/\K)=\cG$ for short.
We note that $\cG^a$ can be
identified with a subgroup of $\cG^-$ via the homomorphism
\[
    1-c\colon \cG\to \cG\quad \gamma\to \gamma/\gamma^c
\]
since $\ker(1-c)=\cG^+$ by definition.
Under the identification
$\cG^-/\cG^a$ is an abelian group of type $(2,\cdots,2)$.

On the adelic side we also define
\[
1-c\colon \K^\times\backslash\A_{\K,f}^\times\to
\K^\times\backslash\A_{\K,f}^\times\quad
z\mapsto z/\bar{z}
\]
which is compatible with the map on the Galois side
under the reciprocity map.
Since $1-c$ maps $\K^\times\backslash \A_{\K,f}^\times$
surjectively onto  $\K^1\backslash \A_{\K,f}^1$,
we can then define an anticyclotomic
reciprocity map $\Art^a\colon \K^1\backslash\A_{\K,f}^1\to \cG^a$
by the commutative diagram
\begin{equation}\label{eq:anticyc_rec}
\begin{tikzcd}
    K^\times\backslash\A_{\K,f}^\times
    \arrow[r,twoheadrightarrow]\arrow[d,"\Art"]&
    K^1\backslash\A_{\K,f}^1 \arrow[r]\arrow[d,"\Art^a",dashed]&
    K^\times\backslash\A_{\K,f}^\times\arrow[d,"\Art"]\\
    \cG \arrow[r,twoheadrightarrow]&
    \cG^a\arrow[r]&
    \cG
\end{tikzcd}
\end{equation}
In particular, suppose $\psi$ is a finite order 
anticyclotomic character of $\cG$.
Identify $\cG^a$ with a subgroup of $\cG^-$,
we may extend $\psi$ to a finite order character $\mu$ of $\cG$.
Then, when viewed as Hecke characters via the reciprocity map,
the characters satisfy the relation $\psi=\mu^{1-c}$.

\subsection{Characters}

Let $\kappa=\sum_{\sigma\in \Sigma} 
a_\sigma\sigma+b_\sigma\sigma c\in \Z[I_\K]$. We define 
\[
    \kappa(\alpha_\infty)=\alpha_\infty^\kappa=
    \prod_{\sigma\in \Sigma} 
    (\alpha_\sigma)^{a_\sigma}(\bar{\alpha}_\sigma)^{b_\sigma}\in \C^\times,\quad
    \kappa(\alpha_p)=\alpha_p^\kappa=
    \prod_{w\in \Sigma_p}
    \prod_{\sigma\in I_w}
    \sigma_p(\alpha_w)^{a_\sigma}\sigma_p(\alpha_{\bw})^{b_\sigma}
    \in \C_p^\times.
\]
for $\alpha_\infty=(\alpha_\sigma)_{\sigma\in\Sigma}
\in \A_{\K,\infty}^\times$
and $\alpha_p=(\alpha_w,\alpha_{\bw})_{w\in\Sigma_p}\in 
\prod_{w\mid p}\K_w^\times$.
When the meaning is clear from the context, we also 
let $\Sigma$ denote the formal sum $\sum_{\sigma\in\Sigma}\sigma$
and similarly for $\Sigma^c$.
If $\chi\colon \A_\K^\times/\K^\times\to \C^\times$ 
is a Hecke character of $\K$,
we let $\chi_w$ be the component of $\chi$ at a place $w$ of $\K$.
We also write  $\chi_\infty=(\chi_w)_{w\in\Sigma}$,
$\chi_p=(\chi_w)_{w\in S_p^\K}$, and
$\chi_v=(\chi_w)_{w\mid v}$ when $v$ is a place of $\F$.
We say $\chi$ is an algebraic Hecke character of
the infinity type $\kappa$ if
$\chi_\infty(\alpha)=\alpha^\kappa$.
When this is the case we define the $p$-adic avatar 
$\hat{\chi}\colon \A_\K^\times\to \bar{\Z}_p^\times$  of $\chi$ by
\[
    \hat{\chi}(\alpha)=
    \iota(\chi(\alpha)\alpha_\infty^{-\kappa})\alpha_p^{\kappa}.
\]

\subsection{Matrices}
When $R$ is an $\F$-algebra and 
$m=(m_{ij})\in \text{M}_{r,s}(\K\otimes_\F R)$,
we denote by 
$m^\intercal=(m_{ji}), 
m^c=(m^c_{ij})$, and
$m^*=(m^c_{ji})$
respectively the transpose, conjugate, and conjugate-transpose of $m$.

When $r=s$ and $g\in \GL_r(\K\otimes_\F R)$ is invertible, we write
$g^{-\intercal}=(g^{-1})^\intercal$ and $g^{-*}=(g^{-1})^*$.
We write $\mtr(m)$ for the trace of a square matrix $m$,
and reserve $\Tr$ for the traces between fields extensions.

When $v=w\bw$ is a place that is split in $\K$,
identify $\K_w=\F_v=\K_{\bw}$ and 
write $\K_v=\F_v^2$, 
where the first component corresponds to $\K_w$.
Then $m=(m_w,m_{\bw})\in M_n(\K\otimes_\F\F_v)=M_n(\F_v)\times M_n(\F_v)$ 
denotes an element in $m\in M_n(\K\otimes_\F\F_v)$ and its components.

\subsection{Representations of $p$-adic groups}

Let $\oo$ be the ring of integers of a finite extension $E$
over  $\Qp$ with a uniformizer $\varpi$.
When $G$ is a $p$-adic analytic group,
we let $\Mod_G(\oo)$ be the category
of all $\oo[G]$-modules.
Then $\Mod^{\sm}_{G}(\oo)$ is the full subcategory 
of $\oo[G]$-modules $V$ such that 
\begin{itemize}
    \item Each $v\in V$ is fixed by some open compact subgroup of $G$.
    \item Each $v\in V$ is annihilated by $\varpi^h$ for some power 
    $h\geq 0$.
\end{itemize}
And $\Mod^{\adm}_{G}(\oo)$ is the full subcategory 
in $\Mod^{\sm}_{G}(\oo)$ of $\oo[G]$-modules $V$ such that 
$V^H[\varpi^h]$ is finite over $\oo$
for any compact open subgroup $H$ of $G$ and any power $\varpi^h$.

We also define $\Mod^{\ladm}_{G}(\oo)$ to be the full subcategory 
in $\Mod^{\sm}_{G}(\oo)$ of $\oo[G]$-modules $V$ 
such that for each $v\in V$ 
the submodule $\oo[G]v$ belongs to $\Mod^{\adm}_G(\oo)$.
And $\Mod^{\lfin}_G(\oo)$ is the full subcategory 
in $\Mod^{\sm}_{G}(\oo)$ of $\oo[G]$-modules $V$ 
such that for each $v\in V$ 
the submodule $\oo[G]v$ has finite length as a $\oo[G]$-module.
We refer the readers to \cite[\S 2]{emeI} and \cite[\S 2]{pask}
for more properties regarding the above categories.
At last, if $\zeta$ is a character of the center of $G$,
then $\Mod^{\bullet}_{G,\zeta}(\oo)$
denote the subcategory in $\Mod^{\bullet}_{G}(\oo)$ of 
$\oo[G]$-modules with central character $\zeta$.

\section{Modular forms on definite unitary groups}

Let $G$ be the definite unitary group over $\F$,
such that for any $\F$-algebra $R$
\begin{equation}\label{def:def_unitary}
    G(R)=\{g\in \GL_{n}(\K\otimes_\F R) \mid gg^*=\id_n\}.
\end{equation}
In this section we introduce
the space of algebraic modular forms on $G$
and their associated Galois representations
following \cite{ger}.
Moreover, for a general parabolic subgroup $P$,
we apply Emerton's functor of $P$-ordinary parts from \cite{emeI}
systematically to define 
the subspace of $P$-ordinary modular forms
and the big Galois pseudo-representation of $\Gal_\F$
which takes values in the big $P$-ordinary Hecke algebra.
We also show the density of crystalline points 
in such big Hecke algebra
after incorporating the techniques developed in \cite{pan}.
The density result will be crucial for checking
the local-global compatibility in the next section.

\subsection{Algebraic modular forms}

Let $B_n=T_nN_n\subset \GL_n$ be the subgroup of
upper triangular matrices and its Levi decomposition,
where $T_n$ is the diagonal torus.
An element $w$ in the Weyl group $W_n$ of $\GL_n$
acts on $k\in X^*(T_n)$, the set of algebraic characters,
by $(wk)(t)=k(w^{-1}tw)$.
Following \cite[Def 2.3]{ger},
we identify $X^*(T_n)$ with $\Z^n$
and $k=(k_1,\cdots,k_n)\in X^*(T_n)$
is said to be dominant if $k_1\geq \cdots\geq k_n$.
Let $w_0\in W_n$ denote the longest element
and $k\in X^*(T_n)$ be dominant.
We define $\xi_k\coloneqq \Ind_{B_n}^{\GL_n}(w_0k)$,
which is an algebraic representation 
of highest weight $k$.

For each $v\in\finite$ that splits in $\K$
we fix a prime $w$ above $v$
and write $v=w\bw$ in $\K$.
Then we can write $g_v=(g_w,g_{\bw})$
for $g_v\in \GL_n(\F_v\otimes_\F\K)
\cong \GL_n(\K_w)\times\GL_n(\K_{\bw})$.
In particular, the map
\begin{equation}
\iota_w\colon G(\F_v)\to \GL_n(\K_w)\quad
\iota_w(g_v)=g_w
\end{equation}
is an isomorphism and satisfies
$\iota_w(g_v)=\iota_{\bw}(g_v)^{-\intercal}$.
We will identify $G_w\coloneqq\GL_n(\K_w)$
with $G(\F_v)$ via $\iota_w$
and similarly identify 
$B_w=B_n(\K_w), N_w=N_n(\K_w)$ and
$T_w=T_n(\K_w)$
with subgroups of $G(\F_v)$.
Moreover, we define the open 
compact subgroups
$K_w=\GL_n(\oo_w)$,
$\Iw(w)=\{k\in K_w\mid k\bmod \varpi_w\in B_n(\oo_w)\}$, and
$\Iw_1(w)=\{k\in \Iw(w)\mid k\bmod \varpi_w\in N_n(\oo_w)\}$.
In particular, we pick $w\in \Sigma_p$ when $v\in S_p$ and define
\[
	G_p\coloneqq\prod_{w\in \Sigma_p}G_w,\quad
	K_p\coloneqq\prod_{w\in \Sigma_p}K_w.
\]

Throughout the section
we fix a finite extension $E$ over $\Qp$
that contains $\iota_p(\sigma(\K))$
for all $\sigma\in I_\K$ and
let $\oo=\oo_E$ be the ring of integers in $E$.
When $\wt{k}=(k_\sigma)\in (\Z^n)^{\Sigma}$
is dominant in the sense that
$k_\sigma=(k_{\sigma,1},\cdots,k_{\sigma,n})$
is dominant for each $\sigma\in \Sigma$,
let $\xi_{\wt{k}}$ be
the algebraic $K_p$-representation over $\oo$ given by
\begin{equation}\label{def:algrep}
	\xi_{\wt{k}}=\bigotimes_{\sigma\in \Sigma}
	\Ind_{B_n}^{\GL_n}(w_0k_{\sigma}),\quad
	\xi_{\wt{k}}(g)=
	\otimes_{w\in \Sigma_p}
	\otimes_{\sigma\in I_w}\xi_{k_\sigma}(g_w)\,
	\text{ for } g=(g_w)\in K_p.
\end{equation}
Note that over $E$, the representation $\xi_{\wt{k}}$ 
is an algebraic $G_p$-representation.

\begin{defn}\label{def:algform}
When $A$ is an $\oo$-module and  
$\wt{k}\in (\Z^n)^{\Sigma}$ is dominant,
we let $g=(g^p,g_p)\in G(\A_f^p)\times K_p$ acts on 
the space of functions
$f\colon G(\F)\backslash G(\A_f)\to A\otimes_{\oo}\xi_{\wt{k}}(\oo)$
by $(g\cdot f)(g_0)=\xi_{\wt{k}}(g_p)\cdot f(g_0g)$.
Note that the action can be extended to the whole $G(\A_f)$
when either $A$ is an $E$-module or $\xi_{\wt{k}}$
is the trivial representation.
We say such a function $f$ is an algebraic modular form of
weight $\wt{k}$ with coefficients in $A$
if $f$ is invariant under the above action 
by some open compact subgroup
$U\subset G(\A_f^p)\times K_p$.
Let $S_{\wt{k}}(A)$
denote the space of all algebraic modular forms
of weight $\wt{k}$ and coefficients in $A$.
And for $U$ as above we define
\begin{equation}
S_{\wt{k}}(U,A)=
S_{\wt{k}}(A)^U=
\left\{ f: G(\F)\backslash G(\A_f)/U^p 
\rightarrow A\otimes_{\oo}\xi_{\wt{k}}(\oo)
\mid f(gu)=\xi_{\wt{k}}(u_p)^{-1}\cdot f(g), u\in U\right\} 
\end{equation}
We will write $S(A)=S_{\wt{k}}(A)$ and
$S(A,U)=S_{\wt{k}}(A,U)$
when $\xi_{\wt{k}}$ is the trivial representation.
\end{defn}

Since $G(\F)\backslash G(\A_f)/U$ is a finite set
for any open compact subgroup $U\subset G(\A_f)$,
an algebraic modular form $f\in S_{\wt{k}}(U,A)$ 
is determined by its values on a finite set of points.
Moreover, throughout the section
we fix a finite subset $S\subset \finite$
of prime-to-$p$ places $v=w\bw$ that splits in $\K$
and an open compact subgroup $U^p=\prod_{v\nmid p}U_v\subset G(\A_f^p)$ 
that satisfies the following conditions.
\begin{align}
    \label{cond:s-ram}\tag{$S$-$\Iw$}
    &\Iw_1(w)\subset U_v \subset \Iw(w) \text{ for all } 
    v\in S \text{ and the fixed }w\mid v\\
    \label{cond:small}\tag{\text{small}}
	&G(\F)\cap t_i(\tilde{U})t_i^{-1}=\{1\} \text{ for all } 
    i\in I \text{ in }
    G(\A_f)=\bigsqcup_{i\in I} G(\F)t_i \tilde{U},
    \text{ where } \tilde{U}=U^pK_p\prod_{w\mid v\in S}\Iw(w).
\end{align}
Then by \cite[Lem 2.6]{ger} the space $S_{\wt{k}}(U^pU_p,A)$
is finite free over 
$A[\prod_{v\in S}(\Iw(w)/U_v)]$
whenever $A$ is an $\oo$-algebra
and $U^p\subset K_p$ is an open compact subgroup.

\subsection{Hecke operators}

Given integers $b$ and $c$
such that $c\geq b\geq 0$ and $c>0$,
we define the open compact subgroups
$\Iw(p^{b,c})=\prod_{w\in \Sigma_p}\Iw(w^{b,c})$ of $G_p$, where 
$\Iw(w^{b,c})$, for $w\in \Sigma_p$, is defined as
\begin{equation}\label{def:Iwahori}
	\Iw(w^{b,c})=\{
	k\in K_w\mid 
    k \bmod \varpi_w^c \in B_n(\oo/\varpi_w^c)
	\text{ and }
	k \bmod \varpi_w^b \in N_n(\oo/\varpi_w^b)
	\}.
\end{equation}

Let $T\subset\finite$ be a finite set containing $S_p\cup S$
such that $U_v=K_w$ if $v\notin T$ and $v$ splits in $\K$,
where $w$ is the fixed prime above $v$.
We recall from \cite{ger} 
the definitions of the following Hecke operators,
which are defined as double cosets operators
acting on $S_{\wt{k}}(U^p\Iw(p^{b,c}),A)$.
\begin{itemize}

\item 
If $v\notin T$ and $v=w\bw$ splits in $\K$,
for $1\leq j\leq n$ we define 
\begin{equation}\label{def:hecke_away_p}
	T_w^{(j)}=
	\left[
	K_w
	\begin{pmatrix}
		\varpi_w\id_{j}&\\&\id_{n-j}
	\end{pmatrix}
	K_w
	\right],
\end{equation}
which satisfies the relation
$T_{\bw}^{(j)}=(T_{w}^{{n}})^{-1}T_w^{(n-j)}$.

\item
If $w\in \Sigma_p$, let  
$\alpha_w^{(j)}=
\smat{ \varpi_w\id_{j}&\\&\id_{n-j} } $
for $1\leq j\leq n$
and $u\in T_n(\oo_w)$ we define
\begin{equation}\label{def:hecke_at_p}
	U_{\wt{k},w}^{(j)}=
	(w_0\wt{k})^{-1}(\alpha_{w}^{(j)})\cdot
	[\Iw(p^{b,c})\alpha_w^{(j)}\Iw(p^{b,c})],\,
	\langle u\rangle=
	[\Iw(p^{b,c})u\Iw(p^{b,c})],
	\text{ and }
	\langle u\rangle_{\wt{k}}= (w_0\wt{k})^{-1}(u)\cdot \langle u\rangle.
\end{equation}
Here $w_0\wt{k}$ is viewed as an algebraic character of $T_n(\F_w)$ 
by the same recipe in \eqref{def:algrep}.

\item 
If $v\in S$ and $w$ is the fixed place above $v$,
for $\alpha_w^{(j)}$ defined as above
and $u\in T_n(\oo_w)$ we define 
\begin{equation}\label{def:hecke_at_s}
	U_{w}^{(j)}=
	[U_v\alpha_w^{(j)}U_v]
	\text{ and }
	\langle u\rangle= 
	[U_vu U_v]
\end{equation}

\end{itemize}

\begin{rem}
For $u\in T_n(\oo_w)$ in either 
\eqref{def:hecke_at_p} or \eqref{def:hecke_at_s},
the operator $\langle u\rangle$
coincides with the usual action of $u$ in
Definition \ref{def:algform}.
Therefore the action of which
factors through the quotient by some 
open compact subgroup in $T_n(\oo_w)$.
On the other hand, 
we introduce $\langle u\rangle_{\wt{k}}$
in \eqref{def:hecke_at_p}
so that some Hecke-equivariant properties
are easier to state.
Also note that the two versions of diamond operators
coincide when $\xi_{\wt{k}}$ is the trivial representation,
so there is no confusion when we write
$U_{\wt{k},w}^{(j)}=U_{w}^{(j)}$ and 
$\langle u\rangle_{\wt{k}}=\langle u\rangle$ in such case.
\end{rem}

\subsection{The space of $P$-ordinary forms}

\subsubsection{Emerton's functor}

We temporarily let $G$ be an arbitrary $p$-adic reductive group
and $P=QU$ be a parabolic subgroup of $G$ and its Levi decomposition.
To recall the functor 
$\Ord_P\colon \Mod_G^{\sm}(\oo)\to \Mod_Q^{\sm}(\oo)$
defined in \cite{emeI},
we fix an open compact subgroup $P_0\subset P$
and put $Q_0=P_0\cap Q, U_0=P_0\cap U$,
then define $Z_Q^+=Z_Q\cap Q^+$,
where $Z_Q$ is the center of $Q$ and
\[
	Q^+=\{m\in Q\mid mU_0m^{-1}\subset U_0\}.
\]
If  $V$ is a $P$-representation over $\oo$
and  $m\in Q^+$,
we follow \cite[Def 3.1.3]{emeI} and define
\begin{equation}\label{def:hUm}
	 h_{U}(m)\colon V^{U_0}\to V^{U_0}\quad
	 h_{U}(m)(v)=\sum_{u\in U_0/m U_0 m^{-1}}um\cdot v
\end{equation}
Then the functor
$\Ord_P\colon \Mod_G^{\sm}(\oo)\to \Mod_Q^{\sm}(\oo)$
is defined in \cite[Def 3.1.3]{emeI} by
\begin{equation}\label{def:OrdP}
	\Ord_P(V)=\Hom_{\oo[Z_Q^+]}
    (\oo[Z_Q], V^{U_0})_{Z_Q-\textnormal{fin}}.
\end{equation}
Here $\oo[Z_Q]$ is a $\oo[Z_Q^+]$-modules by translation
and $V^{U_0}$ is a $\oo[Z_Q^+]$-modules by $h_U$.
And the action of $Q=Z_Q\cdot Q^+$ on $\Ord_P(V)$ is defined by 
having $Z_Q$ act by translation on $\oo[Z_Q]$ and 
$Q^+$ act by $h_U$ on $V^{U_0}$.

\subsubsection{P-ordinary forms and Hecke algebras}

We now resume the earlier notations and
relate the Hecke operators in \eqref{def:hecke_at_p}
with the operators from \eqref{def:hUm}.

For each  $w\in \Sigma_p$ let $P_w=Q_wU_w\supset B_w$ 
be a standard parabolic subgroup
and its Levi decomposition.
Define 
$P=\prod_{w\in \Sigma_p}P_w$, $Q=\prod_{w\in \Sigma_p}Q_w$, and
$U=\prod_{w\in \Sigma_p}U_w$.
Then $P=QU$ is a parabolic subgroup
of the $p$-adic group $G_p$ and $P_0=K_p\cap P$
is an open compact subgroup of $P$.
We then define $U_0, Z_Q, Q^+, Z_Q^+$ as above.
In particular when $P_w=B_w=T_wN_w$ for all $w\in\Sigma_p$
we have $Z_T=T$, and $T^+=Z_T^+$.

For integers $b$ and $c$ that satisfies $c\geq b\geq 0$ and $c>0$,
we define $\Iw^P(p^{b,c})=\prod_{w\in\Sigma_p}\Iw^P(w^{b,c})$ where
\begin{equation}\label{def:Iwahori_P}
	\Iw^P(w^{b,c})=\{
	k\in K_w\mid 
    k \bmod \varpi_w^c \in P_w(\oo/\varpi_w^c)
	\text{ and }
	k \bmod \varpi_w^b \in U_w(\oo/\varpi_w^b)
	\}.
\end{equation}

\begin{defn}\label{def:hecke}
Let the monoid $U_0T^+$ act on 
$f\in S_{\wt{k}}(U^p\Iw^P(p^{b,c}),A)$ by
\begin{equation}\label{def:T_act}
	(ut* f)(g)=(w_0\wt{k})^{-1}(t)\cdot (ut\cdot f)(g)=
    (w_0\wt{k})^{-1}(t)\cdot \big(\xi_{\wt{k}}(ut)f(gut)\big)\quad
    u\in U_0, t\in T^+.
\end{equation}
Then for $\alpha_w^{(j)}$ and $u\in T_n(\oo_w)$ 
as in \eqref{def:hecke_at_p}
We define the Hecke operators on
$S_{\wt{k}}(U^p\Iw^P(p^{b,c}),A)$ by
$U_{\wt{k},w}^{(j)}=h_U(\alpha_w^{(j)})$
and $\langle u\rangle_{\wt{k}}=h_U(u)$.
We also define $U_P$
as the product of all $U_{\wt{k},w}^{(j)}$
for which $\alpha_w^{(j)}\in Z_Q$.

In particular, when $P=B$
and $m\in T^+$ equals either $\alpha_{w}^{(j)}$ or $u\in T_w(\oo_w)$,
a set of representatives for $U_0/mU_0m^{-1}$
is also a set of representatives for 
$\Iw(p^{b,c})/m\Iw(p^{b,c})m^{-1}$
and consequently 
$h_N(\alpha_w^{(j)})$ and $h_N(u)$
coincide with the operators $U_{\wt{k},w}^{(j)}$ and 
$\langle u\rangle_{\wt{k}}$ in \eqref{def:hecke_at_p}.
\end{defn}

Aside from the operators defined above,
the operators from \eqref{def:hecke_at_p} and
\eqref{def:hecke_at_s} also 
acts on $S_{\wt{k}}(U^p\Iw^P(p^{b,c}),A)$.

\begin{lem}
The Hecke operators commutes with each other
and are equivariant with respect to the inclusions
$ S_{\wt{k}}(U^p\Iw^P(p^{b,c}),A)\hookrightarrow
S_{\wt{k}}(U^p\Iw^P(p^{b',c'}),A)$
if $b'\geq b$ and $c'\geq c$.
\end{lem}
\begin{proof}
That the operators $T_w^{(j)}$ from \eqref{def:hecke_away_p}
commutes with other Hecke operators is classical,
and the equivariance is clear.
For the Hecke operators at $p$,
the commutativity follows from \cite[Lem 3.1.4]{emeI},
and the equivariance follows from 
that of the action \eqref{def:T_act}.
See also \cite[Lem 2.10]{ger} for the proof when $P=B$.
Then the result also holds for the Hecke operators from 
\eqref{def:hecke_at_s}
since the set of representatives for the double cosets
in which can be chosen as those at $p$.
\end{proof}

From this point on,
we assume that $A$ is $E$, or a finite $\oo$-module,
or the Pontryagin dual of a finite $\oo$-module.
In particular, let $\varpi\in \oo$ be a uniformizer, $A$ could be 
$\oo, \oo/\varpi^n\oo, \varpi^{-n}\oo/\oo,$ or $E/\oo$.
Except when $A=E$,
the operator $e_P\coloneqq\lim_{n\to \infty}(U_P)^{n!}$
converges to an idempotent.
We then define the space of $P$-ordinary forms 
with coefficient in $A$ by
\[
	S_{\wt{k}}^{P-\ord}(U^p\Iw^P(p^{b,c}),A)\coloneqq
	e_PS_{\wt{k}}(U^p\Iw^P(p^{b,c}),A)
\]
and $S_{\wt{k}}^{P-\ord}(U^p\Iw^P(p^{b,c}),E)\coloneqq 
S_{\wt{k}}^{P-\ord}(U^p\Iw^P(p^{b,c}),\oo)\otimes_{\oo}E$.
Alternatively,
$S_{\wt{k}}^{P-\ord}(U^p\Iw^P(p^{b,c}),A)$
is characterized as the subspace on which 
any  $U_{\wt{k},w}^{(j)}$ such that 
$\alpha_w^{(j)}\in Z_Q$ acts invertibly.
When $P=B$ we write 
$S_{\wt{k}}^{B-\ord}(U^p\Iw(p^{b,c}),A)=
S_{\wt{k}}^{\ord}(U^p\Iw(p^{b,c}),A)$,
which coincides with \cite[Def 2.13]{ger}.

\begin{defn}\label{def:ord_hecke}
	We let $\TT^P_{\wt{k}}(U^p\Iw^P(p^{b,c}),A)\subset 
    \End_{\oo}S_{\wt{k}}^{P-\ord}(U^p\Iw^P(p^{b,c}),A)$
	be the $\oo$-subalgebra
	generated by all
	$T_w^{(j)}$, for $1\leq j\leq n$,
	and $(T_w^{(n)})^{-1}$ from \eqref{def:hecke_away_p};
	all $U_{\wt{k},w}^{(j)}$ and $\langle u\rangle_{\wt{k}}$
    from Definition \ref{def:hecke}
	for $\alpha_w^{(j)}$ belonging to $Z_Q$ and $u\in T_n(\oo_w)$;
    and all $\langle u\rangle$ from \eqref{def:hecke_at_s}.
    When $P=B$
    we write $\TT^{\ord}_{\wt{k}}(U^p\Iw(p^{b,c}),A)=
    \TT^B_{\wt{k}}(U^p\Iw(p^{b,c}),A)$.
\end{defn}

\begin{lem}\label{lem:control}
	The following natural inclusions are also surjective.
    In particular we may restrict ourselves 
    to modular forms of levels $\Iw^P(p^{b,b})$ or $\Iw^P(p^{0,1})$
    at $p$ when considering $P$-ordinary forms.
	\begin{align*}
	&S_{\wt{k}}^{P-\ord}(U^p\Iw^P(p^{b,b}),A)\hookrightarrow	
	S_{\wt{k}}^{P-\ord}(U^p\Iw^P(p^{b,c}),A)\quad 
	\text{ for } c\geq b\geq 1\\
	&S_{\wt{k}}^{P-\ord}(U^p\Iw^P(p^{0,1}),A)\hookrightarrow	
	S_{\wt{k}}^{P-\ord}(U^p\Iw^P(p^{0,c}),A)\quad \text{ for } c\geq 1
	\end{align*}
\end{lem}
\begin{proof}
	It suffices to show that 
	$(U_P)^{n!}S_{\wt{k}}(U^p\Iw^P(p^{b,c}),A)
	\subset S_{\wt{k}}(U^p\Iw^P(p^{b,b}),A)$
	for $n$ sufficiently large, 
	which follows from \cite[Lem 3.3.2]{emeI}
	since $\Iw^P(p^{b,c})$ admits Iwahori decompositions.
	The same argument also applies to 
	$S_{\wt{k}}(U^p\Iw^P(p^{0,c}),A)$.
	See also \cite[Lem 2.19]{ger} for the proof when $P=B$.
\end{proof}

\begin{lem}\label{lem:PtoB}
	For any $b\geq 1$
    the space $S_{\wt{k}}^{\ord}(U^p\Iw(p^{b,b}),A)$
    is a subspace of 
    $S_{\wt{k}}^{P-\ord}(U^p\Iw(p^{b,b}),A)$
	  and the inclusions
	$S_{\wt{k}}^{\ord}(U^p\Iw(p^{b,b}),A)\subset
	S_{\wt{k}}^{P-\ord}(U^p\Iw^P(p^{b,b}),A)$
    are equivariant with respect to the Hecke operators.
    In particular they induce homomorphisms of $\oo$-algebras
	\[
		\TT^P_{\wt{k}}(U^p\Iw^P(p^{b,b}),A)\to
		\TT^{\ord}_{\wt{k}}(U^p\Iw(p^{b,b}),A)
	\]
\end{lem}
\begin{proof}
	We have the natural inclusions
	$S_{\wt{k}}(U^p\Iw(p^{b,b}),A)\subset 
	S_{\wt{k}}(U^p\Iw^P(p^{b,b}),A)$
	since $U\subset N$ and consequently 
    $\Iw^P(p^{b,b})\subset \Iw(p^{b,b})$.
    Thus the lemma follows if the inclusions
	are equivariant with respect to the Hecke operators.
	This is clear except for the operators 
    $U_{\wt{k},w}^{(j)}=h_U(\alpha)$
    in Definition \ref{def:hecke},
    for $\alpha=\alpha_w^{(j)}\in Z_Q$.
    In this case we need the observation that
    the set of representatives
	\[
	\begin{pmatrix}
		\id_j&X\\&\id_{n-j}
	\end{pmatrix},\quad
	X \text{ runs through a set of representatives of }
	M_{j,n-j}(\oo_w/\varpi_w)
	\]
    given in \cite[Lem 2.10]{ger} for $N_0/\alpha N_0\alpha^{-1}$
	is also a set of representatives for 
	$U_0/\alpha U_0\alpha^{-1}$.
\end{proof}

\subsection{Weights independence}

For a parabolic subgroup $P$ as introduced above,
we would like to compare the space of 
$P$-ordinary modular forms
of a domiant weight 
$\wt{k}=(k_\sigma)\in (\Z^n)^{\Sigma}$
and that of the trivial weight.
For each $w\in\Sigma_p$ and $\sigma\in I_w$
let $\pi_{k_{\sigma}}$ denote 
the algebraic $Q_w$-representation
$\Ind_{B_w\cap Q_w}^{Q_w}(w_0 k_\sigma)$
and let $\pi_{\wt{k}}$ be the $Q$-representation
as defined by the recipe in \eqref{def:algrep}.
Then $\pi_{\wt{k}}(\oo)$ is an algebraic representation on
$Q_0\coloneqq \Q\cap K_p$,
which we extends  to $P_0\coloneqq P\cap K_p$
via the projection $P=QU\to Q$.

Note that $\pi_{\wt{k}}$
is simply the character $w_0\wt{k}$
of $Q=T$ when $P=B$.
Therefore the following proposition can be seen as
a generalization of \cite[Prop 2.22]{ger}.

\begin{lem}
	Let $\pi_{\wt{k}}^*$ denote the contragredient
	representation and $A=\varpi^{-r}\oo/\oo$.
	For each $b\geq r$ there exists 
    the following isomorphism 
    which is equivariant with respect to the Hecke operators
    in Definition \ref{def:ord_hecke}.
	\[
		\epsilon_{\wt{k}} \colon 
		S_{\wt{k}}^{P-\ord}(U^p\Iw^P(p^{b,b}),A)\cong 
		\Hom_{\oo}(\pi_{\wt{k}}^*(\oo),
		S^{P-\ord}(U^p\Iw^P(p^{b,b}),A)).
	\]
	Moreover, the isomorphism is $Q_0$-equivariant 
    with respect to
	the action of $u\in Q_0$ defined as follows.
	\begin{align*}
	&u\cdot F(g)=\xi_{\wt{k}}(u)\cdot F(gu),\quad
	F(g)\in S_{\wt{k}}^{P-\ord}(U^p\Iw^P(p^{b,b}),A)\\
	&u\cdot \phi(v^*)(g)=
	\phi(\pi^*_{\wt{k}}(u^{-1})\cdot v^*)(gu),\quad
	\phi\in \Hom_{\oo}(\pi_{\wt{k}}^*(\oo),
	S^{P-\ord}(U^p\Iw^P(p^{b,b}),A))
	\end{align*}
\end{lem}

\begin{proof}
	By inductions in steps
	we can fix an isomorphism 
	$\xi_{\wt{k}}\cong \Ind_{P}^{G_p}\pi_{\wt{k}}$.
	Let $ev\colon \xi_{\wt{k}}\to \pi_{\wt{k}}$
	be the evaluation at the identity.
	For $F(g)\in S_{\wt{k}}(U^p\Iw^P(p^{b,b}),A)$,
	we define 
	$\epsilon_{\wt{k}}(F)$ by
	\begin{equation}\label{eq:wt_indep}
	\epsilon_{\wt{k}}(F)\colon 
	\pi^*_{\wt{k}}(\oo)\rightarrow
	S(U^p\Iw^P(p^{b,b}),A)\quad
	v^*\mapsto [g\mapsto v^*(ev(F(g)))].
	\end{equation}
    As $b\geq r$,
	the action of $\Iw^P(p^{b,b})$ on 
	$A\otimes_{\oo}\pi_{\wt{k}}(\oo)$
	is trivial.
	Thus the function defined above is indeed 
	a modular form
	of trivial weight.
	That the map is equivariant with respect
    to the $Q_0$-actions
    follows from $ev\circ \xi_{\wt{k}}(u)=\pi_{\wt{k}}(u)\circ ev$
    when $u\in Q_0$.
    It is also straightforward that the map
    is equivariant with respect to the Hecke operators away $p$.
	For the Hecke operators at $p$,
    the equivariance follows from that 
    $(w_0\wt{k})^{-1}(z)\cdot ev\circ \xi_{\wt{k}}(z)=
    (w_0\wt{k})^{-1}(z)\cdot \pi_{\wt{k}}(z)\circ ev=
    ev$ when $z\in Z_Q$.

	To construct the reversed map,
	note that if $\mu$ is a weight character of $T$ in  
	$\pi_{\wt{k}}$, then it is also a weight character 
	of $T$ in $\xi_{\wt{k}}$.
	We fix weight vectors $v_\mu\in \xi_{\wt{k}}$
	and $v^*_\mu\in \pi_{\wt{k}}^*$
	such that $v^*_{\mu}(ev(v_\mu))=1$.
	Now, let $\alpha_P\in Z_Q^+$ be the product
	of all $\alpha_w^{(j)}\in Z_Q$ and $\alpha=\alpha_P^r$,
	Let $\{x_i\}_{i\in I}$
	be a set of represntatives 
	for $U_0/\alpha U_0\alpha^{-1}$,
	we put 
	\begin{align*}
		\varphi\colon 
		\Hom_{\oo}(\pi_{\wt{k}}^*(\oo),&
		S(U^p\Iw^P(p^{b,b}),A))\longrightarrow
		S_{\wt{k}}(U^p\Iw^P(p^{b,b}),A)\\
		\phi&\mapsto 
		F_\phi(g)=\sum_{i\in I} \sum_{\mu}
		\phi(v^*_\mu)(gx_i\alpha)\otimes
		\xi_{\wt{k}}(x_i) v_\mu
	\end{align*}
	where $\mu$ runs through the weight characters in 
	$\pi_{\wt{k}}$.
	To show that the resulting function 
	defines a modular form,
	let $u\in \Iw^P(p^{b,b})$, 
	then as explained in \cite[Prop 2.22]{ger}
	there exists a bijection $i\mapsto i'$ of $I$
	such that 
	 \[
		ux_i=x_{i'}v_i,\quad
		v_i\in\alpha\Iw^P(p^{b,b})\alpha^{-1} 
		\cap \Iw^P(p^{b,b})
	\]
	Since each $v_i$ is reduced to the identity matrix 
	modulo $\varpi^r$ and thus acts trivially on 
	$A\otimes_{\oo}\xi_{\wt{k}}(\oo)$,
	\[
		\xi_{\wt{k}}(u)\cdot F_\phi(gu)=
		\sum_{i\in I}\sum_{\mu}
		\phi(v^*_\mu)(gx_i'v_i\alpha)\otimes
		\xi_{\wt{k}}(x_i'v_i) v_\mu=
		\sum_{i\in I}\sum_{\mu}
		\phi(v^*_\mu)(gx_i'\alpha)\otimes
		\xi_{\wt{k}}(x_i')
        v_\mu=F_\phi(g).
	\]

	At last, we observe that for each $\mu$ 
	the composition
	$\epsilon_{\wt{k}}(F_\phi)$ is the homomorphism
	\[
		v_\mu^*\mapsto \sum_{i\in I}\phi(v_\mu^*)
		(gx_i\alpha) =U_P^r\phi(v_\mu^*)(g)
	\]
	On the other hand 
	if we decompose $F$ with respect to a choice of 
	weight vectors
	$F(g)=\sum_\mu F_\mu(g)v_\mu+
	\sum_{\mu'}F_{\mu'}(g)v_{\mu'}$, 
	with $\mu$ goes through weight vectors 
	that also appears in $\pi_{\wt{k}}$
	and $\mu'$ goes through the complement,
	then we have
	$\mu(\alpha)=(w_0\wt{k})(\alpha)$ for all $\mu$
	and  $\varpi^r(w_0\wt{k})(\alpha)$ divides $\mu'(\alpha)$
	in $\oo$ for all $\mu'$.
	Therefore
	\begin{multline*}
	U_P^rF(g)=
	\sum_{i\in I}
	\sum_\mu \xi_{\wt{k}}(x_i)\cdot F_\mu(gx_i\alpha)v_\mu+
	\sum_{i\in I}
	\sum_{\mu'}\frac{\mu'(\alpha)}{(w_0\wt{k})(\alpha)}
	\xi_{\wt{k}}(x_i)\cdot F_{\mu'}(gx_i\alpha)v_{\mu'}\\=
	\sum_{i\in I}
	\sum_\mu \xi_{\wt{k}}(x_i)\cdot F_\mu(gx_i\alpha)v_\mu=
	\sum_{i\in I}
	\sum_\mu 
	\epsilon_{\wt{k}}(F)(v^*_\mu)(gx_i\alpha)\otimes
    \xi_{\wt{k}}(x_i)v_\mu
	=F_{\epsilon_{\wt{k}}(F)}(g).
	\end{multline*}

	We thus have the following commutative diagram,
	from which the proposition follows.
	\[
	\begin{tikzcd}
		S_{\wt{k}}(U^p\Iw^P(p^{b,b}),A)
		\arrow[r,"\epsilon_{\wt{k}}"]
		\arrow[d,"U_P^r"]
		& \Hom_\oo(\pi^*_{\wt{k}}(\oo), S(U^p\Iw^P(p^{b,b}),A))
		\arrow[d,"U_P^r"]
		\arrow[dl,"\varphi"]\\
		S_{\wt{k}}(U^p\Iw^P(p^{b,b}),A)
		\arrow[r,"\epsilon_{\wt{k}}"]
		& \Hom_\oo(\pi^*_{\wt{k}}(\oo), S(U^p\Iw^P(p^{b,b}),A))
	\end{tikzcd}	
	\]
\end{proof}

Let $S_{\wt{k}}^{P-\ord}(U^p,E/\oo)=
\varinjlim_{b}
S_{\wt{k}}^{P-\ord}(U^p\Iw^P(p^{b,b}),E/\oo)$
be the injective limit under inclusions
and define 
\[
	\TT^P_{\wt{k}}(U^p,E/\oo)=
	\varprojlim_{b}
	\TT^P_{\wt{k}}(U^p\Iw^P(p^{b,b}),E/\oo)
\]
Since $S_{\wt{k}}^{P-\ord}(U^p,E/\oo)$
is also the injective limit of 
$S_{\wt{k}}^{P-\ord}(U^p\Iw^P(p^{b,b}),\varpi^{-b}\oo/\oo)$
and $\pi_{\wt{k}}^*(\oo)$ is finite over $\oo$,
the following proposition
follows immediately from the previous lemma.

\begin{prop}\label{prop:wt_indep}
	We have the following isomorphism
	which is equivariant with respect to the 
	Hecke operators in Definition \ref{def:ord_hecke}
    and the $Q_0$-action  defined in previous lemma.
	\[
		\epsilon_{\wt{k}} \colon 
		S_{\wt{k}}^{P-\ord}(U^p,E/\oo)\cong 
		\Hom_{\oo}(\pi_{\wt{k}}^*(\oo),
		S^{P-\ord}(U^p,E/\oo)).
	\]
	In particular, this isomorphism 
	induces the following surjective homomorphism
	between the Hecke algebras
	\[
		\varphi_{\wt{k}}\colon 
		\TT^P(U^p,E/\oo)\twoheadrightarrow
		\TT^P_{\wt{k}}(U^p,E/\oo).
	\]
    which satisfies the following relations.
    \begin{itemize}
    \item $\varphi_{\wt{k}}(T_w^{(j)})=T_w^{(j)}$ for $T_w^{(j)}$
    from \eqref{def:hecke_away_p}.
    \item $\varphi_{\wt{k}}(U_{w}^{(j)})=U_{\wt{k},w}^{(j)}$ 
    and  $\varphi_{\wt{k}}(\langle u\rangle)=
    \langle u\rangle_{\wt{k}}$ 
    for $U_{\wt{k},w}^{(j)}$ and $\langle u\rangle_{\wt{k}}$
    from Definition \eqref{def:hecke}.
    \item $\varphi_{\wt{k}}(U_{w}^{(j)})=U_{w}^{(j)}$ 
    and  $\varphi_{\wt{k}}(\langle u\rangle)=
    \langle u\rangle$  
    for $U_{w}^{(j)}$ and $\langle u\rangle$
    from \eqref{def:hecke_at_s}.
    \end{itemize}
\end{prop}

\subsection{Completed homology and cohomology}

When $A=E/\oo, \oo/\varpi^{r}\oo$ or $\varpi^{-r}\oo/\oo$
and $\{U_p\}$ is the filtered system of
all compact open subgroups in $K_p$, 
we define the inverse limit
\begin{equation}\label{eq:complete}
	S(U^p,A)\coloneqq
	\varinjlim_{U_p}S(U^pU_p,A)\in 
	\Mod^{\adm}_{G_p}(\oo).
\end{equation}
The natural $G_p$ action on which
defined by Definition \ref{def:algform}
is then the usual right translation
on modular forms of trivial weight.
Moreover, when $A=E/\oo$ we have
\[
	S(U^p,E/\oo)^{U_0}=
	\varinjlim_{b}
	S(U^p\Iw^P(p^{b,b}),\varpi^{-b}\oo/\oo)
\]
where each object on the right is 
a finite $\oo$-module.
It then follows from \cite[Lem 3.1.5]{emeI} and \cite[Prop 3.2.4]{emeI}
that $S^{P-\ord}(U^p,E/\oo)\cong \Ord_P(S(U^p,E/\oo))$,
which belongs to $\Mod^{\adm}_Q(\oo)$
by \cite[Thm 3.3.3]{emeI}.
We define the $P$-ordinary completed homology and cohomology by
\begin{align}\label{eq:completed_coh}
	M_P(U^p)&=
	\Ord_P(S(U^p,E/\oo))^\vee
	\coloneqq \Hom_\oo(\Ord_P(S(U^p,E/\oo)),E/\oo)\\
	S_P(U^p)&=\Hom_\oo(E/\oo, \Ord_P(S(U^p,E/\oo)))
	\cong \varprojlim_r \Ord_P(S(U^p,\oo/\varpi^{r}))
\end{align}
Let $\Hom_\oo^{\cts}(M(U^p),\oo)$
be the set of
$\Phi\in \Hom_\oo(M(U^p),\oo)$ 
such that for any positive integer $r$,
there exists a sufficiently large integer $b$ so that 
the reduction of $\Phi$ modulo $\varpi^r$
factors through
the Pontryagin dual of 
$S^{P-\ord}(U^p\Iw^P(p^{b,b}),E/\oo)\subset \Ord_P(S(U^p,E/\oo))$. 
It can be verified that 
\begin{equation}\label{eq:complete_isom}
	M_P(U^p)\cong \Hom_\oo(S_P(U^p),\oo),\qquad
	S_P(U^p)\cong \Hom_\oo^{\cts}(M_P(U^p),\oo)
\end{equation}
from which we see that
$\TT^P(U^p,E/\oo)$ acts faithfully
on both $M_P(U^p)$ and  $S_P(U^p)$.
In fact, 
if we define $S_{\wt{k}}^{P-\ord}(U^p,\oo)
=\varinjlim_{b}S_{\wt{k}}^{P-\ord}(U^p\Iw^P(p^{b,b}),\oo)$ and 
$\TT^P_{\wt{k}}(U^p,\oo)=
\varprojlim_b\TT_{\wt{k}}^{P}(U^p\Iw^P(p^{b,b}),\oo)$,
the same argument in \cite[Lem 2.17]{ger}
shows that  $\TT^P_{\wt{k}}(U^p,\oo)\cong \TT^P_{\wt{k}}(U^p,E/\oo)$
as $\oo$-algebras.
We remark that this is in line with the fact that 
$S^{P-\ord}(U^p,\oo)$ is dense in $S_P(U^p)$.

\begin{defn}\label{def:big_hecke}
    We call $\TT^P(U^p,\oo)\cong \TT^P(U^p,E/\oo)$
	the big $P$-ordinary Hecke algebra,
    which, through the above identifications,
    acts faithfully on the spaces
	$S^{P-\ord}(U^p,E/\oo)=\Ord_P(S(U^p,E/\oo)), M_P(U^p), S_P(U^p)$,
	and $S^{P-\ord}(U^p,\oo)$. 
	We also define $\TT^P(U^p,E)=\TT^P(U^p,\oo)\otimes_{\oo}E$,
	which acts faithfully
	on the $E$-Banach space $G_p$-representation
    $S_P(U^p)_E\coloneqq S_P(U^p)\otimes_{\oo}E$.
    When $P=B$ we write 
    $\TT^B(U^p,A)=\TT^{\ord}(U^p,A)$.
\end{defn}

Note that by Lemma \ref{lem:PtoB}
we have the following inclusions
that are compatible among integers $b\geq 1$
\[
	S_{\wt{k}}^{\ord}(U^p\Iw(p^{b,b}),E/\oo)\subset
	S_{\wt{k}}^{P-\ord}(U^p\Iw^P(p^{b,b}),E/\oo)
\]
which are equivariant with the Hecke operators.
Consequently we have a Hecke-equivariant inclusion
$\iota_P\colon \Ord_B(S(U^p,E/\oo)\hookrightarrow\Ord_P(S(U^p,E/\oo)$
and the lemma below follows directly.

\begin{lem}\label{lem:coh_to_ord}
    The Hecke-equivariant inclusion
    induces an $\oo$-algebra homomorphism
    \[
        \TT^P(U^p,\oo)\to \TT^{\ord}(U^p,\oo)
    \]
    with respect to which 
    the surjective homomorphism 
	$(\iota_P)^\vee\colon M_P(U^p)\to M^{\ord}(U^p)$ 
    is also equivariant.
\end{lem}

\begin{lem}\label{lem:inj}
	The restriction of the $Q$-module
	$\Ord_P(S(U^p,E/\oo))$ to $Q_0$ 
	is an injective object
	in $\Mod^{\sm}_{Q_0}(\oo)$.
\end{lem}
\begin{proof}
	Following the strategy of 
	\cite[Prop 3.2.4]{pan}, 
	it suffices to show the surjectivity of
	\[
		\Hom_{\oo[Q_0]}(\pi,\Ord_P(S(U^p,E/\oo)))\to 
		\Hom_{\oo[Q_0]}(\pi_1,\Ord_P(S(U^p,E/\oo)))
	\]
	when $\pi_{1}\hookrightarrow \pi$ 
	is an injective morphism between admissible 
    $Q_0$-representations of finite $\oo$-modules.
    Expand the definition of $\Ord_P$ from \eqref{def:OrdP}
    and note that any homomorphism  in 
    $\Hom_{\oo[Q_0]}(\pi_1,\Ord_P(S(U^p,E/\oo)))$ factors through
	\begin{multline*}
		\Hom_{\oo[Q_0]}(\pi_1,
		\Hom_{\oo[Z_Q^+]}
		(\oo[Z_Q], S(U^p\Iw^P(p^{b,b}),
		\varpi^{-r}\oo/\oo)))\\=
		\Hom_{\oo[Z_Q^+]}(\oo[Z_Q],
		\Hom_{\oo[Q_0]}(\pi_1, 
		S(U^p\Iw^P(p^{b,b}),\varpi^{-r}\oo/\oo)))
	\end{multline*}
	for some $b$ and $r$ sufficiently large
    since $\pi_1$ is admissible and $\oo$-finite,
	on which 
	the $Z_Q$-fintieness condition is automatic
	by \cite[Lem 3.1.5]{emeI}.

    Let $\Iw(p^{0,b})$ 
    act on $\pi$ through $Q_0\cap \Iw(p^{0,b})$
    by the Iwahori decomposition
    and $\pi^\vee$ denote the Pontryagin dual of $\pi$.
	Define the space of modular forms with
    coefficients in $\pi^\vee$ by
	\[
		S_{\pi^\vee}(U^p\Iw(p^{0,b}))=
		\{
			F\colon G(\F)\backslash G(\A_f)\to 
			\pi^\vee\mid 
			F(gu)=\pi^\vee(u_p)^{-1}\cdot F(g),\,
			u\in \Iw(p^{0,b})
		\}
	\]
	Enlarge $b$ if necessary, we may assume
	$Q_0\cap \Iw(p^{b,b})$ acts trivially on $\pi$.
	Then there exists an isomorphism
	\[
		\epsilon\colon 
		S_{\pi^\vee}(U^p\Iw(p^{0,b}))\cong 
		\Hom_{\oo[Q_0]}(\pi,
		S(U^p\Iw^P(p^{b,b}),\varpi^{-r}\oo/\oo)))
		\quad \epsilon(F)\colon
		v\mapsto [g\mapsto v(F(g))]
	\]
	and similarly for $\pi_1$.
    Since \eqref{cond:small} implies that
	that $S_{\pi^\vee}(U^p\Iw(p^{0,b}))$
	and $S_{\pi_1^\vee}(U^p\Iw(p^{0,b}))$
	are direct sums of 
	$\pi^\vee$ and  $\pi_1^\vee$
    indexed by
	$G(\F)\backslash G(\A_f)/U^p\Iw(p^{0,b})$,
    the natural map
	$S_{\pi^\vee}(U^p\Iw(p^{0,b}))\to 
	S_{\pi_1^\vee}(U^p\Iw(p^{0,b}))$ is then
	surjective
	as $\pi^\vee\to \pi_1^\vee$ is surjective.
	Localize to the $P$-ordinary parts
	and apply \cite[Lem 3.1.5]{emeI} again,
	we see that
	\[
		\Hom_{\oo[Z_Q^+]}(\oo[Z_Q],
		S_{\pi^\vee}(U^p\Iw(p^{0,b})))\to 
		\Hom_{\oo[Z_Q^+]}(\oo[Z_Q],
		S_{\pi^\vee_1}(U^p\Iw(p^{0,b})))
	\]
	is also surjective, from which 
	the proposition follows.
\end{proof}

\begin{prop}\label{prop:density}
The subspace $S_P(U^p)_E^{\textnormal{alg}}$ 
of $Q_0$-algebraic vectors, defined as
\[
\Image\left(ev\colon
\bigoplus_{\wt{k}}\Hom_{E[Q_0]}(\pi_{\wt{k}}^*(\oo), S_P(U^p)_E)
\otimes_E \pi_{\wt{k}}^*(E)\rightarrow S_P(U^p)_E\right)
\]
where $\wt{k}$ ranges through all dominant weights,
is dense in the $E$-Banach space $S_P(U^p)_E$.
\end{prop}
\begin{proof}
	Since $\Ord_P(S(U^p,E/\oo))$ is an injective object
	in $\Mod_{Q_0}^{\sm}(\oo)$
	by Lemma \ref{lem:inj},
	we may follow the strategy of 
	\cite[Prop 3.2.9]{pan}
	and use \cite[Cor 3.2.6]{pan}
	to reduce the statement to that of
	$\mathcal{C}(Q_0,E)$,
	the space of continuous  $E$-valued
	functions on $Q_0$.
	The density result then follows from
	\cite[Prop 6.A.17]{Pask14}.
\end{proof}

\begin{prop}\label{prop:wt_space}
	There exists a Hecke-equivariant isomorphism
	\[
	S_{\wt{k}}^{P-\ord}(U^p\Iw^P(p^{0,1}),E)\cong 
	\Hom_{\oo[Q_0]}(\pi_{\wt{k}}^*(\oo), S_P(U^p)_E)
	\]
\end{prop}
\begin{proof}
    Apply $\Hom_\oo(E/\oo,*)$ to the isomorphism 
    in Proposition \ref{prop:wt_indep}
    and use \eqref{eq:complete_isom}, we obtain 
	\begin{multline*}
		\Hom_\oo(E/\oo, S_{\wt{k}}^{P-\ord}(U^p,E/\oo))\cong 
		\Hom_\oo(E/\oo,
		\Hom_{\oo}(\pi_{\wt{k}}^*(\oo),
		S^{P-\ord}(U^p,E/\oo)))\\=
		\Hom_{\oo}(\pi_{\wt{k}}^*(\oo),
		\Hom_\oo(E/\oo,
		S^{P-\ord}(U^p,E/\oo)))\cong
		\Hom_{\oo}(\pi_{\wt{k}}^*(\oo), S_P(U^p))
	\end{multline*}
	which is equivariant with respect to 
	the Hecke operators and the $Q_0$-actions defined there.
    Observe that 
	taking the $Q_0$-invariant subspaces
	on the right hand side gives
	$\Hom_{\oo[Q_0]}(\pi_{\wt{k}}^*(\oo), S(U^p))$.

	On the other hand, since
	$\Hom_\oo(E/\oo, S_{\wt{k}}^{P-\ord}(U^p,E/\oo))\cong
	\varprojlim_r S_{\wt{k}}^{P-\ord}(U^p,\oo/(\varpi^r))$,
	by Lemma \ref{lem:control}
	the $Q_0$-invariant subspace 
	on the left hand side
	is $S_{\wt{k}}^{P-\ord}(U^p\Iw^P(p^{0,1}),\oo)$.
	The claimed result now follows by
	tensoring both subspaces with $E$.
\end{proof}

\begin{rem}
	The above results generalize
	\cite[Prop 3.2.9]{pan} and 
	\cite[\S 3.2.10]{pan},
	which deals with the case 
	when $n=2$ and $P=G_p$,
	in which the $P$-ordinary condition is empty.
\end{rem}

\begin{cor}\label{cor:density}
    When $\wt{k}$ goes through all dominant weight, 
    there exists an injective homomorphism
    \[
        \TT^P(U^p,E)\to \prod_{\wt{k}}\TT^P(U^p\Iw^P(p^{0,1}),E)
    \]
\end{cor}
\begin{proof}
This is an immediate consequence of the previous two propositions.
\end{proof}

\subsection{Hida family}

We digress temporarily 
and give an ad-hoc definition
of Hida families of $B$-ordinary modular forms
in terms of the completed homology and cohomology.

For each $w\in \Sigma_p$ and $b\geq 0$
let $T_n(w^b)=\{u\in T_n(\oo_w)\mid u\equiv \id \mod\varpi_w^b\}$
so that $T_n(w^b)=T_n(\oo_w)$ when $b=0$
and put $T_n(p^b)=\prod_{w\in\Sigma_p}T_n(w^b)$.
By the Teichm\"{u}ller lift
there exists a finite subgroup $\Delta_p\subset T_n(p^0)$
of prime-to-$p$ order such that 
$T_n(p^0)=\Delta_p\times T_n(p^1)$.
And for $v\in S$ 
we put $\Delta_v=U_v/\Iw_1(w)$,
where $w$ is the fixed prime above $v$,
and define $\Delta_S=\prod_{v\in S}\Delta_v$.
Then we define the completed group rings
\begin{equation}\label{def:lambda_rings}
    \Lambda\coloneqq \oo\llbracket T_n(p^1)\rrbracket,\,
    \text{ and }\,
    \Lambda^+\coloneqq \Lambda[\Delta_p\times \Delta_S].
\end{equation}
The diamond operators 
defined in \eqref{def:hecke_at_p} and \eqref{def:hecke_at_s}
then induces an $\oo$-algebra homomorphism
\begin{equation}\label{eq:Lambda_hecke}
    \langle *\rangle\colon \Lambda^+\to 
    \TT^{\ord}(U^p, \oo)\quad
    u\mapsto \langle u\rangle.
\end{equation}

\begin{defn}\label{def:Hida_family}

Let $\I$ be a local complete Noetherian ring
that is finite over $\Lambda^+$ and flat over $\oo$.
Write $S_B(U^p)=S^{\ord}(U^p)$, we define
the space of $\I$-adic Hida families 
of tame level $U^p$ as 
\begin{equation}\label{eq:Hida_family}
    S^{\ord}(U^p,\I)\coloneqq 
    \{\euF\in S^{\ord}(U^p)\widehat{\otimes}_\oo\I\mid 
    (\langle u\rangle\otimes\id)\euF=
    (\id\otimes\langle u\rangle)\euF\, \text{ for }
    u\in \Lambda^+\}.
\end{equation}
Here $(\langle u\rangle\otimes\id)$ denotes
the action through the Hecke algebra and
$(\id\otimes\langle u\rangle)$ denotes the action
through $\I$.
Note that $S^{\ord}(U^p,\I)$
is stable by the usual Hecke action on 
the first factor since $\TT^{\ord}(U^p,\oo)$ is commutative.
This defines the action of $\TT^{\ord}(U^p,\oo)$ on $S^{\ord}(U^p,\I)$.
\end{defn}

\begin{rem}
Let $\Delta'$ be the prime-to-$p$ part of the 
finite abelian group $\Delta_p\times \Delta_S$.
Since $\I$ is local,
the restriction of the $\Lambda^+$-algebra to $\oo[\Delta']$,
potentially after replacing $\oo$ by a finite extension,
determines a character 
$\gamma\colon\Delta'\to \oo^\times$
such that $\oo[\Delta']\to \I$
factors through the homomorphism induced by $\gamma$.

\end{rem}

\begin{lem}
    Write $M_B(U^p)=M^{\ord}(U^p)$
    and $M^{\ord}(U^p,\I)=M^{\ord}(U^p)\otimes_{\Lambda^+}\I$.
    Then there exists a Hecke-equivariant isomorphism
    $S^{\ord}(U^p,\I)\cong\Hom_\I(M^{\ord}(U^p,\I), \I)$.
\end{lem}
\begin{proof}
Recall from \eqref{eq:complete_isom} that
$S^{\ord}(U^p)\cong \Hom^{\cts}_\oo(M^{\ord}(U^p),\oo)$.
From which we have
$S^{\ord}(U^p)\widehat{\otimes}_\oo\I
\cong \Hom^{\cts}_\oo(M^{\ord}(U^p),\I)$.
Then the subspace on which the actions
$(\langle u\rangle\otimes\id)$ and
$(\id\otimes\langle u\rangle)$
coincide is precisely
\[
    \Hom^{\cts}_{\Lambda^+}(M^{\ord}(U^p),\I)\cong
    \Hom_{\I}(M^{\ord}(U^p,\I),\I).
\]
\end{proof}

\begin{prop} \label{prop:ord_to_dual}
The spaces $M^{\ord}(U^p,\I)$ and
$S^{\ord}(U^p,\I)$,
are finite free over $\I$
under the condition \eqref{cond:s-ram} and \eqref{cond:small}.
In particular there exists a Hecke-equivariant isomorphism
\[
    M^{\ord}(U^p,\I)\cong \Hom_\I(S^{\ord}(U^p,\I),\I).
\]
\end{prop}

\begin{proof}
Under the condition \eqref{cond:s-ram} and \eqref{cond:small},
\cite[Lem 2.6]{ger} implies that the spaces 
$S(U^p\Iw(p^{b,b}),\oo)$ are finite free over 
$\oo[T_n(p^1)/T_n(p^b)][\Delta]$.
Let $\fa_b\subset\Lambda^+$
be the augmented ideal generated by $T_n(p^b)$ for $b\geq1$.
We can then argue as in \cite[Prop 2.20]{ger} and show that
\[
    M^{\ord}(U^p)/\fa_b M^{\ord}(U^p)
\]
is finite-free over $\I/\fa_b\I$
of constant rank
$r=\operatorname{rank}_\oo S^{\ord}(\tilde{U}^p\Iw(p^{0,1}),\oo)$,
where $\tilde{U}^p$ is the prime-to-$p$ part
of the open compact subgroup $\tilde{U}$ in \eqref{cond:small}.
Therefore $M^{\ord}(U^p)$ is finite free over $\Lambda^+$
of rank $r$ and 
$M^{\ord}(U^p,\I)$ is finite free over $\I$ of rank $r$.
It then follows from the previous lemma that 
$M^{\ord}(U^p,\I)\cong \Hom_\I(S^{\ord}(U^p,\I),\I)$.

\end{proof}

\subsection{Hecke algebras and Galois representations}

Let $\mathcal{A}$ be
the space of automorphic forms on $G(\A)$,
on which $G(\A)$ acts by right translation.
And for dominant $\wt{k}\in (\Z^n)^{\Sigma}$
let $G(\A_f)$ act on $S_{\wt{k}}(\bar{\Q}_p)$
as in Definition \ref{def:algform}.
We let $\xi_{\wt{k}}^*(\C)$
denote the algebraic
$G(\A_\infty)$-representation
defined by the inclusions
$G(\F_\sigma)\subset \GL_n(\F_\sigma\otimes \K)=\GL_n(\C)$
for each $\sigma\in\Sigma$.
By \cite[Prop 3.3.2]{CHT},
there exists an $G(\A_f)$-equivariant isomorphism
\begin{equation}\label{eq:p_to_infty}
	\iota\colon S_{\wt{k}}(\bar{\Q}_p)\otimes_{\iota,\bar{\Q}_p}\C
	\rightarrow \Hom_{G(\A_\infty)} (\xi_{\wt{k}}^*(\C), \mathcal{A})\quad
	\iota(F)\colon v^*\mapsto 
	[g\mapsto \xi_{\wt{k}}(g_\infty)\cdot 
    \iota\left(\xi_{\wt{k}}(g_p)F(g_f)\right)].
\end{equation}

\begin{prop}\cite[Prop.2.27]{ger}\label{prop:gal_ger}
	Let $\pi$ be an irreducible constituent of the
	$G(\A_f)$-representation $S_{\wt{k}}(\bar{\Q}_p)$,
	then there exist a unique 
	continuous semi-simple Galois representation
	\[
	r_\pi: \Gal_\K \rightarrow \GL_n(\bar{\Q}_p)\quad
	\]
    satisfying $r_\pi^c \cong r_\pi^{\vee} \epsilon^{1-n}$,
	where $\epsilon$ is the $p$-th cyclotomic character,
	with the following properties.
\begin{enumerate}[label=(\alph*)]
\item If $v\in \finite$ is inert and $\pi_v$
is fixed by a hyperspecial maximal compact open subgroup 
of $G(\F_v)$ then $r_\pi\vert_{D_v}$ is unramified.
\item If $v\in \finite$ splits into $w\bw$ in $\K$
let $\pi_w$ denote the $\GL_n(\K_w)$-representation
$\pi_v\circ \iota_w^{-1}$, then
\[
\WD\left(\left.r_\pi\right|_{D_w}\right)^{\mathrm{ss}} \cong
\Rec(\pi_w|\cdot|^{\frac{1-n}{2}})^{\mathrm{ss}}.
\]
Moreover, $r_\pi\vert_{D_w}$ is unramified if $\pi_w$ is unramified.
\item 
If $w\in\Sigma_p$ the representation
$r_\pi\vert_{D_w}$ is potentially semi-stable.
Moreover $r_\pi\vert_{D_w}$ is crystalline
if the representation $\pi_w$ defined as above is unramified,
in which case the characteristic polynomial of 
the geometric Frobenius $\Fr_w$ on 
$\WD\left(D_{\mathrm{cris }}\left(\left.r_\pi\right|_{D_w}\right)\right)$
coincides with that of 
$\Rec(\pi_w|\cdot|^{\frac{1-n}{2}})^{\mathrm{ss}}$.
\item 
Define $k_{\sigma,j}=-k_{\sigma c, n-j+1}$ for $\sigma\in \Sigma^c$.
If $w\mid p$ and  $\sigma\in I_w$, then 
$\dim_{\bar{\Q}_p}\operatorname{gr}^i
\left(r_\pi \otimes_{\sigma, \K_w} B_{\dR}\right)^{D_w}=1$
exactly when $i=k_{\sigma, j}+n-j$ 
for $j=1, \ldots, n$ and is equal to 0 otherwise.
\end{enumerate}
\end{prop}

In the rest of the section,
we restrict ourselves to the following case.
\begin{equation}\label{cond:parabolic}\tag{P}
	n=2,\, 
	P_{w_0}=G_{w_0} \text{ for the fixed prime }w_0\in \Sigma_p,\,
	P_{w'}=B_{w'} \text{ for } w'\in \Sigma_p\setminus\{w_0\}.
\end{equation}
Note that by definition \eqref{def:Iwahori_P}
we have $K_{w_0}\subset \Iw^P(p^{0,1})$.

\begin{lem}\label{lem:hecke_red}
Both
$\TT^P_{\wt{k}}(U^p\Iw^P(p^{b,b}),\oo)$ and
$\TT^{\ord}_{\wt{k}}(U^p\Iw(p^{b,b}),\oo)$
are finite flat reduced $\oo$-algebras.
\end{lem}
\begin{proof}
For $\TT^{\ord}_{\wt{k}}(U^p\Iw(p^{b,b}),\oo)$
this is \cite[Lem 2.14]{ger}
so we only prove the lemma for
$\TT^P_{\wt{k}}(U^p\Iw^P(p^{b,b}),\oo)$.

It is clear that the Hecke algebra,
which is a subalgebra of the endormorphism ring
of the finite free $\oo$-module
$S^{P-\ord}_{\wt{k}}(U^p\Iw^P(p^{b,b}),\oo)$,
is finite flat over $\oo$.
For the reducedness,
since the Hecke algebra is $\oo$-flat,
it suffices to tensor $E$ and show that 
$\TT^P_{\wt{k}}(U^p\Iw^P(p^{b,b}),E)$
is semi-simple,
or the Hecke operators
generating which act semi-simply on 
$S^{P-\ord}_{\wt{k}}(U^p\Iw^P(p^{b,b}),\oo)$.
This is clear for 
the operators $T_{w}^{(j)}$ from \eqref{def:hecke_away_p}
and the operators $\langle u\rangle $
from \eqref{def:hecke_at_s}.
For the operators at $w'\in\Sigma_p\setminus\{w_0\}$
this follows from \cite[Lem 2.14]{ger}.
For the operators at $w=w_0$,
we only need to consider the operators
$h_U(\smat{\varpi_{w}&\\&\varpi_w})$ and
$h_U(\smat{x&\\&x})$ for $x\in \oo_w^\times$.
These are essentially the central action
of $\K_w^\times\subset \GL_2(\K_w)$
and the semi-simplicity of which is also clear.

\end{proof}

Consequently $\TT^P_{\wt{k}}(U^p\Iw^P(p^{b,b}),E)$
is the a product of residue fields 
$E_{\fp}\coloneqq \TT^P_{\wt{k}}(U^p\Iw^P(p^{b,b}),E)/\fp$,
which are finite extensions over $E$,
where $\fp$ goes through minimal primes.
By abuse of notation, let $\fp$ also denote 
the height one prime ideal 
$\fp\cap \TT^P_{\wt{k}}(U^p\Iw^P(p^{b,b}),\oo)$
in $\TT^P_{\wt{k}}(U^p\Iw^P(p^{b,b}),\oo)$.

\begin{defn}\label{def:rep_prime}
	Write 
	$\lambda_\fp\colon \TT^P_{\wt{k}}(U^p\Iw^P(p^{b,b}),E)\to E_\fp$
    for the projection associated 
    to a minimal prime $\fp\subset \TT^P_{\wt{k}}(U^p\Iw^P(p^{b,b}),E)$.
    Let $\pi$ be an irreducible constituent 
    of $S_{\wt{k}}(\bar{\Q}_p)$
    and $r_\pi$ be the corresponding Galois representation.
    We say $\pi$ belongs to $\fp$ if
	$\pi\cap S_{\wt{k}}^{P-\ord}(U^p\Iw^P(p^{b,b}),E)_{\fp}\neq 0$
    and write $r_\fp=r_\pi$.
    Then part (b) in Proposition \ref{prop:gal_ger} implies that
	\begin{equation}\label{eq:Gal_hecke_away_p}
		\mtr(r_\pi(\Fr_w))=\lambda_\fp(T_w^{(1)}),\quad
		\det(r_\pi(\Fr_w))=q_w\lambda_\fp(T_w^{(2)}),\,
	\end{equation}
	for primes $v=w\bw$ in \eqref{def:hecke_away_p}.
    Since the set of such primes has density one,
	by Chebotarev's density theorem
	the representation $r_\pi$ is defined over $E_{\fp}$
	and independent of the choice of $\pi$ belonging to $\fp$.
\end{defn}

Following \cite{ger},
we say $\wt{k}$ is sufficiently regular
at $w\in \Sigma_p$
if there exists  $\sigma\in I_{w}$
such that  $k_{\sigma,1}>k_{\sigma,2}$.
\begin{lem}\label{lem:galois_at_p}
	Let $r_{\fp}$ be the Galois representation
	associated to a minimal prime
	$\fp\subset \TT^P_{\wt{k}}(U^p\Iw^P(p^{0,1}),E)$.
	\begin{enumerate}[label=(\alph*)]
	\item The representation $r_\fp$ is crystalline at
    $w'\in \Sigma_p\setminus\{ w_0\}$
	if $\wt{k}$ is sufficiently regular at $w'$.
	Moreover, $r_\fp\vert_{D_{w'}}$ 
	fits into an exact sequence
	$0\to \psi_1\to r_{\fp}\vert_{D_{w'}} \to \epsilon^{-1}\psi_2\to 0$
	for characters $\psi_i\colon D_{w'}\to E_{\fp}^{\times}$ with
	\begin{equation}\label{eq:Gal_hecke_at_p'}
	\begin{aligned}
		\psi_1\circ \Art_{w'}(\varpi_{w'})&=
		\lambda_{\fp}(U_{\wt{k},w'}^{(1)}) &
		\psi_1\circ \Art_{w'}(x)&=
		\lambda_{\fp}
		(\langle 
		(\begin{smallmatrix}
			x&\\&1
		\end{smallmatrix})
		\rangle)\, \text{ for }x\in \oo_{w'}^{\times}\\
		\psi_2\circ \Art_{w'}(\varpi_{w'})&=
		\lambda_{\fp}(U_{\wt{k},w'}^{(2)})/
		\lambda_{\fp}(U_{\wt{k},w'}^{(1)}) &
		\psi_2\circ \Art_{w'}(x)&=
		\lambda_{\fp}
		(\langle 
		(\begin{smallmatrix}
			1&\\&x
		\end{smallmatrix})
		\rangle)\, \text{ for }x\in \oo_{w'}^{\times}
	\end{aligned}
	\end{equation}
	\item The representation $r_\fp$ is 
	crystalline at $w=w_0$, with 
	\begin{equation}\label{eq:Gal_hecke_at_p}
	(\epsilon\det r_\fp)\circ \Art_w(\varpi_w)=
	\lambda_{\fp}(U_{\wt{k},w}^{(2)}),\quad
	(\epsilon\det r_\fp)\circ \Art_w(x)=
	\lambda_{\fp}
	(\langle 
	(\begin{smallmatrix}
		x&\\&x
	\end{smallmatrix})
	\rangle)\, \text{ for }x\in \oo_{w}^{\times}
	\end{equation}
	\end{enumerate}
\end{lem}
\begin{proof}
The first part is a restatement of \cite[Cor 2.33]{ger}.
Here we have used the fact that
$\lambda_\fp(\langle u\rangle)=(w_0\wt{k})^{-1}(u)$ 
for $u\in T_n(\oo_{w'})\subset \Iw(w'^{0,1})$. 
And the second part follows from part (c)
of Proposition \ref{prop:gal_ger}.
\end{proof}

\begin{rem}
If we consider the ordinary Hecke algebra
$\TT^{\ord}(U^p\Iw(p^{0,1}),E)$ instead,
then \cite[Cor 2.33]{ger} states that 
part (a) of the lemma above
holds for all $w\in\Sigma_p$.
\end{rem}

By Lemma \ref{lem:hecke_red},
when $\fp$ goes through all minimal primes
of $\TT^P_{\wt{k}}(U^p\Iw^P(p^{b,b}),E)$
we have an injection
\[
	\TT^P_{\wt{k}}(U^p\Iw^P(p^{b,b}),\oo)\hookrightarrow
	\TT^P_{\wt{k}}(U^p\Iw^P(p^{b,b}),E)\xrightarrow{\sim} 
	\prod_{\fp}E_{\fp},
\]
in which the image is a closed subring.
Therefor \eqref{eq:Gal_hecke_away_p}
together with Chebotarev's density theorem
implies that there exists
a pseudo-representation
$T=T_{\wt{k}}(U^p\Iw^P(p^{b,b}))$ 
such that $\lambda_\fp\circ T=\mtr(r_\fp)$ 
for all $\fp$ and
\begin{equation}\label{eq:pseudo_rep_finite}
\begin{aligned}
	T&\colon \Gal_{\K}
	\to \TT^P_{\wt{k}}(U^p\Iw^P(p^{b,b}),\oo)\quad
	\Fr_w\mapsto T_w^{(1)}\\\
	(\epsilon\det T)&\colon \Gal_{\K}
	\to \TT^P_{\wt{k}}(U^p\Iw^P(p^{b,b}),\oo)\quad
	\Fr_w\mapsto T_w^{(2)}
\end{aligned}
\end{equation}
for $v=w\bw$ in \eqref{def:hecke_away_p}.
By the formula above the pseudo-representations
are compatible among different levels.
We take the inverse limit and define
the big $P$-ordinary Galois pseudo-representation.
\[
T_{\wt{k}}(U^p)\colon \Gal_{\K} \to \TT^P_{\wt{k}}(U^p,\oo)
\]
which also satisfies \eqref{eq:pseudo_rep_finite}.
We also define 
$T_{\wt{k}}^{\ord}(U^p)\colon 
\Gal_{\K} \to \TT^{\ord}_{\wt{k}}(U^p,\oo)$
in a similar fashion.

\begin{defn}\label{def:big_Gal}
We write $T(U^p)=T_{\wt{k}}(U^p)$ and 
$T^{\ord}(U^p)=T_{\wt{k}}^{\ord}(U^p)$
when $\wt{k}$ is the trivial weight.
Note that since all pseudo-representations satisfy
\eqref{eq:pseudo_rep_finite}
for $v=w\bw$ in \eqref{def:hecke_away_p},
by Chebotarev's density theorem 
the pseudo-representations are comptatible among
the commutative diagram
\[
\begin{tikzcd}
    \TT(U^p,\oo) 
    \arrow[r] \arrow[d]&
    \TT_{\wt{k}}(U^p,\oo) 
    \arrow[d]\\
    \TT^{\ord}(U^p,\oo) 
    \arrow[r] &
    \TT^{\ord}_{\wt{k}}(U^p,\oo) 
\end{tikzcd}
\]
where the horizontal maps come from Proposition \ref{prop:wt_indep}
and the vertical maps come from Lemma \ref{lem:coh_to_ord}.
\end{defn}

\begin{prop}\label{prop:big_char_at_p}
    For each $w'\in \Sigma_p\setminus\{w_0\}$
    the restriction of $T(U^p)$ to $D_w$
    is reducible and 
    \[
        T(U^p)\vert_{D_{w'}}=\Psi_{w',1}+
        \epsilon^{-1}\Psi_{w',2},\quad
        T(U^p)^c\vert_{D_{w'}}=\epsilon^{-1}\Psi_{w',1}^{-1}+
        \Psi_{w',2}^{-1}
    \]
    where $T(U^p)^c(\gamma)\coloneqq T(U^p)(\gamma^c)$,
    for characters 
    $\Psi_{w',1},\Psi_{w',2}\colon D_{w'}\to \TT^P(U^p,\oo)$
    such that
	\begin{equation}
	\begin{aligned}\label{eq:big_char_atw'}
		\Psi_{w',1}\circ \Art_{w'}(\varpi_{w'})&=
		U_{w'}^{(1)}=h_U(\smat{\varpi_{w'}&\\&1}) &
		\Psi_{w',1}\circ \Art_{w'}(x)&=
		\langle 
		(\begin{smallmatrix}
			x&\\&1
		\end{smallmatrix})
		\rangle\, \text{ for }x\in \oo_{w'}^{\times}\\
		\Psi_{w',2}\circ \Art_{w'}(\varpi_{w'})&=
		U_{w'}^{(2)}/
		U_{w'}^{(1)} =h_U(\smat{1&\\&\varpi_{w'}}) &
		\Psi_{w',2}\circ \Art_{w'}(x)&=
		\langle 
		(\begin{smallmatrix}
			1&\\&x
		\end{smallmatrix})
		\rangle\, \text{ for }x\in \oo_{w'}^{\times}
	\end{aligned}
	\end{equation}
    Moreover, $T(U^p)$ is $\Psi_{w',1}$-ordinary 
    in the sense of \cite[Def 5.2.3]{pan}.
    In other word 
    \[
        T(U^p)(\sigma\tau\eta)
        -\Psi_{w',1}(\sigma)T(U^p)(\tau\eta)
        -\Psi_{w',2}(\tau)T(U^p)(\sigma\eta)
        +\Psi_{w',1}(\sigma)\Psi_{w',2}(\tau)T(U^p)(\eta)=0
    \]
    for all $\sigma, \tau\in D_{w'}$ and $\eta\in\Gal_\K$.
    And the statement holds for all $w\in \Sigma_p$
    if we consider $T^{\ord}(U^p)$ instead.
\end{prop}
\begin{proof}

Since the proof is the same for $T(U^p)$ and $T^{\ord}(U^p)$,
we only consider $T(U^p)$.
By Lemma \ref{lem:galois_at_p}
we can define characters $\Psi_{w',i,\wt{k}}$ as above
with values in $\TT^P(U^p\Iw^P(p^{0,1}),E)$
when $\wt{k}$ is sufficiently regular
at $w'\in \Sigma_p\setminus{w_0}$.
The pushforward of $T(U^p)\vert_{D_{w'}}$ to which
is then the sum of $\Psi_{w',1,\wt{k}}$ and
$\epsilon^{-1}\Psi_{w',2,\wt{k}}$ and
moreover $\Psi_{w',1,\wt{k}}$-ordinary 
by \cite[Lem 5.3.3]{pan} and Lemma \ref{lem:galois_at_p}.

Now the same argument in the proof of \cite[Cor 3.4]{ger}
shows that the set of $\fp$ such that there exists 
$\pi$ of sufficiently regular weight that belongs to $\fp$
is a dense set.
This implies that Corollary \ref{cor:density} still holds
if we restricts to those $\wt{k}$
that are sufficiently regular
at all $w'\in \Sigma_p\setminus{w_0}$.
Therefore the image of $\TT^P(U^p,\oo)$
via the injection in the corollary
is a closed subring of $\prod_{\wt{k}}\TT^P(U^p\Iw^P(p^{0,1}),\oo)$
and the direct product of the characters
$\Psi_{w',i,\wt{k}}$ factors through which.
From this we conclude the existence 
of the characters $\Psi_{w',i}$.
Using the density result again,
we have that $T(U^p)\vert_{D_{w'}}$
is the sum of $\Psi_{w',1}$ and $\epsilon^{-1}\Psi_{w',2}$
and $\Psi_{w',1}$-ordinary.
The other equality regarding $T(U^p)^c$ is proved similarly
if we notice
$\mtr(r_\fp^c)\vert_{D_{w'}}=
\mtr(r_\fp^\vee\epsilon^{-1})\vert_{D_{w'}}=
\epsilon\psi_1^{-1}+\psi_2^{-1}$
for characters $\psi_i$ in Lemma \ref{lem:galois_at_p}.

\end{proof}

\section{Results on p-adic local Langlands}

We keep the notations and conditions from last subsection.
Specifically, we have
\begin{itemize}
\item The definite unitary group $G$ over $\F$
defined by \eqref{def:def_unitary} for $n=2$.
\item A finite subset $S\subset\finite$ of prime-to-$p$
places that splits in $\K$
and an open compact subgroup $U^p$
satisfying \eqref{cond:s-ram} and \eqref{cond:small}.
\item A parabolic subgroup $P\subset G_p\cong\prod_{v\in S_p}G(\F_v)$
satisfying \eqref{cond:parabolic}.
\end{itemize}
Throughout the section
we fix a sufficiently large finite extension $E/\Qp$.
Then over the ring of integers $\oo=\oo_E$ 
we can introduce the space of 
algebraic modular forms 
$S^{P-\ord}(U^p,E/\oo), S^{\ord}(U^p,E/\oo)$;
the completed homology
$M(U^p)=M_P(U^p), M^{\ord}(U^p)=M_B(U^p)$; 
the completed cohomology
$S(U^p)=S_P(U^p), S^{\ord}(U^p)=S_B(U^p)$;
and the big Hecke algebras
$\TT(U^p,\oo)=\TT^P(U^p,\oo)$ and $\TT^{\ord}(U^p,\oo)$.

For the degree one place $w_0\in \Sigma_p$,
we will write $D_{w_0}$ and $\Gp$,
the absolute Galois group of  $\Qp$, interchangeably,
where the two Galois groups are identified 
through an fixed isomorphism $\K_{w_0}\cong \Qp$.
Condition \eqref{cond:parabolic}
implies that the Levi part $Q$
of the parabolic subgroup $P$
is the product of $\GL_2(\K_{w_0})\cong \GL_2(\Qp)$
and the torus $\prod_{w'\in\Sigma_p\notin\{w_0\}}T(\K_{w'})$
(we write $T=T_n$ for $n=2$).
In particular the restriction of the $Q$-admissible representation
$S^{P-\ord}(U^p,E/\oo)\coloneqq \Ord_P(S(U^p,E/\oo))$
to $\GL_2(\Qp)$
is a locally admissible representation.

The goal of the section is as follows.
Fix a maximal ideal $\fm\subset \TT(U^p,\oo)$.
By the lemma below,
potentially after replacing $E$ by a finite extension,
we may assume that 
$\TT(U^p,\oo)/\fm$ coincides with the residue field $\fF$ of $\oo$.
Let $T_{\fm}=T(U^P)_{\fm}\colon \Gal_\K\to \TT(U^p)_{\fm}$
denote the localization of 
the big Galois pseudo-representation in 
Definition \ref{def:big_Gal}.
We further assume  $T_{\fm}$ is 
residually reducible and locally generic at $w_0$,
in the sense that there exists characters
$\bar{\delta}_1, \bar{\delta}_2\colon \Gal_{\K}\to \fF^\times$
such that 
\begin{equation}\tag{red.gen}\label{cond:red_gen}
	T_\fm\equiv \bar{\delta}_1+\bar{\delta}_2
	\mod \fm,\quad
    \text{ and }\quad
	\bar{\delta}_1\bar{\delta}_2^{-1} \vert_{D_{w_0}}
	\neq \id,\omega^{\pm}
\end{equation}
where $\omega$ is the Teichmuler character.
We apply P\v{a}sk\={u}nas' theory of blocks
of $\GL_2(\Qp)$-representations
to $M(U^p)_{\fm}$.
A modified version of the method in \cite{urban}
then shows that $\TT(U^p,\oo)_{\fm}$ is Noetherian
and the ``reducible'' part of  $M(U^p)_{\fm}$
is essentially given by $M^{\ord}(U^p)_{\fm}$,
here we take $M^{\ord}(U^p)$
as a $\TT(U^p,\oo)$-module 
through the homomorphism in Lemma \ref{lem:coh_to_ord}
and take the localization.
From these result we derive a fundamental exact sequence
in Corollary \ref{cor:fund} that will be critical
for our construction of the Euler systems in next section.

\begin{lem}\label{lem:big_red}
The big Hecke algebras 
$\TT(U^p,\oo)$ and $\TT^{\ord}(U^p,\oo)$
are reduced and semi-local.
In fact, the natural homomorphism
$\TT(U^p,\oo)\to \TT^P(U^p\Iw^P(p^{0,1}),\oo/\varpi)$ 
induces a bijection between maximal ideals of the two rings
and similarly for $\TT^{\ord}(U^p,\oo)$.
\end{lem}
\begin{proof}
The the Hecke algebras are reduced
follows from Lemma \ref{lem:hecke_red}.
And the last part of the statement
follows from the same argument
in \cite[Prop 3.3.6]{pan}.
\end{proof}

\subsection{Generically reducible deformation}

We first recall from \cite[\S B.1]{pask}
the structure of the universal deformation ring 
$R$ of a $2$-dimensional pseudo-representation 
$\chi_1+\chi_2$,
where $\chi_1,\chi_2\colon \Gp\to \fF^\times$ 
are continuous characters that satisfies the 
the following generic assumption
\begin{equation}\label{cond:generic}\tag{gen}
	\chi_1\chi_2^{-1}\neq \id,\omega^{\pm1}.
\end{equation}

By \textit{loc.cit}, the assumption
implies the existence of non-split extensions
of Galois representations
\begin{equation*}
    0\to \chi_1\to \rho_{12}\to \chi_2\to 0\quad
    0\to \chi_2\to \rho_{21}\to \chi_1\to 0
\end{equation*}
which are unique up to isomorphisms;
and that the universal deformation rings
$R_{ij}$ of the Galois representations $\rho_{ij}$
are formally smooth of relative dimension $5$ over $\oo$.

Let $\tilde{\rho}_{ij}\colon \Gp\to\GL_2(R_{ij})$
denote the universal deformations.
Under suitable choices of bases we may assume
their reductions modulo the maximal ideals give
$\rho_{12}=\smat{\chi_1&*\\&\chi_2}$ and
$\rho_{21}=\smat{\chi_1&\\ *&\chi_2}$.
The trace then induces 
$\theta_{ij}\colon R\cong R_{ij}$
by \cite[Prop B.17]{pask}.
Let $R_{\red}$ be the quotient of $R$
that represents reducible deformations.
Then $R_{\red}$ is isomorphic
to the complete tensor of the
deformation rings for the characters $\chi_1$ and $\chi_2$.
Since the pro-$p$ completion of $\Qp^\times$
has rank 2 over $\Zp$,
we see that $R_{\red}$ is formally smooth of
relative dimension $4$ over $\oo$.
Consequently the reducibility ideal 
$\tau\coloneqq \ker(R\to R_{\red})$ 
is a principal ideal generated by 
an element $\xx\in\fm_R\setminus \fm_R^2$,
where $\fm_R\subset R$ is the maximal ideal.
By \cite[Prop B.23]{pask},
we have $\theta_{ij}(\tau)=\tau_{ij}$
for the ideal $\tau_{ij}\subset R_{ij}$
generated by the $(j,i)$-entry of $ \tilde{\rho}_{ij}(g)$
for all $g\in \Gp$.

Write $x=\xx$ and $\theta=\theta_{12}$.
By \cite[Prop B.24]{pask},
the representation 
$\tilde{\rho}_{12}^x\colon \Gp\to \GL_2(R_{12})$
defined by
\begin{equation*}
	\tilde{\rho}_{12}^x(g)\coloneqq 
	\smat{\theta(x)&\\&1}
	\tilde{\rho}_{12}(g)
	\smat{\theta(x)&\\&1}^{-1}
\end{equation*}
is a deformation of $\rho_{21}$ to $R_{12}$
and the induced map
$\alpha\colon R_{21}\to R_{12}$ is an isomorphism.
Moreover, the following diagram is commutative.
From now on we identify 
$R_{21}$ with $R_{12}$ and
$\tilde{\rho}_{21}$ with 
$\tilde{\rho}_{12}^x$.
\begin{equation*}
	\begin{tikzcd}
		R_{21} \arrow[r,"\alpha"] &
		R_{12}\\
		R\arrow[u,"\theta_{21}"] \arrow[r,equal] &
		R\arrow[u,"\theta_{12}",swap]
	\end{tikzcd}
\end{equation*}

\begin{lem}\cite[Prop B.26]{pask}
\label{lem:B26}
The modules
$\Hom_{\Gp}(\tilde{\rho}_{12}, \tilde{\rho}_{21})$ and
$\Hom_{\Gp}(\tilde{\rho}_{21}, \tilde{\rho}_{12})$
are free of rank one over $R\cong R_{12}$ and
generated respectively by
\begin{equation}\label{eq:Phi_ij}
	\Phi_{12}=\smat{\theta(x)&\\&1} \text{ and }
	\Phi_{21}=\smat{1&\\&\theta(x)}.
\end{equation}
Moreover, the algebra
$\End_{\Gp}(\tilde{\rho}_{12}\oplus \tilde{\rho}_{21})$
is isomorphic to the generalized matrix algebra
$\smat{R& R\Phi_{12}\\ R\Phi_{21}& R}$.
In particular it is free of rank $4$ over $R$ and
its center is isomorphic to $R$.
\end{lem}

\subsection{Generically reducible block}

We briefly recall from \cite{pask} the results on
$p$-adic local Langlands for $\GL_2(\Qp)$ 
that will be relevant.
Following the notations in \textit{loc.cit.},
we temporarily let $G$, $K$ and $B=TN$ respectively
denote $\GL_2(\Qp)$, $\GL_2(\Zp)$,
the subgroup of upper triangular matrices in $\GL_2(\Qp)$
and its Levi decomposition.
We also identify the center $Z\subset T$ with $\Qp^\times$
and write $\bar{B}$ for
the subgroup of lower triangular matrices in $\GL_2(\Qp)$.

Let $\chi_1,\chi_2$ be characters satisfying
\eqref{cond:generic} and fix a continuous character 
$\zeta\colon \Gp\to \oo^\times$
such that $\epsilon\zeta\equiv \chi_1\chi_2$ 
modulo $\varpi$.
For the deformation rings
$R, R_{\red}, R_{12}$ and $R_{21}$
introduced earlier, we let 
$R^{\zeta\epsilon}, R^{\zeta\epsilon}_{\red}, 
R^{\zeta\epsilon}_{12}$ and $R^{\zeta\epsilon}_{21}$
denote the quotients that represent
deformations with the
fixed determinant $\epsilon\zeta$.
All the results in the last subsection
still hold true for these deformation rings,
except that the dimensions are decreased by $2$.

We identify $\chi_i$ with the characters
$\chi_i\circ \Art$ on $\Q^\times$,
where $\Art\colon \Qp^\times\to \Gp^{\textnormal{ab}}$
is the normalized reciprocity map
so that $\Fr\coloneqq \Art(p)$
is a geometric Frobenius.
We define the character
$\chi=\chi_1\otimes\chi_2\omega^{-1}$
of $T\cong \Qp^\times\times\Qp^\times$
and put $\chi^s(a,b)=\chi(b,a)$, so that 
$\chi^s\alpha=\chi_2\otimes \chi_1\omega^{-1}$
for $\alpha\coloneqq \omega\otimes\omega^{-1}$.
By \cite[Thm 30]{barthel},
the inductions
\[
\pi_1\coloneqq \Ind_{B}^G\chi\cong
\Ind_{B}^G\chi_1\otimes\chi_2\omega^{-1}\quad
\pi_2\coloneqq \Ind_{B}^G\chi^s\alpha\cong 
\Ind_{B}^G\chi_2\otimes\chi_1\omega^{-1} 
\]
are irreducible representations
in $\Mod^{\sm}_{G,\zeta}(\oo)$.

\begin{defn}\label{def:block}
Let $\Mod^{\lfin}_{G,\zeta}(\oo), \Mod^{\lfin}_{T,\zeta}(\oo)$
be the subcategories
of representations that are 
locally of finite length 
and with central character $\zeta$;
and let $\fC_G(\oo), \fC_T(\oo)$
be the dual categories
of the Pontryagin duals.
Then $\B\coloneqq\{\pi_1,\pi_2\}$,
for $\pi_1,\pi_2$ as above, is a block 
in the sense of \cite[\S 5]{pask}.
We then let $\Mod^{\lfin}_{G,\zeta}(\oo)^\B$
denote the subcategory
of representations which have the property that
all irreducible subquotients belong to $\B$;
and $\fC_G(\oo)^\B$
be the category of Pontryagin duals.
\end{defn}

Let $\Ord\colon \Mod_{G,\zeta}^{\sm}(\oo)
\to \Mod_{T,\zeta}^{\sm}(\oo)$
denote Emerton's functor of $B$-ordinary parts.
When $V\in \Mod^{\sm}_G(\oo)$ and
$U\in \Mod^{\sm}_T(\oo)$,
the adjunction formula
from \cite[Thm 4.4.6]{emeI} states that
\begin{equation}\label{eq:adjunct}
	\Hom_{\oo[G]}(\Ind_{\bar{B}}^GU,V)
	\xrightarrow{\Ord}
	\Hom_{\oo[T]}(\Ord(\Ind_{\bar{B}}^GU),\Ord V)
	\cong
	\Hom_{\oo[T]}(U,\Ord V)
\end{equation}
is an isomorphism, where the last isomorphism
is induced by $\Ord(\Ind_{\bar{B}}^GU)\cong U$
from  \cite[Prop 4.3.4]{emeI}.

By \cite[Prop 7.1]{pask},
if $\iota\colon \pi_1\hookrightarrow \tilde{J}_1$
is the injective envelope of $\pi_1$
in $\Mod^{\lfin}_{G,\zeta}(\oo)$,
then we have $\Ord\pi_1=\Ord(\Ind_B^G\chi)=\chi^s$
and $\Ord(\iota)\colon \chi^s \to \Ord(\tilde{J}_1)$
is isomorphic to an injective envelope
$\tilde{J}_{\chi^s}$ of $\chi^s$
in $\Mod^{\lfin}_{T,\zeta}(\oo)$.
Furthermore, if we fix an isomorphicm 
$\tilde{J}_{\chi^s}\to \Ord_B(\tilde{J}_1)$, then the morphism 
$\iota_1\colon \Ind_{\bar{B}}^G(\tilde{J}_{\chi^s})\to
\tilde{J}_1$
induced by the adjunction formula \eqref{eq:adjunct} is injective.
To simplify notations,
we identify $\Mod^{\lfin}_{T,\zeta}(\oo)$
with $\Mod^{\lfin}_{\Qp^\times}(\oo)$ through 
the map $\Qp^\times\cong \{\smat{1&\\&*}\}\subset T$
and write $\tilde{J}_{\chi_1}=\tilde{J}_{\chi^s}$.
Then for $ \tilde{J}_{\chi_2}$
and $ \tilde{J}_2$ defined as above
we also have the injective morphism
$\iota_2\colon \Ind_{\bar{B}}^G(\tilde{J}_{\chi_2})\to
\tilde{J}_2$.
We also put
$\tilde{P}_{\chi_i^\vee}\coloneqq \tilde{J}_{\chi_i}^\vee
\in\fC_T(\oo)$ for $i=1,2$.

\begin{lem}\cite[Cor 7.7]{pask}
\label{lem:proj_enve}
Define $\tilde{M}_i\coloneqq 
\Ind_{\bar{B}}^G(\tilde{J}_{\chi_i})^\vee$ and 
$\tilde{P}_i\coloneqq \tilde{J}_i^\vee\in\fC_G(\oo)$.
Then the morphism
$p_i\colon \tilde{P}_i\twoheadrightarrow \tilde{M}_i$
which is dual to 
$\iota_i\colon 
\Ind_{\bar{B}}^G(\tilde{J}_{\chi_i})\hookrightarrow 
\tilde{J}_i$
can be extended to the exact sequences
\begin{equation}\label{eq:exact_PPM}
	0\to \tilde{P}_{2}\xrightarrow{\phi_{12}} 
	\tilde{P}_{1}\xrightarrow{p_1} \tilde{M}_1\to 0 
	\text{ and }
	0\to \tilde{P}_{1}\xrightarrow{\phi_{21}} 
	\tilde{P}_{2}\xrightarrow{p_2} \tilde{M}_2\to 0
\end{equation}
\end{lem}

Let $\Rep_{\Gp}(\oo)$
be the category of compact $\oo$-modules with
continuous actions of $\Gp$,
and $\V\colon \fC_G(\oo)\to \Rep_{\Gp}(\oo)$
be the Colmez functor introduced 
in \cite[\S 5.7]{pask},
which is exact and covariant.
By \cite[Cor 8.7]{pask},
for $(i,j)=(1,2)$ or  $(2,1)$,
there exists unique non-split extensions
in $\Mod^{\sm}_{G,\zeta}(\oo)$ 
\[
	0\to \pi_2\to \kappa_{12}\to \pi_1\to 0,\quad
	0\to \pi_1\to \kappa_{21}\to \pi_2\to 0
\]
such that
$\V(\pi_i^\vee)=\chi_i$, $\V(\kappa_{ij}^\vee)=\rho_{ij}$,
and $\V(\tilde{P}_j)=\tilde{\rho}^{\zeta\epsilon}_{ij}$
are the universal deformations
with determinant $\zeta\varepsilon$.

It then follows from \cite[Lem 8.10]{pask} that 
taking the Colmez functor 
$\V$ induces the isomorphisms below.
\begin{equation}\label{eq:end_deform}
\begin{split}
	\End_{\fC_{G}(\oo)}(\tilde{P_2})\cong 
    R^{\zeta\epsilon}_{12}\cong R^{\zeta\epsilon},\quad
	\Hom_{\fC_G(\oo)}(\tilde{P}_2, \tilde{P}_1)\cong
    R^{\zeta\epsilon}\Phi_{12}\\
	\Hom_{\fC_G(\oo)}(\tilde{P}_1, \tilde{P}_2)\cong
    R^{\zeta\epsilon}\Phi_{21},\quad
	\End_{\fC_{G}(\oo)}(\tilde{P_1})\cong 
    R^{\zeta\epsilon}_{21}\cong R^{\zeta\epsilon}
\end{split}
\end{equation}
Now we define $ \tilde{P}_\B=\tilde{P}_1\oplus \tilde{P}_2$,
so that by the isomorphisms in \eqref{eq:end_deform} the algebra
$\tilde{E}_\B\coloneqq
\End_{\fC_G(\oo)}(\tilde{P}_\B)$
is isomorphic to 
$\End_{\Gp}(\tilde{\rho}^{\zeta\epsilon}_{12}\oplus
\tilde{\rho}^{\zeta\epsilon}_{21})$.

By \cite[Prop 7.1]{pask},
every morphism from $\tilde{P}_i$ to $\tilde{M}_i$
factors through $p_i$ and
$\Hom_{\fC_G(\oo)}(\tilde{P}_i, \tilde{M}_i)=
\End_{\fC_G(\oo)}(\tilde{M}_i)$.
Applying the adjunction formula \eqref{eq:adjunct} to which 
shows that the endomorphism algebra is isomorphic to 
$\End_{\fC_T(\oo)}(\tilde{P}_{\chi_i^\vee})$,
which is isomorphic by \cite[Prop 3.34]{pask} to
a formal power series ring
$ \oo\llbracket x,y\rrbracket$
of two variables.

\begin{lem}\label{lem:ker_red}
	The kernel of the homomorphism
    $p_{i*}\colon R^{\zeta\epsilon}\cong
    \End_{\fC_G(\oo)}(\tilde{P}_i)\twoheadrightarrow
	\End_{\fC_G(\oo)}(\tilde{P}_i, \tilde{M}_i)$
	is the reducibility ideal $\tau\subset R^{\zeta\epsilon}$.
    In other word the morphism $p_i$ and the functor $\Ord$ induces
    the isomorphisms
	\begin{equation}
	R^{\zeta\epsilon}_\red\cong 
	\Hom_{\fC_G(\oo)}(\tilde{P}_i, \tilde{M}_i)\cong
	\End_{\fC_G(\oo)}(\tilde{P}_{\chi_i^\vee})\cong
	\oo\llbracket x,y\rrbracket
	\end{equation}
\end{lem}
\begin{proof}
It suffices to show that 
the image of $\tau$ consists of 
$\phi\in \End_{\fC_G(\oo)}(\tilde{P}_i)$
such that $p_i\circ \phi$ is trivial.
Apply $\Hom(\tilde{P}_j,*)$
to the following exact sequences
from \eqref{eq:exact_PPM}.
\[
\begin{tikzcd}
	0 \arrow[r]&
	\tilde{P}_1  \arrow[r,"\phi_{21}"]  &
	\tilde{P}_2 \arrow[r,"p_2"] \arrow[dr,equal] &
	\tilde{M}_2  \arrow[r] & 0 \\
	0 & 
	\tilde{M}_1 \arrow[l]&
	\tilde{P}_1 \arrow[l,"p_1",swap]  &
	\tilde{P}_2  \arrow[l,"\phi_{12}",swap]  & 
	0  \arrow[l] 
\end{tikzcd}
\]
The adjunction formula implies that
$\Hom_{\fC_G(\oo)}(\tilde{P}_j,\tilde{M}_i)\cong
\Hom_{\fC_T(\oo)}
(\tilde{P}_{\chi_j^\vee},\tilde{P}_{\chi_i^\vee})$,
which are zeros for $i\neq j$
since blocks in $\fC_T(\oo)$ are singletons
by the discussion in \cite[\S 7.2]{pask}.
Consequently there are isomorphisms
\[
	\phi_{12*}\colon
	\End_{\fC_G(\oo)}(\tilde{P}_2)\cong
	\Hom_{\fC_G(\oo)}(\tilde{P}_2, \tilde{P}_1)\quad
	\phi_{21*}\colon
	\End_{\fC_G(\oo)}(\tilde{P}_1)\cong
	\Hom_{\fC_G(\oo)}(\tilde{P}_1, \tilde{P}_2)
\]
Therefore $\Hom_{\fC_G(\oo)}(\tilde{P}_i, \tilde{P}_j)$
is a free module over $\End_{\fC_G(\oo)}(\tilde{P}_i)$
of rank one generated by $\phi_{ji}$ when $i\neq j$.
Compare with the isomorphisms in \eqref{eq:end_deform},
we may assume that $\V(\phi_{ij})$ agree with 
$\Phi_{ij}$ in \eqref{eq:Phi_ij}.
Thus the image of 
$\phi_{ij}\circ\phi_{ji}\in \End_{\fC_G(\oo)}(\tilde{P}_i)$,
which belongs to the kernel in question,
corresponds to the generator 
$\xx=\Phi_{ij}\circ\Phi_{ji}$ of 
$\tau\subset R^{\zeta\epsilon}$.
This implies that the kernel in question
contains $\tau$, and they are actually equal
since $R^{\zeta\epsilon}$ is formally smooth of relative dimension  $3$.
\end{proof}

\subsubsection{Univeral unitary completion}

Let $\Ban(E)$
denote the category of admissible $E$-Banach space
representations of $G$ with central character $\zeta$
as defined in \cite{pask}.
Identify the center of $\fC_G(\oo)^\B$ with $R^{\zeta\epsilon}$,
then $R^{\zeta\epsilon}[\frac{1}{p}]$ acts on objects of $\Ban(E)^{\B}$.
If $\fn\subset R^{\zeta\epsilon}[\frac{1}{p}]$ 
is a maximal ideal,
let $\Irr(\fn)$ denote the set of
irreducible representations in  $\Ban(E)^{\B}$
on which the action of $R^{\zeta\epsilon}[\frac{1}{p}]$ 
factors through $\fn$.

On the other hand,
let $T\colon \Gp\to R^{\zeta\epsilon}$ 
and $T_\fn\colon \Gp\to R^{\zeta\epsilon}[\frac{1}{p}]/\fn$ 
denote the universal deformation and 
the reduction.
Enlarge $E$ if necessary,
we assume $R^{\zeta\epsilon}[\frac{1}{p}]/\fn\cong E$.
Then there exists 
the unique semi-simple Galois representation
$r_{\fn}\colon \Gp\to \GL_2(E)$
such that $\mtr r_{\fn}=T_{\fn}$.
When $r_\fn$ is potentially semi-stable,
we let $\Delta_{\fn}$ denote
the associated Weil-Deligne representation 
on $D_{\pst}(r_\fn)$
and $\pi_{\fn}$ denote the smooth irreducible
$\GL_2(\Qp)$-representation attached to $\Delta_{\fn}$
by local Langlands correspondence.
We normalize the correspondence so that 
when $r_{\fn}$ is crystalline, 
thus $\Delta_{\fn}$
is the sum of unramified characters $\mu_1$ and $\mu_2$,
the representation $\pi_\fn$ 
is the smooth un-normalized induction
$\Ind_B^G(\mu_1\otimes\mu_2|\cdot|^{-1})_{\sm}$.
This is the Hecke correspondence used in \cite{pan}.

\begin{rem}
	The representation 
	$\Ind_B^G(\mu_1\otimes\mu_2|\cdot|^{-1})_{\sm}$
	is irreducible since otherwise
	$T_\fn$ would be of the form
	$\eta+\eta\epsilon$ for some character  $\eta$,
    which contradicts 
    the assumption \eqref{cond:generic}.
\end{rem}

\begin{lem}\label{lem:uni_completion}
	Suppose $r_\fn$ is crystalline of
	Hodge-Tate weights $\{-l,-l-k\}$
    up to twisting by a finite character.
    Let $\pi_{-l,1-l-k}^*\cong W_{l,k}\coloneqq
    \Sym^{k-1}\otimes\det^l$, which is the contragredient of
    the algebraic $\GL_2$-representation
    with highest weight $(-l,1-l-k)$.
	Then the universal unitary completion of 
	$\pi_\fn\otimes \pi_{-l,1-l-k}^*(E)$ belongs to 
	$\Irr(\fn)$.
\end{lem}
\begin{proof}
    By assumptions $r_{\fn}$ is potentially semi-stable and
    $\pi_{\fn}$ is a smooth irreducible principal series.
    When $r_\fn$ is irreducible, $\Irr(\fn)$
    has only one object by \cite[Cor 8.14]{pask},
    which we denote by $\Pi_{\fn}$.
    On the other hand,
    the universal unitary completion $\Pi$ in question
    is a non-ordinary admissible absolutely 
    irreducible $E$-Banach space representation
    by \cite[Thm 12.3]{pask},
    and the $(\varphi,N)$-module
    associated to which
    coincides with $D_{\pst}(r_\fn)$
    by \cite[Thm. 1.3]{CDP}.
    Since the same is true for $\Pi_{\fn}$
    by definition,
    it follows from the same theorem
    that $\Pi=\Pi_{\fn}\in \Irr(\fn)$.

	When  $r_\fn$ is reducible
	and $T_\fn=\psi_1+\psi_2$
	for characters
	$\psi_i\colon \Gp\to E^\times$,
	say of Hodge-Tate weights
	$-l-k$ and  $-l$ respectively,
	then $\pi_\fn=\Ind_B^G(\mu_1\otimes\mu_2|\cdot|^{-1})$
	for the charaters
	\[
	\mu_1=\psi_1(\epsilon|\cdot|^{-1})^{-l-k},\quad
	\mu_2=\psi_2(\epsilon|\cdot|^{-1})^{-l}	
	\]
	By \cite[Thm 12.3]{pask} the universal unitary completion
	of $\pi_\fn\otimes W_{l,k}$
	is $\Ind_B^G(\psi)_{cont}$ for
	\[
		\psi( (\begin{smallmatrix}
			a&b\\&d
		\end{smallmatrix}))
		=\mu_2(a)a^l\mu_1(d)|d|^{-1}d^{l+k-1}
		=\psi_2(a)\psi_1\epsilon^{-1}(d)
	\]
	since $\val_p(\mu_2(p))=-l$ and $\varepsilon(a)=a|a|$.
    And the lemma follows since
	by \cite[Cor 8.15]{pask}
	\[
	\Irr(\fn)=\{(\Ind_B^G\psi_1\otimes\psi_2\varepsilon^{-1})_{cont},
	(\Ind_B^G\psi_2\otimes\psi_1\varepsilon^{-1})_{cont}\}.
	\]
\end{proof}

\subsection{Local-global compatibility}
\label{sub:compatible}

We now resume to previous notations and settings.
Recall that $\fm\subset \TT(U^p,\oo)$ 
is a maximal ideal such that
$T_\fm\bmod\fm =\bar{\delta}_1+\bar{\delta}_2$
for characters $\bar{\delta}_i\colon \Gal_\K\to \fF^\times$
satisfying \eqref{cond:red_gen}.
For the same reaseon in \cite[Rem 3.5.3]{pan},
we normalize $T_\fm$ by the cyclotomic character so that
\[
    \epsilon T_\fm=\omega\delta_1+\omega\delta_2\mod\fm.
\]
Therefore the pair fo characters
$\chi_i\coloneqq \omega\bar{\delta}_i\vert_{\Gp}$
satisfies \eqref{cond:generic} because of 
\eqref{cond:red_gen}.
Let $T\colon \Gp\to R$
denote the universal deformation 
of $\chi_1+\chi_2$
and $\B=\{\pi_1,\pi_2\}$
be the associated block of irreducible of $\GL_2(\Qp)$-representations.
By the universal property
there exists homomorphisms of $R$
to the Noetherian finite level
Hecke operators $\TT(U^p\Iw^P(p^{b,b}),\oo)_\fm$
that push the $T$ to 
$\epsilon\cdot T(U^p\Iw^P(p^{b,b}))_\fm\vert_{\Gp}$.
Then the inverse limit of which
gives the homomorphism
\begin{equation}\label{eq:RtoT}
    R\to \TT(U^p,\oo)_{\fm}=
    \varprojlim_b\TT^P(U^p,\Iw^P(p^{b,b}),\oo)_{\fm}
\end{equation}
that pushes $T$ to $\epsilon T_\fm\vert_{\Gp}$.
Note that the homomorphism has to be defined this way
since we do not know a priori whether
$\TT(U^p,\oo)_{\fm}$ is Noetherian.

In order to apply the results recalled earlier,
we need to follow the approach in \cite[\S 5]{urban}
and twist the $\GL_2(\Qp)$-action on 
$\Ord_P(S(U^p,E/\oo))_\fm$.
To this end we fix a crystalline character
$\zeta\colon \Gp\to \oo^\times$
such that $\epsilon\zeta\equiv\chi_1\chi_2$
and let 
$T^{\zeta\epsilon}\colon \Gp\to R^{\zeta\epsilon}$
be the universal deformation
with determinant $\zeta\epsilon$.

Let $\fm_R\subset R$ be the maximal ideal.
Since $\det T\equiv \epsilon\zeta\bmod \fm_R$ and $p$ is odd,
there exists a character
\begin{equation}\label{eq:root_char}
	\xi^{1/2}
    \colon \Gp\to 1+
    \fm_{R},
    \, \text{ such that }
	\xi\coloneqq(\xi^{1/2})^2=\epsilon\zeta(\det T)^{-1}.
\end{equation}
Consequently there exists a homomorphism 
$R^{\epsilon\zeta}\to R$ that pushes 
$T^{\epsilon\zeta}$ to $T'\coloneqq \xi^{1/2}T$.
We write $\xi_\fm^{1/2}\colon \Gp\to \TT(U^p,\oo)_\fm$
for the composition of $\xi^{1/2}$ with \eqref{eq:RtoT}.
Then as $T$ is sent to $\epsilon T_\fm$ by \eqref{eq:RtoT} we have
\begin{equation}\label{eq:root_charm}
	\xi_\fm\coloneqq (\xi_\fm^{1/2})^2=
	\epsilon\zeta(\det \epsilon T_\fm)^{-1}\vert_{\Gp}
	=\zeta(\epsilon\det T_\fm)^{-1}\vert_{\Gp}.
\end{equation}

\begin{defn}\label{def:twist}

Let $\Qp^\times$ be the central torus in $\GL_2(\Qp)$, and 
\begin{align*}
\text{ for }
Q&=\GL_2(\Qp)\times \prod_{w'\in\Sigma_p\setminus{w_0}}T(\K_{w'}) & 
\text{define }\,
Q'&=\GL_2(\Qp)\times Z_Q=\GL_2(\Qp)\times 
\big(\Q^\times\prod_{w'\in\Sigma_p\setminus{w_0}}T(\K_{w'})\big),\\
T&=\GL_2(\Qp)\times \prod_{w'\in\Sigma_p\setminus{w_0}}T(\K_{w'}) & 
T'&=T(\Qp)\times Z_Q=T(\Qp)\times 
\big(\Q^\times\prod_{w'\in\Sigma_p\setminus{w_0}}T(\K_{w'})\big),\\
\end{align*}
We define $\Ord_P(S(U^p,E/\oo))_{\fm}'$
to be the $Q'$-representation on
$\Ord_P(S(U^p,E/\oo))_{\fm}$,
where the action of $Z_Q$ is as usual
but the $\GL_2(\Qp)$-action is twisted by 
$(\xi_{\fm}^{1/2}\circ\Art)$.
The $T'$-representation
$S^{\ord}(U^p,\oo)_\fm'$
is then defined by a similar twist on $T(\Qp)$.
We also write $M(U^p)'_\fm$ and $M^{\ord}(U^p)'_\fm$
for the Pontryagin duals.
Note that we did not twist the Hecke actions, 
so $\TT(U^p,\oo)_\fm$ and $\TT^{\ord}(U^p,\oo)_\fm$
still acts on $\Ord_P(S(U^p,E/\oo))_\fm'$
and $S^\ord(U^p,E/\oo)_\fm'$ as usual.
\end{defn}

\begin{prop}\label{lem:twist}
The twist $\Ord_P(S(U^p,E/\oo))_\fm'$
has the following properties.
\begin{enumerate}
\item There exists a homomorphism 
$R^{\zeta\epsilon}\to \TT(U^p,\oo)_\fm$
that pushes $T^{\zeta\epsilon}$ to 
$T_\fm'\coloneqq \xi_{\fm}^{1/2}\epsilon T_\fm\vert_{\Gp}$.
\item For $w'\in \Sigma_p\setminus{w_0}$,
the action of $T(\K_w)=\K_w^2$ on
$\Ord_P(S(U^p,E/\oo))_\fm=\Ord_P(S(U^p,E/\oo))_\fm'$
coincides with the action induced by
the characters $\Psi_{w',1}\otimes\Psi_{w',2}$ from
Proposition \ref{prop:big_char_at_p}.
\item For $w=w_0$,
the action of $\Qp^\times$ on
$\Ord_P(S(U^p,E/\oo))_\fm=\Ord_P(S(U^p,E/\oo))_\fm'$
coincides with the action induced by
$(\epsilon\det T_\fm)\vert_{\Gp}\circ \Art$.
\item
The twisted $\GL_2(\Qp)$-action on $\Ord_P(S(U^p,E/\oo))_\fm'$.
has central character $\zeta$.
\item 
As a $Q'$-representation
$\Ord_P(S(U^p,E/\oo))_\fm'$ is admissible.
\end{enumerate}
\end{prop}

\begin{proof}

The homomorphism in the first statement 
is the composition of \eqref{eq:RtoT},
which pushes $T$ to $\epsilon T_\fm\vert_{\Gp}$, and
the homomorphism $R^{\epsilon\zeta}\to R$,
that pushes $T^{\epsilon\zeta}$ to $\xi^{1/2}T$.
It is then clear that it pushes $T^{\epsilon\zeta}$ to
$\xi_\fm^{1/2}\epsilon T_\fm\vert_{\Gp}$.

For $w'\in\Sigma_p\setminus w_0$, by definition
$T(\K_w)$ acts by $h_U$ on $\Ord_P(S(U^p,E/\oo))_\fm$.
Therefore the second statement follows from
\eqref{eq:big_char_atw'}.
For the third one, the same argument in the proof of
Proposition \ref{prop:big_char_at_p} shows that 
\[
    (\epsilon \det T_\fm)\vert_{\Gp}\circ \Art(p)=
    U_{p}^{(2)}\coloneqq h_U(\smat{p&\\&p})\quad
    (\epsilon \det T_\fm)\vert_{\Gp}\circ \Art(x)=
    \langle \smat{x&\\&x}\rangle
    \coloneqq h_U(\smat{x&\\&x})\text{ for }x\in\Zp^\times,
\]
which coincides with the central action of $\Qp^\times$ on
$\Ord_P(S(U^p,E/\oo))_\fm=\Ord_P(S(U^p,E/\oo))_\fm'$.

The fourth statement follows from the first one and 
\eqref{eq:root_charm}.
For the last one, we would like to show that 
for any compact open subgroup $U'\subset Q'$
the submodule $(\Ord_P(S(U^p,E/\oo))'_\fm)^{U'}[\varpi^m]$ is finite.
Shrinking $U'$ if necessary, we may assume that 
$\det(U'\cap \GL_2(\Qp))$ is pro-$p$ and contained in
$U'\cap \Qp^\times$, and that
$(\zeta\circ\det)\bmod\varpi^m$ is trivial on 
$\det(U'\cap \GL_2(\Qp))$.
But then by \eqref{eq:root_charm} and the third statement
$\xi_\fm$ is trivial on $\det(U'\cap \GL_2(\Qp))$,
thus so is $\xi_\fm^{1/2}$ since $p$ is odd.
Consequently we have 
\[
    (\Ord_P(S(U^p,E/\oo))'_\fm)^{U'}[\varpi^m]=
    (\Ord_P(S(U^p,E/\oo))_\fm)^{U'}[\varpi^m],
\]
and the latter is finite since the untwisted $Q'$-representation
$\Ord_P(S(U^p,E/\oo))_\fm$ is admissible.

\end{proof}

\begin{prop}\label{prop:twist_ord}
The twist $S^{\ord}(U^p,E/\oo)_\fm'$
has the following properties.
\begin{enumerate}
\item The action of $Z_Q$ on
$S^{\ord}(U^p,E/\oo)_\fm=S^{\ord}(U^p,E/\oo)_\fm'$
is as described in last proposition.

\item For $w=w_0$,
the action of $T(\Qp)=\Qp^2$ on $S^{\ord}(U^p,E/\oo)_\fm'$
coincides with the action induced by the character
$\xi_\fm^{1/2}\Psi_{w,1}\otimes \xi_\fm^{1/2}\Psi_{w,2}$
from Proposition \ref{prop:big_char_at_p}
and has central character $\zeta$.

\item 
Let $\Ord\colon \Mod^{\sm}_{\GL_2(\Qp)}(\oo)
\to \Mod^{\sm}_{T(\Qp)}(\oo)$ denote the Emerton's
functor of $B(\Qp)$-ordinary parts.
Then $\Ord(\Ord_P(S(U^p,E/\oo))_\fm')=S^{\ord}(U^p,E/\oo)_\fm'$.

\end{enumerate}
\end{prop}
\begin{proof}
The above statements either follows directly from 
the definition or can be proved by the same way
as the previous proposition.
\end{proof}

Let $\fC(\oo)$ denote the category of Pontryagin duals of objects in
$\Mod^{\lfin}_{\GL_2(\Qp),\zeta}(\oo)$.
By the proposition above we see that 
$M(U^p)_{\fm}'$ is an object in $\fC(\oo)$
as a $\GL_2(\Qp)$-representation.
In particular it admits an action of the center
$R^{\zeta\epsilon}\subset\tilde{E}_\B$.

\begin{prop}\label{prop:compatibility}
    The restriction of  $M(U^p)_{\fm}'$ to $\GL_2(\Qp)$
    belongs to $\fC(\oo)^{\B}$
	and the $R^{\epsilon\zeta}$-action on which
    as the center of the category
	factors through the homomorphism to $\TT(U^p,\oo)_\fm$
    given in Proposition \ref{lem:twist}.
\end{prop}

\begin{proof}

Let $\pi$ be an irreducible constituent of
the $G(\A_f)$-representation $S_{\wt{k}}(\bar{\Q}_p)$
that belongs to a minimal prime
$\fp\subset \TT_{\wt{k}}^P(U^p\Iw^P(p^{0,1}),E)_{\fm}$.
Then $\pi_{\fp}\coloneqq\pi_{w_0}$
is an unramified $\GL_2(\Qp)$-representation
and corresponds to the crystalline $\Gp$-representation
$r_\fp\vert_{\Gp}$ by Lemma \ref{lem:galois_at_p}.
The beginning of the proof of
Proposition \ref{prop:wt_space} shows that 
$S_{\wt{k}}^{P-\ord}(U^p,\oo)\cong
\Hom_{\oo}(\pi_{\wt{k}}^*(\oo), S(U^p))$,
which give a homomorphism
\begin{equation}\label{eq:hom_wt}
    \pi_\fp\otimes \pi_{\wt{k}}^*(E)\to S(U^p)_\fm\otimes_{\oo}E
\end{equation}
between $\GL_2(\Qp)$-representations
after tensoring with $E$ since
$\pi_\fp\subset S^{P-\ord}_{\wt{k}}(U^p,E)$.

Let $\fn\subset R^{\zeta\epsilon}[\frac{1}{p}]$ be the kernel of 
$R^{\zeta\epsilon}[\frac{1}{p}]\to \TT(U^p,E)_{\fm}\to E_\fp$,
the composition of the homomorphism in Proposition \ref{lem:twist}
and $\lambda_\fp$.
Combine the strategy of the proof of \cite[Thm 3.5.5]{pan}
with Corollary \ref{cor:density},
it suffices to show that
the universal unitary completions of
the twist of $\pi_{\fp}\otimes\pi_{\wt{k}}^*(E)$
by $\xi_\fm^{1/2}$ belongs to $\Irr(\fn)$.

Since $w_0\in \Sigma_p$ is of degree one
there is a unique embedding $\sigma$ in $I_{w_0}$.
Write $k_1=k_{\sigma,1}$ and $k_2=k_{\sigma,2}$,
then $r_\fp\vert_{\Gp}$ is crystalline
of Hodge-Tate weights  $k_1+1,k_2$
and $\epsilon\det r_{\fp}$
is of Hodge-Tate weight $k_1+k_2$.
Define $\xi_\fp^{1/2}=\lambda_\fp\circ \xi_\fm^{1/2}$
so that $\xi_\fp\coloneqq(\xi_\fp^{1/2})^2=
\zeta(\epsilon\det r_\fp)^{-1}$.
Suppose $\zeta$ has Hodge-Tate weight $w_0$,
then the Hodge-Tate weight $w_0-(k_1+k_2)$ of
$\xi_\fp$ is even as $\xi_\fp$
is congruence to the trivial character.
Therefore the character $\xi_{\fp}^{1/2}$
is crystalline of Hodge-Tate weight 
$w_{\fp}\coloneqq \frac{1}{2}(w_0-k_1-k_2)$ 
up to a quadratic twist.

Now the representation
$r_{\fp}\vert_{\Gp}(\xi_{\fp}^{1/2}\epsilon)$
is crystalline up to a quadratic twist
and has Hodge-Tate weights  $\{k_1',k_2'-1\}$ 
for $k_1'=k_1+w_\fp, k_2'=k_2+w_\fp$.
And $\pi_\fp(\xi_\fp^{1/2}(\epsilon|\cdot|^{-1})^{w_\fp})$
is the smooth irreducible
representation associated to 
$D_{\pst}(r_{\fp}(\xi_\fp^{1/2}\epsilon))$
since $\pi_\fp$ is associated to 
$D_{\pst}(r_{\fp}(\epsilon))$.
Now twist \eqref{eq:hom_wt} by $\xi_\fm^{1/2}$ gives
\[
    (\pi_{\fp}\otimes \pi_{\wt{k}}^*(E))(\xi_\fp^{1/2})=
    \pi_\fp(\xi_\fp^{1/2}(\epsilon|\cdot|^{-1})^{w_\fp})
    \otimes \pi_{\wt{k}}^*(E)(\epsilon|\cdot|^{-1})^{-w_{\fp}} \to 
    S(U^p)_\fm'\otimes_{\oo}E
\]
Since  $T_\fn=\mtr(r_{\fn}\vert_{\Gp}(\xi_\fm^{1/2}\epsilon))$
and $\pi_{\wt{k}}^*(E)(\epsilon|\cdot|^{-1})^{-w_{\fp}}\cong 
\pi_{k_1',k_2'}^*(E)$,
it follows from Lemma \ref{lem:uni_completion}
that the universal completion of 
$(\pi_{\fp}\otimes \pi_{\wt{k}}^*(E))(\xi_\fp^{1/2})$
does belong to $\Irr(\fn)$.
\end{proof}

The proposition implies that
all the irreducible $\GL_2(\Qp)$-subquotients
of $\Ord_P(S(U^p,E/\oo))_\fm'$ are 
isomorphic to either $\pi_1$ or $\pi_2$.
On the other hand, the action of $Z_Q$
on $\Ord_P(S(U^p,E/\oo))_\fm'$ coincides 
with the character 
$Z_Q\to \TT(U^p,\oo)_{\fm}^{\times}$
described in Proposition \ref{lem:twist}.
Let $\upsilon\colon Z_Q\to \TT^P(U^p,\oo)/\fm=\fF^\times$
denote the residual character.
By abuse of notation,
we let $\pi_1,\pi_2$ denote
the $Q'$-representations on which $\GL_2(\Qp)$ acts as usual
and $Z_Q$ acts by $\upsilon$.
Then all the irreducible $Q'$-subquotients
of $\Ord_P(S(U^p,E/\oo))_\fm'$ are 
isomorphic to either $\pi_1$ or $\pi_2$.
We similarly let $\chi^s$ and $\chi\alpha^s$ denote
the $T'$-representations defined in the same way,
with $Z_Q$ acting by $\upsilon$.

\begin{defn}

We let $\Mod^{\lfin}_{Q',\zeta}(\oo)$
and $\Mod^{\lfin}_{T',\zeta}(\oo)$
denote the categories of smooth representations of 
$Q'$ and $T'$ that are locally of finite type
whose restriction to $\GL_2(\Qp)$ or $T(\Qp)$
has central character $\zeta$.
We also let $\fC_{Q'}(\oo), \fC_{T'}(\oo)$
denote the category of Pontryagin duals.
When the context is clear,
we will simply write $\Hom_{Q'}$ and $\Hom_{T'}$
for the space of homomorphisms.
\end{defn}

\begin{lem}\label{lem:end_comp}
Let $\tilde{P}$ be the projective envelope of 
$\upsilon^\vee\in \fC_{Z_Q}(\oo)$, the category
of Pontryagin duals of $\Mod_{Z_Q}^{\lfin}(\oo)$,
and $\tilde{E}\coloneqq\End_{\fC_{Z_Q}(\oo)}(\tilde{P})$.
Then $\tilde{P}$ is isomorphic to 
the universal deformation of $\upsilon^{-1}$
and free of rank one over $\tilde{E}$.
More precisely,
let $Z_Q^\wedge$ denote the pro-$p$ completion of $Z_Q$, then
\[
    \tilde{P}=\tilde{E}
    =\oo\llbracket Z_Q^\wedge\rrbracket
\]
where $Z_Q$ acts on $\tilde{P}$
by the product of the tautological character
and a lift of $\upsilon$.
\end{lem}
\begin{proof}
The same argument in \cite[Prop. 3.34]{pask} 
shows that $\tilde{P}$
is free of rank one over $\tilde{E}$
and we may assume $\upsilon$ is the trivial character.
The description of $\tilde{P}$ then follows from 
\cite[Cor. 3.27]{pask} and the proof of
\cite[Lem. 3.32]{pask}.
\end{proof}

\begin{prop}\label{prop:envelope}
    Write $\tilde{P}_{*,\fm}=
    \tilde{P}_*\hat{\otimes}_{\oo}\tilde{P}$
    for $\tilde{P}_*\in\{\tilde{P}_\B, \tilde{P}_1, \tilde{P}_2,
    \tilde{P}_{\chi_1^\vee}, \tilde{P}_{\chi_2^\vee}\}$
    introduced in Lemma \ref{lem:proj_enve}
	and write 
    $\tilde{E}_{\fm}=\tilde{E}_\B
	\hat{\otimes}_{\oo}\tilde{E},
    R_{\fm}=R^{\epsilon\zeta}
	\hat{\otimes}_{\oo}\tilde{E},
    R_{\fm}^\red=R^{\epsilon\zeta}_\red
	\hat{\otimes}_{\oo}\tilde{E},$.
	Then $\tilde{P}_{i,\fm}\in \fC_{Q'}(\oo)$
	is the projective envelope of $\pi_i^\vee$ and 
	$\tilde{P}_{\chi_1^\vee,\fm}, \tilde{P}_{\chi_1^\vee,\fm}
    \in \fC_{T'}(\oo)$ are the projective envelopes
	of $(\chi^s)^\vee$ and
    $(\chi\alpha^s)^\vee$ respectively.
    Moreover
    \begin{align}\label{eq:proj_0}
    &\Hom_{Q'}(\tilde{P}_{\B,\fm}, \tilde{P}_{\B,\fm}) \cong 
    \Hom_{\GL_2(\Qp)}(\tilde{P}_\B, \tilde{P}_\B)
    \hat{\otimes}_\oo \End_{\fC_{\Z_Q}(\oo)}(\tilde{P})\cong 
    \tilde{E}_\B \hat{\otimes}_\oo \tilde{E}
    \eqqcolon \tilde{E}_\fm\\\label{eq:proj_1}
    &\Hom_{Q'}(\tilde{P}_{i,\fm}, \tilde{P}_{j,\fm}) \cong 
    \Hom_{\GL_2(\Qp)}(\tilde{P}_{i}, \tilde{P}_{j})
    \hat{\otimes}_\oo \End_{\fC_{\Z_Q}(\oo)}(\tilde{P})\cong 
    R^{\zeta\epsilon} \hat{\otimes}_\oo \tilde{E}
    \eqqcolon R_\fm\\\label{eq:proj_2}
    &\End_{T'}(\tilde{P}_{\chi_i^\vee,\fm}) \cong 
    \End_{(\Qp^\times)^2}(\tilde{P}_{\chi_i^\vee})
    \hat{\otimes}_\oo \End_{\fC_{\Z_Q}(\oo)}(\tilde{P})\cong 
    R^{\zeta\epsilon}_\red \hat{\otimes}_\oo \tilde{E}
    \eqqcolon R_\fm^{\red}
    \end{align}
\end{prop}
\begin{proof}
    We first observe that any $Q'$-admissible vector 
	is $Z_Q$-finite by \cite[Lem 2.3.5]{emeI}
    and thus generates an admissible $Q'$-representation
    that is finitely-generated over $\GL_2(\Qp)$,
	which is also of finite length by \cite[Thm 2.3.8]{emeI}.
	In other word,
	any locally admissible $Q'$-representation 
	is also locally of finite length.
    That $\tilde{P}_{*,\fm}$ gives
    the corresponding projective envelops now follows from 
    \cite[Lem B.6]{GN}
	and the argument in \cite[Lem B.8]{GN}.
    And the isomorphisms 
    \eqref{eq:proj_1} and
    \eqref{eq:proj_2} follows respectively from  
    \eqref{eq:end_deform} and 
    Lemma \eqref{lem:ker_red}.
\end{proof}

\begin{cor}
Let $\Mod^{\lfin}_{Q',\zeta}(\oo)^{\B}$
denote the subcategory of $Q'$-representations
with all irreducible subquotients isomorphic 
to either $\pi_1$ or $\pi_2$
and $\fC_{\Q'}(\oo)^\B$ be the corresponding 
subcategory of $\fC_{\Q'}(\oo)$. Then
\begin{equation}\label{eq:anti_equiv}
        M  \mapsto 
        \mathbf{m}(M)\coloneqq 
        \Hom_{Q'}(\tilde{P}_{\B,\fm}, M)
\end{equation}
gives an equivalence between $\fC_{Q'}(\oo)^\B$
and the category of compact modules over 
$\tilde{E}_{\fm}$.
\end{cor}
\begin{proof}
This follows from the general formalism of abelian categories.
See also the proof of \cite[Prop.5.45]{pask}.
\end{proof}

\begin{cor}\label{cor:Hecke_Noetherian}
    The $\TT(U^p,\oo)_{\fm}$-module
	$\mathbf{m}\coloneqq \mathbf{m}(M(U^p)_{\fm}')$
	is finite faithful and $\TT(U^p,\oo)_{\fm}$
    is Noetherian.
\end{cor}
\begin{proof}
    Since $\TT(U^p,\oo)_\fm$ acts faithfully on 
    $M(U^p)_\fm'\in \fC_{Q'}(\oo)^\B$,
    it also acts faithfully on 
	$\mathbf{m}$
    by the previous corollary.
    Moreover, $\mathbf{m}$
    is finitely-generated  over $\tilde{E}_{\fm}$ 
    by \cite[Prop 4.17]{pask} since 
    $\Ord_P(S(U^p,E/\oo))'_{\fm}$ is $Q'$-admissible.
    Consequently it is also finite over  $R_{\fm}$
    by Lemma \ref{lem:B26}.
    It then follows from Proposition
    \ref{prop:compatibility}
    that  $\mathbf{m}(M(U^p)_{\fm}')$ is finite over 
    $\TT(U^p,\oo)_\fm$.

    We now have an injective $R_\fm$-algebra homomorphism 
	$\TT(U^p,\oo)_{\fm}\to\End_{R_{\fm}}(\mathbf{m})$.
    Since $\mathbf{m}$ is finite over $R_\fm$,
    so is $\End_{R_\fm}(\mathbf{m})$.
    Therefore $\TT(U^p,\oo)_\fm$
	is finite over $R_{\fm}$ and thus Noetherian.
\end{proof}

\subsection{Reducible part of completed homology}

We fix a generator $\xx$ of the reducibility ideal
$\tau=\ker(R^{\epsilon\zeta}\to R^{\epsilon\zeta}_{\red})$,
so $\xx$ also generates $\ker(R_\fm\to R_\fm^{\red})$ over $R_\fm$.
\begin{defn}
For any $R_{\fm}$-module $M$
we call $M^{\red}\coloneqq M/\xx M$
the reducible part of $M$.
\end{defn}
In this subsection we follow \cite{urban}
and relate $(M(U^p)_\fm')^{\red}$
with $S^{\ord}(U^p,E/\oo)_\fm'$.

\begin{lem}
	For some non-negative integer $r$
	there exists a projective resolution 
	\begin{equation}\label{eq:resolution}
	0\to \tilde{P}_{\B,\fm}^{\oplus r}\to 
	\tilde{P}_{\B,\fm}^{\oplus r}\to 
	M(U^p)_{\fm}'\to 0
	\end{equation}
\end{lem}
\begin{proof}
Identify $\Qp^\times$ with the center of $\GL_2(\Qp)$.
By \cite[Thm 33]{barthel} and \cite[Thm 19]{barthel}, 
there exists
smooth irreducible 
$\GL_2(\Zp)\Qp^\times$-representations $\sigma_i$
and exact sequences
\begin{equation}
	0\to 
	\cInd_{\GL_2(\Zp)\Qp^\times}^{\GL_2(\Qp)}\sigma_i\to
	\cInd_{\GL_2(\Zp)\Qp^\times}^{\GL_2(\Qp)}\sigma_i\to
	\pi_i\to 0
\end{equation}
Let $Z_Q$ act on above by $\upsilon$
and apply $\Ext^i_{Q'}(*,S')$
for $S'=\Ord_P(S(U^p,E/\oo))_{\fm}'$,
we obtain
\begin{equation*}
    \begin{tikzcd}[row sep=2ex]
	    \Ext^{i-1}_{Q'}
	    (\pi, S')\arrow[r] &
	    \Ext^{i}_{Q'}
        (\cInd_{KZ\times Z_Q}^{Q'}\sigma, S')
	    \arrow[r] \arrow[d,equal] &
	    \Ext^{i}_{Q'}
        (\cInd_{KZ\times Z_Q}^{Q'}\sigma, S')
	    \arrow[r] \arrow[d,equal] &
	    \Ext^{i}_{Q'}(\pi, S')\\ 
	 & \Ext^i_{\GL_2(\Zp)\times Z_Q}(\sigma ,S') &
	    \Ext^i_{\GL_2(\Zp)\times Z_Q}(\sigma ,S') &
    \end{tikzcd}
\end{equation*}
for $K=\GL_2(\Zp)$ and $Z=\Qp^\times$.
By Lemma \ref{lem:inj}
the restriction of $\Ord_P(S(U^p,E/\oo))_\fm$ to $K$
is an injective object.
Since Lemma \ref{lem:twist} and \eqref{eq:root_charm} imply that 
the restriction of $\xi_{\fm}^{1/2}\circ \Art$ to $\Zp^\times$
factors through  $\oo\llbracket 1+p\Zp\rrbracket$,
we can apply the twist to all smooth $K$-representations.
In particular as a $K$-representation
the twist $S'=\Ord_P(S(U^p,E/\oo))_\fm'$
is still an injective object.
Therefore the long exact sequence reduces to 
the following sequence, where
all the terms are finite-dimensional over $\fF$
since $S'$ is $Q'$-admissible.
\begin{equation*}
	0 \to \Hom_{Q'}(\pi_i,S')\to 
	\Hom_{\GL_2(\Zp)\times Z_Q}(\sigma_i,S')\to 
	\Hom_{\GL_2(\Zp)\times Z_Q}(\sigma_i,S')\to 
	\Ext^1_{Q'}(\pi_i,S')\to 0
\end{equation*}

Write $a_i\coloneqq \dim_{\fF} \Hom_{Q'}(\pi_i,S')=
\dim_{\fF} \Ext^1_{Q'}(\pi_i,S')$
so that $\soc(S')=\pi_1^{a_1}\oplus \pi_2^{a_2}$.
Let $\tilde{J}_{i,\fm}$ denote 
the injective envelope of $\pi_i$.
Then the injective envelope 
$\soc(S')\hookrightarrow \tilde{J}=
\tilde{J}_{1,\fm}^{'a_1}\oplus \tilde{J}_{2,\fm}^{'a_2}$
factors through an injective morphism 
$\phi_0\colon S'\to \tilde{J}$
since the inclusion $\soc(S')\hookrightarrow S'$
is essential.
Apply the same construction 
to $\soc(\coker(\phi_0))$
and use $\dim_{\fF}\Hom_{Q'}(\pi_i, \coker(\phi_0))=
\dim_{\fF}\Ext^1_{Q'}(\pi_i, S')=a_i$
gives another injective morphism
$\phi_1\colon \coker(\phi_0)\to \tilde{J}$,
which is surjective as 
$\Hom(\pi_i,\coker(\phi_1))
\cong \Ext^1(\phi_i,\coker(\phi_0))
\cong \Ext^2(\pi, S')=0$.
Now the lemma follows by picking
$r=\max\{a_1,a_2\}$ and taking the Pontryagin dual.
\end{proof}

Let 
$A\colon \tilde{P}_{\B,\fm}^{\oplus r}\to\tilde{P}_{\B,\fm}^{\oplus r}$ 
denote the morphism in \eqref{eq:resolution} and
identify $\Phi_{ij}$ with $\Phi_{ij}\otimes \id_{\tilde{P}}\in
\Hom_{Q'}(\tilde{P}_{j,\fm},\tilde{P}_{i,\fm})$.
By \eqref{eq:proj_1}
the morphism $A$
can be represented by a matrix
$A=\smat{A_{11} & A_{12}\Phi_{12}\\A_{21}\Phi_{21} & A_{22}}$,
where $A_{ij}\in M_r(R_{\fm})$.
By Lemma \ref{lem:proj_enve}
and Lemma \ref{lem:ker_red} we have
$\Ord(\tilde{P}_{i,\fm}^\vee)^\vee\cong 
\tilde{P}_{\chi_i^\vee,\fm}$
and $\Ord(A)$ is represented by the matrix
$\smat{\tilde{A}_{11} & \\& \tilde{A}_{22}}$,
where $\bar{A}_{ij}\in M_r(R^{\red}_{\fm})$ are 
the reductions of the matrices.
Apply the right exact functor
$\Ord$ to 
\eqref{eq:resolution} gives
\begin{equation}\label{eq:exact_ord}
	\tilde{P}_{\chi_1^\vee,\fm}^{\oplus r}\oplus 
	\tilde{P}_{\chi_2^\vee,\fm}^{\oplus r}
	\xrightarrow{\overline{A}_{11}\oplus\overline{A}_{22}}
	\tilde{P}_{\chi_1^\vee,\fm}^{\oplus r}\oplus 
	\tilde{P}_{\chi_2^\vee,\fm}^{\oplus r}
	\to M^{\ord}(U^p)_\fm'\to 0
\end{equation}
It follows that 
$M^{\ord}(U^p)_\fm'$ is the direct sum of
$M^{\ord}(U^p)_{\fm,i}'\coloneqq \coker(\bar{A}_{ii})$,
on which the action of 
$\Qp^\times=\{\smat{1&\\&*}\}\subset \GL_2(\Qp)$
has all irreducible subquotients
isomorphic to $\chi_i^\vee$.

\begin{cor}\label{cor:two_max}
There are only two maximal ideals 
$\fm_1,\fm_2$ in $\TT^{\ord}(U^p,\oo)_{\fm}$,
which are determined by 
$\Psi_{2,w_0}\bmod \fm_i=\chi_i$,
and $M^{\ord}(U^p)_{\fm,i}'=M^{\ord}(U^p)_{\fm_i}$
as $\TT^{\ord}(U^p,\oo)_{\fm_i}$-modules.
\end{cor}
\begin{proof}
Since $\TT^{\ord}(U^p,\oo)$ is 
the $\TT(U^p,\oo)$-algebra generated by the image of $\Psi_{2,w_0}$
in Proposition \ref{prop:big_char_at_p},
the corollary follows from the above decomposition of 
$M^{\ord}(U^p)_{\fm}'$ and the fact that 
the action of $\Qp^\times=\{\smat{1&\\&*}\}$ on which 
factors through $\TT^{\ord}(U^p,\oo)$
via $\xi_{\fm}^{1/2}\Psi_{2,w_0}$ by 
Proposition \ref{prop:twist_ord}.
Note that $\xi_{\fm}^{1/2}\Psi_{2,w_0}\equiv\Psi_{2,w_0}\bmod\fm_i$
since $\xi_{\fm}$ is congruent to the trivial character.
\end{proof}

\begin{lem}    
The sequence \eqref{eq:exact_ord}
is also left exact.
\end{lem}
\begin{proof}
    Since $S^{\ord}(U^p,E/\oo)_\fm'$
    is an admissible $T'$-representation,
	by \cite[Lem 2.2.11]{emeI}
    the Pontryagin dual $M^{\ord}(U^p)_\fm'$
    is finite over the algebra
    $\Lambda=\oo\llbracket T(p^1)\rrbracket$
    introduced in \eqref{def:lambda_rings}.
	Apply $\Hom_{T'}(\tilde{P}_{\chi_i^\vee,\chi},*)$
	to \eqref{eq:exact_ord} gives
\begin{equation*}
	\End_{T'}(\tilde{P}_{\chi_i^\vee,\fm})^r\xrightarrow{\bar{A}_{ii}}
	\End_{T'}(\tilde{P}_{\chi_i^\vee,\fm})^r\to 
	\Hom_{T'}(\tilde{P}_{\chi_i^\vee,\fm}, M^{\ord}(U^p)_\fm')
	\to 0
\end{equation*}
    We observe that 
    $\Hom_{T'}(\tilde{P}_{\chi_i^\vee,\fm}, M^{\ord}(U^p)_\fm')$
    must be torsion over $R_{\fm}^\red$
    since the relative dimension of $\Lambda$ is
	strictly smaller than that of $R_{\fm}^{\red}$,
    while $\End_{T'}(\tilde{P}_{\chi_i^\vee,\fm})$
	is finite free over $R_{\fm}^{\red}$.
	Passing to the field of fraction of $R_{\fm}^{\red}$,
	we see that the matrix $\bar{A}_{ii}$ has to be invertible
	and the result follows.
\end{proof}
\begin{prop}\label{prop:no_torsion}
    The $R_\fm$-module $M(U^p)_{\fm}'$
    has no $\xx$-torsion.
\end{prop}
\begin{proof}
    By the anti-equivalence \eqref{eq:anti_equiv},
	it suffices to show that
	$\Hom_{Q'}(\tilde{P}_{j,\fm}, M(U^p)_{\fm}'[\xx])=0$
	for $j=1,2$.
    Since $R_\fm$ acts faithfully on $\tilde{P}_{i,\fm}$
    by \eqref{eq:proj_1}, 
	we can apply the snake lemma to the commutative diagram
    below then apply the exact functor
    $\Hom_{Q'}(\tilde{P}_{j,\fm}, *)$
    to the resulting sequence.
    \begin{equation*}
    \begin{tikzcd}
    0 \arrow[r] & 
    \tilde{P}_{\B,\fm}^{\oplus r} 
	\arrow[d,"\xx",hookrightarrow] \arrow[r,"A"] & 
	\tilde{P}_{\B,\fm}^{\oplus r} 
	\arrow[d,"\xx",hookrightarrow] \arrow[r] & 
	M(U^p)_{\fm}'
    \arrow[d,"\xx"]  \arrow[r] & 0 \\ 
    0 \arrow[r] & 
    \tilde{P}_{\B,\fm}^{\oplus r}
	\arrow[r,"A"] & 
    \tilde{P}_{\B,\fm}^{\oplus r}
	\arrow[r] &
    M(U^p)_{\fm}'  
    \arrow[r] & 0 
    \end{tikzcd}
\end{equation*}
From which we can observe that
$\Hom_{Q'}(\tilde{P}_{j,\fm}, M(U^p)_{\fm}'[\xx])$
is isomorphic to the kernel of the homomorphism below,
which is injective by the previous lemma.
\begin{equation}\label{eq:fil_ij}
	\smat{\bar{A}_{ii}& \bar{A}_{ij}\Phi_{ij}\\& \bar{A}_{jj}}\colon 
	\Hom_{Q'}(\tilde{P}_{j,\fm},\tilde{P}_{\B,\fm}^{\oplus r})^{\red}\to
	\Hom_{Q'}(\tilde{P}_{j,\fm},\tilde{P}_{\B,\fm}^{\oplus r})^{\red}
\end{equation}
\end{proof}

\begin{cor}\label{cor:fil_by_ord}
	We have the following exact sequence of
	$\TT(U^p,\oo)_\fm\times R_{\fm}^{\red}$-modules
\begin{equation*}
\begin{aligned}
	0\to M^{\ord}(U^p)_{\fm,1}'\to
	\Hom_{Q'}(\tilde{P}_{2,\fm},M(U^p)_{\fm}')^{\red}
    \to
	M^{\ord}(U^p)_{\fm,2}'\to0\\
	0\to M^{\ord}(U^p)_{\fm,2}'\to
	\Hom_{Q'}(\tilde{P}_{1,\fm},M(U^p)_{\fm}')^{\red}
    \to
	M^{\ord}(U^p)_{\fm,1}'\to0
\end{aligned}
\end{equation*}
\end{cor}
\begin{proof}
    Apply the snake lemma again to the 
    commutative diagram in the previous proposition 
    and apply $\Hom_{Q'}(\tilde{P}_{2,\fm},*)$
    to the resulting exact sequence
	gives the middle row of the commutative diagram
\begin{equation*}
    \begin{tikzcd}
	    0 \arrow[r]& 
	    \Hom_{Q'}(\tilde{P}_{2,\fm},\tilde{P}_{1,\fm}^{\oplus r})^{\red}
	    \arrow[r,"\bar{A}_{11}"] \arrow[d]&
	    \Hom_{Q'}(\tilde{P}_{2,\fm}, \tilde{P}_{1,\fm}^{\oplus r})^{\red}
	    \arrow[d] &&\\
	    0\arrow[r] & 
	    \Hom_{Q'}(\tilde{P}_{2,\fm},\tilde{P}_{\B,\fm}^{\oplus r})^{\red}
	    \arrow[r,"\eqref{eq:fil_ij}"] \arrow[d,"\Ord"] &
	    \Hom_{Q'}(\tilde{P}_{2,\fm},\tilde{P}_{\B,\fm}^{\oplus r})^{\red}
	    \arrow[d,"\Ord"] \arrow[r]&
	    \Hom_{Q'}(\tilde{P}_{2,\fm}, M(U^p)_{\fm}')^{\red}\arrow[r]
        \arrow[d,"\Ord"] &0\\
	    0\arrow[r] & 
	    \End_{T'}(\tilde{P}_{\chi_2^\vee,\fm})^{r}
	    \arrow[r,"\bar{A}_{22}"] &
	    \End_{T'}(\tilde{P}_{\chi_2^\vee,\fm})^{r}
        \arrow[r]&
        \Hom_{T'}(\tilde{P}_{\chi_2^\vee,\fm}, M^{\ord}(U^p)_\fm')
        \arrow[r]&0
    \end{tikzcd}
\end{equation*}
The exactness of the columns follows from 
$\Hom_{Q'}(\tilde{P}_{2,\fm}, \tilde{P}_{2,\fm})^{\red}
\cong \End_{T'}(\tilde{P}_{\chi_2^\vee,\fm})$
by Lemma \ref{lem:ker_red}
and the exactness of the rows follows from the previous lemma.
Since $\tilde{P}_{\chi_i^\vee,\fm}$
is free of rank one over over $R_\fm^{\red}$
by Lemma \ref{lem:end_comp} and
$\Hom_{T'}(\tilde{P}_{\chi_i^\vee,\fm},\tilde{P}_{\chi_j^\vee,\fm})=0$
if $i\neq j$, we have
\begin{equation}
    \Hom_{T'}(\tilde{P}_{\chi_i^\vee,\fm}, M^{\ord}(U^p)_\fm')\cong
    \Hom_{T'}(\tilde{P}_{\chi_i^\vee,\fm}, M^{\ord}(U^p)_\fm')
    \hat{\otimes}_{R^{\red}_\fm}\tilde{P}_{\chi_i^\vee,\fm}\cong
    M^{\ord}(U^p)_{\fm, i}'
\end{equation}
as $\TT(U^p,\oo)_\fm\times R_{\fm}^\red$-modules,
where the last isomorphism is the valuation map
(see \cite[Lem. 3.24]{pask}).
Similarly,  since
$\Hom_{Q'}(\tilde{P}_{2,\fm}, \tilde{P}_{1,\fm})^{\red}=
\Hom_{Q'}(\tilde{P}_{1,\fm}, \tilde{P}_{1,\fm})^{\red}
\hat{\otimes}_{R_\fm^\red}R_\fm^{\red}\Phi_{12}$,
the cokernel of the first row is
$\Hom_{Q'}(\tilde{P}_{1,\fm}, M(U^p)_{\fm}')^{\red}
\hat{\otimes}_{R_\fm^\red}R_\fm^{\red}\Phi_{12}$,
which is isomorphic to $M^{\ord}(U^p)_{\fm, 1}'$
as a $\TT(U^p,\oo)_\fm\times R_{\fm}^\red$-module.
From these we get the first claimed exact sequence,
and the second one can be derived similarly.
\end{proof}  

Now recall that both $\TT(U^p,\oo)$
and $\TT^{\ord}(U^p,\oo)$
are algebras over $\Lambda=\oo\llbracket T(p^1)\rrbracket$
via the diamond operators in Definition \ref{def:ord_hecke}.
In particular $M(U^p)_\fm'$ is also a module over $\Lambda$ and
$\Lambda\hat{\otimes}_{\oo}\oo\llbracket\xx\rrbracket=
\Lambda\llbracket\xx\rrbracket$.

\begin{prop}\label{prop:Hecke_finite}
The modules
$\Hom_{Q'}(\tilde{P}_{1,\fm},M(U^p)_{\fm}'),
\Hom_{Q'}(\tilde{P}_{2,\fm},M(U^p)_{\fm}')$,
are finite free over 
$\Lambda\llbracket \xx\rrbracket$
and consequently so is 
$\mathbf{m}=
\Hom_{Q'}(\tilde{P}_{\B,\fm},M(U^p)_{\fm}')$.
\end{prop}
\begin{proof}
    The same proof in Proposition \ref{prop:ord_to_dual}
    shows that $M^{\ord}(U^p)_\fm$ is finite free over $\Lambda$
    under the assumption \eqref{cond:small}.
    Therefore $M^{\ord}(U^p)_{\fm,i}$ are finite free over $\Lambda$
    by Corollary \ref{cor:two_max} and so are 
    $\Hom_{Q'}(\tilde{P}_{i,\fm},M(U^p)_{\fm}')^{\red}$
    by the previous corollary.

    Write $\mathbf{m}_i=\Hom_{Q'}(\tilde{P}_{i,\fm},M(U^p)_{\fm}')$.
    Since the finite $R_\fm$-module
    $\mathbf{m}_i$ has a natural compact topology
    and $\mathbf{m}_i^\red$ is finite free over $\Lambda$,
	by the topological Nakayama lemma
	$\mathbf{m}_i$ is finite over $\Lambda\llbracket x\rrbracket$.
    Suppose $r=\rank_{\Lambda}\mathbf{m}_i^\red$
    and fix an isomorphism $\Lambda^{\oplus r}\cong \mathbf{m}_i^\red$.
    We can lift which to a surjective homomorphism
    $\Lambda\llbracket \xx\rrbracket^{\oplus r}\to\mathbf{m}$.
    Let $N$ denote the kernel of which and 
    apply the snake lemma to the diagram
    \[
    \begin{tikzcd}
    0\arrow[r]& 
    N\arrow[r]\arrow[d,"\xx"]&
    \Lambda\llbracket \xx\rrbracket^{\oplus r}\arrow[r]\arrow[d,"\xx"]&
    \mathbf{m}_i\arrow[r]\arrow[d,"\xx"] &0\\
    0\arrow[r]& 
    N\arrow[r]&
    \Lambda\llbracket \xx\rrbracket^{\oplus r}\arrow[r]&
    \mathbf{m}_i\arrow[r] &0
    \end{tikzcd}
    \]
    Then $N^\red=0$ since $\fm_i[\xx]=0$ by
    Proposition \ref{prop:no_torsion} and 
    $\Lambda^{\oplus r}\cong \fm_i^\red$.
    Therefore $N=0$ by Nakayama's lemma, which proves the claim.
\end{proof}

\begin{cor}\label{cor:Hecke_ff}

Let $\xx$ acts on $\TT(U^p,\oo)_\fm$ via 
the homomorphism in part (1) of Proposition \ref{lem:twist}.
Then $\TT(U^p,\oo)_\fm$ is a finite free 
$\Lambda\llbracket\xx\rrbracket$-algebra.
In particular we may identify $\xx$ with
a non-zero divisor in $\TT(U^p,\oo)_\fm$,
and the ideal $\xx\TT(U^p,\oo)_\fm$
is free of rank one over $\TT(U^p,\oo)_\fm$.

\end{cor}

\begin{proof}
Since $\TT(U^p,\oo)_{\fm}$ acts faithfully on 
$M(U^p)_{\fm}'$, it also acts faithfully on
$\mathbf{m}=\Hom_{Q'}(\tilde{P}_{\B,\fm},M(U^p)_{\fm}')$ 
because of the anti-equivalence \eqref{eq:anti_equiv}.
The previous proposition then gives an injective homomorphism
\[
    \TT(U^p,\oo)_\fm\hookrightarrow
    \End_{\Lambda\llbracket\xx\rrbracket}(\mathbf{m})
\]
between $\Lambda\llbracket\xx\rrbracket$-algebras
where the right hand side is finite free,
from which the corollary follows.
\end{proof}

\subsection{Fundamental exact sequence}
\label{sub:fund_exact_sequence}

Recall that for $\Lambda^+=\Lambda[\Delta_p\times\Delta_S]$
introduced in \eqref{def:lambda_rings},
both $\TT(U^p,\oo)$ and $\TT^{\ord}(U^p,\oo)$
are $\Lambda$-algebras via the diamond operators
and the map 
$\TT(U^p,\oo)\to \TT^{\ord}(U^p,\oo)$
from Lemma \ref{lem:coh_to_ord}
is a homomorphism between $\Lambda^+$-algebras.

Let $\fG=W\times \Delta$ be an abelian pro-$p$ group,
where $W$ is fintie free over $\Zp$
and $\Delta$ is a finite abelian group,
then $\I\coloneqq\oo\llbracket\fG\rrbracket$
is a local complete Noetherian flat $\oo$-algebra.
We suppose there exists a homomorphism
$\Lambda^+\to \I$ such that 
$\I$ is finite over $\Lambda^+$.
Let $\lambda^{\ord}\colon 
\TT^{\ord}(U^p,\oo)\to\I$
be a homomorphism of $\Lambda^+$-algebras and
\begin{equation}
    \lambda\colon 
    \TT(U^p,\oo)\to
    \TT^{\ord}(U^p,\oo)\to\I
\end{equation}
denote the composition, by which
we view $\I$ as a $\TT(U^p,\oo)$-algebra.
If $\fm_\I$ is the maximal ideal of $\I$,
we consider the maximal ideals
$\fm=\lambda^{-1}(\fm_\I)$ 
and $\fm^\circ=(\lambda^{\ord})^{-1}(\fm_\I)$
of the Hecke algebras.
We will assume that
$\lambda$ is surjective and that 
$\fm$ satisfies \eqref{cond:red_gen}.
Then Corollary \ref{cor:two_max} implies that
that $\fm^\circ$ is either $\fm_1$ or $\fm_2$ and,
say $\fm^{\circ}=\fm_1$,
there exists characters 
$\bar{\delta}_i\colon \Gal_\K\to \fF^\times$ such that
\begin{align*}
    (\lambda\circ T_\fm\bmod \fm_\I)&=
    (T_\fm\bmod \fm)=
    (\bar{\delta}_1+\bar{\delta}_2)\colon 
    \Gal_\K\to \fF\\
    (\lambda^{\ord}\circ \Psi_{2,w_0}\bmod\fm_\I)&=
    (\Psi_{2,w_0}\bmod\fm^{\circ})=
    (\bar{\delta}_1\vert_{\Gp})\colon 
    \Gp\to \fF^\times
\end{align*}

We write
$M^{\ord}(U^p,\I)_{\fm_i}\coloneqq
M^{\ord}(U^p)_{\fm_i}\otimes_{\Lambda^+}\I$ for $i=1,2$
and make the following assumptions.
\begin{enumerate}[label=(C\arabic*)]
\item There exists a $\TT(U^p,\oo)_\fm$-module 
homomorphism
$\Theta\colon M^{\ord}(U^p,\I)_{\fm_1}\to \I$.
\label{cond:C1}
\item There exists $F\in M^{\ord}(U^p,\I)_{\fm_1}$
such that,
for any character $\alpha$ of $\Delta$,
the image of $\Theta(F)$ under the induced homomorphism
$\alpha\colon\I\to\oo\llbracket W\rrbracket$ is nonzero.
\label{cond:C3}
\item There exists an ideal 
$\fq\subset \TT(U^p,\oo)_{\fm}$
that contains $\xx$ such that
\[
    \fq M^{\ord}(U^p)_{\fm_2}=0\quad
    \text{ and }\quad
    \fq F=0 \text{ in } M^{\ord}(U^p,\I)_{\fm_1}.
\]
\label{cond:C2}

\end{enumerate}

In the rest of the section, we write
$M=\Hom_{Q'}(\tilde{P}_{1,\fm},M(U^p)_{\fm}')
\otimes_{\Lambda^+}\I$,
$M_i^{\ord}=M^{\ord}(U^p,\I)_{\fm_i}$ for $i=1,2$ and
$\TT=\TT(U^p,\oo)_{\fm}$.
By Corollary \ref{cor:fil_by_ord}
we have a right exact sequence of $\TT(U^p,\oo)_\fm$-modules
\begin{equation}\label{eq:chain_coh}
    M_2^{\ord}\to M/\xx M\xrightarrow{(*)} M_1^{\ord}\to 0
\end{equation}
We define $S_1^{\ord}=\ker(\Theta)$ and let $S\subset M$
be the pre-image of $S_1^{\ord}$ under $(*)$.
Note that $\xx M\subset S$ by definition.
For any character $\alpha$ of $\Delta$,
we let $\Lambda_\alpha$ denote the $\I$-algebra structure
on $\Lambda_W=\oo\llbracket W\rrbracket$ given by 
the induced homomorphism $\alpha\colon\I\to \Lambda_W$.
We also write $\lambda_\alpha$
for the composition of which with $\lambda$.

\begin{lem}\label{lem:right_exact}
There exists an exact sequence
of $\TT$-modules
\begin{equation}\label{eq:fund}
	M/S\xrightarrow{\xx} \fq M/\fq S\to 
	\fq M_1^\ord/\fq S_1^\ord \to 0
\end{equation}
\end{lem}

\begin{proof}
Since $\xx\in \fq$ by \ref{cond:C2}
we have the commutative diagram
\[
\begin{tikzcd}
	S\arrow[r,"\xx"] \arrow[d]
	& \fq S \arrow[r] \arrow[d]
	& \fq S^{\red} \arrow[r] \arrow[d] & 0\\
	M\arrow[r,"\xx"]
	& \fq M \arrow[r]
	& \fq M^{\red} \arrow[r] & 0
\end{tikzcd}
\]
Applying the snake lemma
reduces the lemma to showing that
$\fq M^{\red}/\fq S^{\red}=\fq M_1^{\ord}/\fq S_1^{\ord}$.
Observe that $\fq M^{\red}=\fq M/\xx M$ 
is mapped injectively into $M^{\red}=M/\xx M$ and we have
\begin{align*}
    M_2^{\ord}\to M/\xx M\to M_1^{\ord}\to 0
    &\Longrightarrow
	\Image(\fq M^{\red}\to M^{\red})=
    \fq M/\xx M\cong\fq M_1^{\ord}\\
    M_2^{\ord}\to S/\xx M\to S_1^{\ord}\to 0
    &\Longrightarrow
	\Image(\fq S^{\red}\to M^{\red})=
    \fq S+\xx M/\xx M\cong\fq S_1^{\ord}.
\end{align*}
where the last isomorphisms 
follows from \ref{cond:C2}.
From the above we have the desired isomorphism.
\end{proof}

Let $\hat{\Delta}$ be the set 
of characters on the finite abelian group $\Delta$.
We write $K_W=\textnormal{Frac}\Lambda_W$
and $K_\alpha=\textnormal{Frac}\Lambda_\alpha$
for $\alpha\in \hat{\Delta}$
if we want to specify the $\I$-algebra structure.
Note that the natural map
$\I\hookrightarrow\prod_{\alpha\in\hat{\Delta}}\Lambda_\alpha$ 
is injective and the cokernel is $p$-torsion.
Now observe that each terms in 
the sequence \eqref{eq:fund}
is a quotient of $M/S$ 
and the $\TT$-action on which factor through $\lambda$.
Thus \eqref{eq:fund} is also a sequence of $\I$-modules.

\begin{lem}\label{lem:inj_crit}
The sequence \eqref{eq:fund}
is left exact if 
$(\fq M/\fq S)\otimes_{\I}K_\alpha\neq 0$
for all $\alpha\in\hat{\Delta}$.
\end{lem}
\begin{proof}
Since $\Theta$ induces $M/S\hookrightarrow\I$,
the $\I$-module $M/S$ has no $p$-torsion and we have the diagram
\[
\begin{tikzcd}
M/S
\arrow[r,"(a)"]\arrow[d,hookrightarrow] &
\fq M/\fq S
\arrow[r]\arrow[d] &
\fq M_1^{\ord}/\fq S_1^{\ord}
\arrow[r]\arrow[d] & 0\\
(M/S)\otimes_{\I}\I[1/p]
\arrow[r,"(b)"] &
(\fq M/\fq S)\otimes_{\I}\I[1/p]
\arrow[r] &
(\fq M_1^{\ord}/\fq S_1^{\ord})\otimes_{\I}\I[1/p]
\arrow[r] & 0
\end{tikzcd}
\]
Therefore $(a)$ is injective if
$(b)$ is injective.
Moreover, since $\I[1/p]=\prod_\alpha \Lambda_\alpha[1/p]$,
the map $(b)$ is injective if and only if the map $(b)_\alpha$ below
is injective for each $\alpha\in\hat{\Delta}$.
\[
(M/S)\otimes_{\I}\Lambda_\alpha[1/p]
\xrightarrow{(b)_\alpha}
(\fq M/\fq S)\otimes_{\I}\Lambda_\alpha[1/p] \to
(\fq M_1^{\ord}/\fq S_1^{\ord})\otimes_{\I}\Lambda_\alpha[1/p]\to 0
\]
Since $\Theta$ also induces
$(M/S)\otimes_{\I}\Lambda_\alpha[1/p]
\hookrightarrow\Lambda_\alpha[1/p]$,
we can pass to the field of fraction and get
\[
\begin{tikzcd}
(M/S)\otimes_{\I}\Lambda_\alpha[1/p]
\arrow[r,"(b)_\alpha"]\arrow[d,hookrightarrow] &
(\fq M/\fq S)\otimes_{\I}\Lambda_\alpha[1/p]
\arrow[r]\arrow[d] &
(\fq M_1^{\ord}/\fq S_1^{\ord})\otimes_{\I}\Lambda_\alpha[1/p]
\arrow[r]\arrow[d] & 0\\
(M/S)\otimes_{\I}K_\alpha
\arrow[r,"(b)_\alpha'"] &
(\fq M/\fq S)\otimes_{\I}K_\alpha
\arrow[r] &
(\fq M_1^{\ord}/\fq S_1^{\ord})\otimes_{\I}K_\alpha
\arrow[r] & 0
\end{tikzcd}
\]
From which we see that $(b)_\alpha$ 
is injective if $(b)_\alpha'$ is injective.
Since the map of $K_W$-modules
$(M/S)\otimes_{\I}K_\alpha\hookrightarrow K_\alpha$
induced by $\Theta$ is nonzero by \ref{cond:C3},
we can identify
$(M/S)\otimes_{\I}K_\alpha=K_\alpha$.

On the other hand,
let $E_1^{\ord}\subset M_1^{\ord}$
be the $\TT$-submodule generated by $F$.
Then \ref{cond:C3} implies that
$(M_1^{\ord}/(S_1^{\ord}+E_1^{\ord}))\otimes_{\I}\Lambda_\alpha$
is a torsion module and \ref{cond:C2} implies that
$\fq(S_1^{\ord}+E_1^{\ord})=\fq S_1^{\ord}$.
Consequently
\[
    0=\fq\otimes_{\TT}\left(M_1^{\ord}/(S_1^{\ord}+E_1^{\ord})\right)
    \otimes_{\I}K_\alpha
    =\fq\otimes_{\TT}(M_1^{\ord}/S_1^{\ord})
    \otimes_{\I}K_\alpha
    \twoheadrightarrow
    (\fq M_1^{\ord}/\fq S_1^{\ord})\otimes_{\I}K_\alpha
\]
In other word $(b)_\alpha'$ is a surjective map from $K_W$,
and therefore it is injective if and only if
$(\fq M/\fq S)\otimes_{\I}K_\alpha$ is nonzero.
This conclude the proof of the lemma.
\end{proof}

We now consider a fixed 
$\lambda_\alpha\colon \TT\to \Lambda_\alpha$
and define the prime ideals 
$\wp_1=\ker(\lambda_\alpha)\subset \TT$,
$\fp_1=\wp_1\cap \Lambda\llbracket\xx\rrbracket$, and
$\fp=\fp_1\cap \Lambda$.
Note that $\fp_0\coloneqq (\fp\Lambda)\llbracket\xx\rrbracket\subset \fp_1$
is also a prime ideal.
And in fact $\fp_1=\fp_0+(\xx)$
since $\lambda$ factors through the ordinary Hecke algebra.
Let $\Lambda'\coloneqq \Lambda/\fp$,
which is isomorphic to a subring of $\Lambda_\alpha$ via $\lambda_\alpha$
and $K$ denote the field of fraction of $\Lambda'$,
then $K_\alpha/K$ is a finite extension 
since $\I$ is finite over $\Lambda$.

\begin{lem}\label{lem:smallprime}
There exists a prime ideal 
$\wp_0\subset \wp_1$ of $\TT$ such that 
$\wp_0\cap \Lambda\llbracket \xx\rrbracket=\fp_0$.
\end{lem}
\begin{proof}
By abuse of notation we will write
$\Lambda\llbracket\xx\rrbracket_{\fp_1}$
for the completion of the localization of
$\Lambda\llbracket\xx\rrbracket$ at $\fp_1$
and refer localizations at $\fp_1$ as taking
$\otimes_{\Lambda\llbracket\xx\rrbracket}
\Lambda\llbracket\xx\rrbracket_{\fp_1}$.
Then $\TT_{\fp_1}/\fp_0\TT_{\fp_1}$
is finite free over
the complete discrete valuation ring
\[
\Lambda\llbracket\xx\rrbracket_{\fp_1}/
\fp_0\Lambda\llbracket\xx\rrbracket_{\fp_1}\cong
\left(
\Lambda\llbracket\xx\rrbracket/\fp_0
\right)_{\fp_1}\cong
(\Lambda'\llbracket\xx\rrbracket)_{\fp_1}\cong 
K\llbracket\xx\rrbracket
\]
by Proposition \ref{prop:Hecke_finite}.
Therefore $\TT_{\fp_1}/\fp_0\TT_{\fp_1}$
is the direct product of its localizations at the maximal ideals.
In particular $\TT_{\wp_1}/\fp_0\TT_{\fp_1}$
is one of the component and is also finite free over 
$K\llbracket\xx\rrbracket$.
This implies that $\TT_{\wp_1}/\fp_0\TT_{\fp_1}$ cannot be Artinian
and the pull-back of any minimal ideal of which to $\TT$
gives the desired ideal.

\end{proof}

Let $\wp_0\subset \TT$ be a prime above $\fp_0$,
then $\TT_{\wp_1}/\wp_0\TT_{\wp_1}$
is an integral domain and a finite extension
over $K\llbracket\xx\rrbracket$.
The integral closure of $K\llbracket\xx\rrbracket$
in the field of fraction of 
$\TT_{\wp_1}/\wp_0\TT_{\wp_1}$ is then a Dedekind domain.
Let $\tilde{\TT}$ denote the localization of which
at any maximal ideal.
Then $\tilde{\TT}$ is a discrete valuation ring
and $\fq\tilde{T}\cong \tilde{T}$.
From which we have the isomorphism
\begin{multline*}
((\fq M/\fq S)\otimes_{\I} K_\alpha)
\otimes_{\TT_{\wp_1}}\tilde{\TT}\cong 
(\fq M\otimes_{\I} K_\alpha)
\otimes_{\TT_{\wp_1}}\tilde{\TT}/
(\fq S\otimes_{\I} K_\alpha)
\otimes_{\TT_{\wp_1}}\tilde{\TT}\\\cong
(M\otimes_{\I} K_\alpha)
\otimes_{\TT_{\wp_1}}\tilde{\TT}/
(S\otimes_{\I} K_\alpha)
\otimes_{\TT_{\wp_1}}\tilde{\TT}\cong
((M/S)\otimes_{\I} K_\alpha)
\otimes_{\TT_{\wp_1}}\tilde{\TT}=
K_\alpha \otimes_{\TT_{\wp_1}}\tilde{\TT}
\end{multline*}
The last ring is nonzero since 
the spectrum of which 
is nothing but the pre-image
under the finite morphism
$\Spec(\tilde{\TT})\to \Spec(\TT_{\wp_1})$
of the maximal ideal $\wp_1\TT_{\wp_1}=\ker(\TT_{\wp_1}\to K_\alpha)$,
which is the maximal ideal of $\tilde{\TT}$.
By Lemma \ref{lem:inj_crit}
we can conclude that \eqref{eq:fund} is left exact.
We organize this result into the following corollary.

\begin{cor}\label{cor:fund}
	Under the setting of this subsection
    and the assumptions
    \ref{cond:C1},
    \ref{cond:C2}, and
    \ref{cond:C3},
    we have the following commutative diagram
    of exact sequences of $\TT$-modules
    \begin{equation}
    \begin{tikzcd}
    & (M/S)
    \arrow[r,"\xx"]\arrow[d,equal] &
    (M/S)\otimes_{\TT}\fq
    \arrow[r]\arrow[d] &
    (M/S)\otimes_{\TT}\fq^\red
    \arrow[r]\arrow[d] &0\\
    0\arrow[r] &
    (M/S)
    \arrow[r,"\xx"] &
    \fq M/\fq S
    \arrow[r] &
    \fq M_1^{\ord}/\fq S_1^{\ord}
    \arrow[r] &0
    \end{tikzcd}
    \end{equation}
\end{cor}
\begin{proof}
The vertical maps are the natural ones
and the commutative properties 
follows from the isomorphism
$\fq M^\red/\fq S^\red=\fq M_1^{\ord}/\fq S_1^{\ord}$
shown in the proof of Lemma \ref{lem:right_exact}.
The exactness of the first row
is straightforward and
the exactness of the second row
follows from the discussions above.
\end{proof}

\begin{rem}
Aside from some modifications,
the formulation and the proof of the corollary
are entirely from \cite[Prop. 6.3.5]{urban},
in which the corollary is proven 
for the case of Eisenstein ideals on modular curves.
We also refer to \cite[Thm 3.4.1]{urban}
for a different proof 
which relies on the geometry of eigencurves.
\end{rem}

\section{Construction of Euler systems}

Let $G$ and $H$ denote the algebraic groups over $\F$
defined by \eqref{def:def_unitary}
for $n=2$ and $n=1$ respectively.
The Hermitian pairing on $\K^{\oplus 2}$
that defines $G$ can be twisted into 
a skew-Hermitian pairing after multiplied by 
a nonzero $\delta\in \K$ such that $\delta=-\bar{\delta}$.
Then $(G, H)$ forms a dual reductive pair
and the theory of theta correspondence 
allows us to transfer automorphic representations
between $G$ and $H$.

In particular, 
let $\eta$ be a Hecke character on $\K$,
then the theta lift $\pi=\pi_\eta$ of the restriction of $\eta$
to $H(\A)=\A_{\K}^1$ could be zero or an irreducible 
automorphic representation. 
And when $\pi\neq 0$, the associated Galois representation
$r_\pi$ is essentially the direct product of $\eta$ and $\eta^c$.
Alternatively, one can think of $\pi$ as the Jacquet-Langlands
transfer of the classical CM representations.

In the author's previous work \cite{lee},
it is shown that the theta lifts
can be obtained from the quasi-split unitary group
of signature $(2,2)$ through a pull-back procedure.
Using the theory of integral models
of PEL-type Shimura varieties
associated to quasi-split unitary groups
developed in \cite{Hida04},
it is then shown in \cite{lee}
that these theta lifts 
are $p$-integral up to powers of the canonical
CM periods of $\K$
and form a Hida family.
We note that the same technique has already appeared in \cite{wan}.
And such pull-back argument is well-known 
for Eisenstein series and is used for example in 
\cite{SU},\cite{EHLS}, and \cite{Hsieh2014}.

To fix the idea, let $\psi$ be a finite order Hecke character
of $\K$ which is ramified only at primes that splits in $\K$ and
such that the associated character of the Galois group
is anticyclotomic.
Moreover, we assume that $\psi$ satisfies
\begin{equation}\label{cond:gen_psi}\tag{$\psi$-gen}
    \psi\vert_{D_{w_0}}\not\equiv 1,\omega^{\pm1}.
\end{equation}
For each ideal $\ff\subset \oo_\K$
that is a square-free product of split primes
and such that $\ff+\ff^c=\oo_K$,
we put $\fs=\ff\ff^c$ and
let $\fG_{\fs}$ and $\fG_\fs^a$ denote respectively
the Galois group of the maximal pro-$p$ extension over $\K$
with tame level $\fs$ and the anticyclotomic quotient.
Use the results from \cite{lee},
we will show the existence of a Hida family
$\euF_\fs$ over 
$\I_\fs\coloneqq \eo\llbracket\fG_\fs^a\rrbracket$,
and verify \ref{cond:C1} and \ref{cond:C3}
for the corresponding Hecke eigensystem.
We remark that here $\eo$ is a finite extension
over $\bar{\Q}_p^{un}$,
but the relevant results from previous sections still holds 
after the flat base change from $\oo$,
the ring of integers of a finite field extension over $\Qp$,
to $\eo$.

Then after verifying \ref{cond:C2},
we will apply Corollary \ref{cor:fund}
and construct a nontrivial cohomology class for 
\[
    \Psi_\fs\colon\Gal_\K\to \I_\fs\quad
    \gamma\mapsto \psi(\gamma)^{-1}\langle \gamma\rangle
\]
where $\langle*\rangle$ denote the tautological character
to $\I_\fs=\eo\llbracket\fG_\fs^a\rrbracket$.
At last, we show that these classes are compatible under the norm maps 
in a way that justifies calling them
an ``anticyclotomic Euler system''
when we vary the ideal $\ff$.
We postpone the application of the Euler system
to Iwasawa theory till the next section.

\subsection{Hida family of theta lifts}

To define the theta correspondence between $G$ and $H$
it is necessary to fix an auxiliary Hecke character
$\chi$ of $\K$ which restricts to 
$\qch_{\K/\F}$ on $\A^\times$.
We further require that $\chi$ satisfies the following conditions.
\begin{enumerate}[label={($\chi$\arabic*)}]
    \item \label{cond:chi1}
    The central $L$-value $L(\frac{1}{2},\chi)$ is nonzero.
    \item \label{cond:chi2}
    The conductor $\fc$ of $\chi$ is prime-to-$p$.
    \item \label{cond:chi3}
    The Hecke character 
    $\chi_\circ\coloneqq \chi|\cdot|_\K^{1/2}$ is algebraic and
    has the infinity type $\Sigma^c$
\end{enumerate}

\begin{defn}\label{def:admchar}
An algebraic Hecke character $\eta$ of $\K$ is admissible
if it satisfies the following conditions.
\begin{enumerate}
\item 
There exists an ideal $\fs=\fs^c$
consisting only of split primes such that 
the conductor of $\eta$ divides $\fs$.
\item 
It has the infinity type $k\Sigma$
for some $k\geq 0$.
\end{enumerate}
\end{defn}

\begin{prop}\label{prop:single}
When $\eta$ is admissible of the infinity type $k\Sigma$,
there exists an algebraic modular form
$f(\eta)\in S_{\wt{k}}^{\ord}(K^p\Iw(p^{n,n}),\eo)$,
of weight $\wt{k}=(0,-k\Sigma)$ and level given by
some open compact subgroup $K\subset G(\A_f)$,
with the following properties.
\begin{itemize}
\item The central character of $f(\eta)$
is the restriction of $\hat{\eta}$ to $\A_{\K,f}^1$.
\item  It is an eigenform for 
$\TT_{\wt{k}}(U^p\Iw(p^{n,n}),\eo)$ with eigenvectors:
\begin{align*}
&T_w^{(1)}\mapsto
q_w\chi_\circ(\varpi_w)+\chi_\circ^{-1}\tilde{\eta}(\varpi_w)&
&T_w^{(2)}\mapsto \tilde{\eta}(\varpi_w)\\
&U_{\wt{k},w}^{(1)}\mapsto
(\chi_\circ^{-1}\tilde{\eta}^\wedge)(\varpi_w)&
&U_{\wt{k},w}^{(2)}\mapsto
\tilde{\eta}^\wedge(\varpi_w)&
&\langle u\rangle_{\wt{k}}\mapsto
\tilde{\eta}^\wedge(u_{11})\\
&U_{w}^{(1)}\mapsto\chi_\circ^{-1}\tilde{\eta}(\varpi_w)&
&U_{w}^{(2)}\mapsto\tilde{\eta}(\varpi_w)&
&\langle u\rangle\mapsto \tilde{\eta}(u_{11})
\end{align*}
\end{itemize}
\end{prop}
\begin{proof}
    This is \cite[Prop. 6.6]{lee} and \cite[Prop. 6.7]{lee} 
    respectively.
\end{proof}

\begin{rem}
    In the notation of \cite{lee} $f(\eta)$ is an algebraic modular
    form for the representation $\rho_{\wt{k}}$ of 
    lowest weight $(-k\Sigma, 0)$. Therefore we may identify
    $\rho_{\wt{k}}$ with $\xi_{\wt{k}}$ for 
    $\wt{k}=(0, -k\Sigma)$ in the notation of the current article.
\end{rem}

\begin{lem}\label{lem:good_chi}
Assume that $\K$ satisfies the following conditions.
\begin{enumerate}[label=($\K$\arabic*)]
\item The extension $\K/\F$ is generic 
($\oo_\K^\times\subset \F$
and $Cl_\F\to Cl_\K$ is injective)
in the sense of \cite{Rohrlich},
\label{cond:K2}
\item Every prime of $\F$ above $2$ splits in $\K$.
\label{cond:K3}
\end{enumerate}
Then there exists a unitary Hecke character $\chi$
with the following properties.
\begin{enumerate}
\item The restriction of $\chi$ to $\A_F^\times$ is $\qch_{\K/\F}$.
\item The character $\chi_\circ=\chi|\cdot|^{1/2}_\K$
has the infinity type $\Sigma^c$.
\item The central value $L(1/2,\chi)$ is nonzero.
\end{enumerate}
\end{lem}
\begin{proof}

Let $\chi_{can}$
be the canonical Hecke character 
with the infinity type $-\Phi=-\Sigma^c$
defined in \cite{Rohrlich} 
under the condition \ref{cond:K2}.
Then $\chi_{can}$ 
is ramified precisely at all $w\mid \mathfrak{d}_{\K/\F}$.
And the restriction of $\chi_{can}$ to $\oo_w^\times$
is the unique character 
of conductor $(\varpi_w)$ 
that extends $\qch_{\K_w/\F_v}$ for $w\mid v$.
It follows from \cite[\S 8]{Rohrlich}
(see also \cite[Lem 2.1]{Rod})
that the root number $W(\chi_{can})=1$
under the condition \ref{cond:K3}.

Therefore the conditions in
\cite[Thm A]{Hsieh2012}
are satisfied for the self-dual character
$\chi_{can}^{-1}$.
Consequently for any split prime $\fl_\circ$
that is prime to $p$, there exists 
a twist of which by a finite-order 
anticyclotomic character
ramified at $\fl_\circ\fl^s_\circ$
for which the central $L$-value is nonzero
(in fact, the algebraic part of the 
the central $L$-value is nonzero modulo $p$).
We take $\chi$
for which $\chi_\circ=\chi|\cdot|^{1/2}_\K$
is such a twist.

Now the last property above follows by construction.
The second follows from that 
$\chi_\circ$ is a finite-order twist of $\chi_{can}^{-1}$,
which has the infinity type $\Sigma^c$.
And the first follows from that 
$\chi_{can}$ restricts to $\qch_{\K/\F}|\cdot|^{-1}_\F$.
\end{proof}

From now on we fix a character $\chi$ as in the Lemma. 
To apply the results in \cite{lee},
we assume that there exists a Hecke character $\mu$
such that $\psi=\mu^c/\mu$ and
\begin{equation}\label{cond:can}\tag{can}
    \chi_\circ\mu \text{ is ramified only at split primes},
\end{equation}
which means that $\mu$ is unramifed at all inert places
and $\mu\vert_{\oo_w^\times}$
is the unique nontrivial quadratic character 
at all ramified places (recall that all ramified places are odd
since we assumed that all places above $2$ split in $\K/\F$).
Write $\fG_\fs=\fG$ when $\fs=\oo_\K$
and let $\fX$ denote the set of characters of $\fG$
that are p-adic avatars of algebraic Hecke characters 
of infinity type $(k+1)\Sigma$ for $k\geq0$. 
Pick any $\widehat{\alpha}\in\fX$,
then by Proposition \ref{prop:single}
we have an algebraic modular form 
$f(\eta)\in S^\ord_{\wt{k}}(K^p\Iw(p^{n,n}),\eo)$
for $\eta=\chi_\circ\mu\alpha$,
where $K\subset G(\A_f)$ is an open compact subgroup 
independent of $\eta$
and $\wt{k}=(0,-k\Sigma)$ if $\alpha$
has the infinity type $(k+1)\Sigma$.
By shrinking $K^p$ if necessary, we may assume that
$K^p$ satisfies \eqref{cond:small} for $S=\emptyset$.

Let $\ff$ be an square-free prime-to-$p$ ideal that is divisible
only by split primes such that 
\begin{itemize}
    \item $\ff+\ff^c=\oo_\K$.
    \item if $w\mid\ff$ then $\iota_w^{-1}(\GL_2(\oo_w))\subset K$.
\end{itemize}
We let $U^p=U_\fs^p\subset K^p$ be the subgroup obtained by 
replacing the component of $K$ at each 
$v\in\finite$ which divdes $\fs\coloneqq \ff\ff^c$
with $\Iw_1(w)$.
By definition $U^p$ satisfies \eqref{cond:s-ram}
if $S$ is the set of places dividing $\fs$.

For $w\in \ff$ and $w\mid v$
let $\Delta_v=\Iw(w)/\Iw_1(w)$
and $\Delta_{\fs}=\prod_{v\mid \ff}\Delta_v$.
We put 
\[
\Lambda=\eo\llbracket T(p^1)\rrbracket,\quad
\Lambda_\fs=\Lambda[\Delta_\fs],\quad
\Lambda_\fs^+=\Lambda_\fs[\Delta_p],\quad
\]
as in \eqref{def:lambda_rings}
and let $\Lambda^+_\fs\to \I_\fs$
be the ring homomorphism induced by 
the map $T(p^0)\times \Delta_{\fs}\to \fG_\fs^a$,
which is the composition of
the projection to the upper-left entry
from $T(p^0)\times \Delta_{\fs}$ to 
$\prod_{w\in \Sigma_p}\oo_w^\times\times 
\prod_{w\mid \ff}\big(\oo_w^\times/(1+\varpi_w\oo_w)\big)$
and the reciprocity map \eqref{eq:anticyc_rec}.
We then let $S^{\ord}(U^p,\I_\fs)$ 
denote the space of $\I_\fs$-adic Hida families
as in Definition \ref{def:Hida_family}.

Let $\fX_\fs$ denote the set of characters of $\fG_\fs$
that are p-adic avatars of algebraic Hecke characters 
of infinity type $(k+1)\Sigma$ for $k\geq 0$.
We can view $\widehat{\alpha}\in\fX_\fs$ as a character of
$\fG_\fs^a$ via the injection $\fG_\fs^a\hookrightarrow\fG_\fs$
in \eqref{eq:anticyc_rec}.
\begin{prop}\cite[Thm. 6.8]{lee}\label{prop:family}
    There exists a Hida family
    $\euF\in S^{\ord}(U^p, \I_\fs)$ such that,
    when viewed as a $S^{\ord}(U^p)$-valued measure on $\fG_\fs^a$,
    \[
    \int_{\fG_\fs^a} \widehat{\alpha}\,d\euF
    =\beta_{\wt{k}}(f(\eta))\quad\text{ for }
    \eta=\chi_\circ\mu\alpha\text{ and }
    \alpha\in \fX_\fs.
    \]
    Here $\wt{k}=(0,-k\Sigma)$ if $\alpha$ has the infinity type
    $(k+1)\Sigma$ and $\beta_{\wt{k}}\colon 
    S^{\ord}_{\wt{k}}(U^p)\to S^{\ord}(U^p)$
    is the injection induced by Proposition \ref{prop:wt_indep},
    using that $\pi_{\wt{k}}^*$ is one-dimenlional in the ordinary case.
    See also \cite[Prop 2.22]{ger}.
\end{prop}

Moreover, let
$\mathbf{B}\colon S^{\ord}(U^p,\I_\fs)\times 
S^{\ord}(U^p,\I_\fs)\to\I_\fs$ 
denote the pairing defined in \cite{lee}
and let $L=\Omega_p^{-2\Sigma}\mathbf{B}(\euF,U_\fs^{-1}\euF)$,
where $U_{\fs}$ be the product of all $U_w^{(1)}$ for $w\mid\ff$.
Since $\chi^{-2}\tilde{\eta}=\psi^{-1}\tilde{\alpha}$
if $\eta=\chi_\circ\mu\alpha$,
we have the following result from 
\cite[Thm 6.11]{lee}.
\begin{prop}\label{prop:L-function_at_s}
Suppose $\widehat{\alpha}\in\fX_\fs$ and $\alpha$
has the infinity type $(k+1)\Sigma$, then
\begin{equation}
        \frac{1}{\Omega_p^{(2k+2)\Sigma}}
        \int_{\fG_{\fs}^a}\widehat{\alpha}\,L=C(\K,\chi)
        \left(\frac{2\pi}{\Omega_\infty}\right)^{(2k+2)\Sigma}
    \textnormal{Im}(\delta)^{k}\tilde{\eta}(z_\delta^{-1})
    \frac{\Gamma((k+2)\Sigma)}{(2\pi)^{(k+2)\Sigma}}
    L(1,\psi^{-1}\tilde{\alpha})
    \prod_{v}E_v(\psi^{-1}\tilde{\alpha}).
\end{equation}
where the product ranges through $v\in\finite$
that divides $p\fs$ or the conductor of $\mu$, and 
$E_v(\psi^{-1}\tilde{\alpha})$ is the modified Euler factor 
\[
\frac{
(1-\psi^{-1}\tilde{\alpha}(\varpi_\bw)q_\bw^{-1})
(1-\psi\tilde{\alpha}^{-1}(\varpi_w))}
{{\varepsilon(1,(\psi^{-1}\tilde{\alpha})_w,\psi_w)}}
\]
Here for each $v\in\finite$ that splits we fix a place $w\mid v$
and assume that $w\in\Sigma_p$ or $w\mid \ff$ if $v\in p\fs$.
\end{prop}
We refer to \textit{loc.cit.} for the precise definition
of the notations.
Here we only note that
$\Omega_p$ and $\Omega_\infty$
are the canonical CM periods associated to $\K$,
$z_\delta$ is an element in $\A_{\K,f}^\times$, and 
$C(\chi,\K)$ is a nonzero 
constant which involves the central $L$-value
of $\chi$.

\begin{defn}\label{def:family_at_s}

From now on we only consider $\fs=\ff\ff^c$ for 
$\ff$ satisfies the conditions listed above.
For such $\fs$, we let $U_\fs^p$
be the open compact subgroup defined above, which satisfies
\eqref{cond:small} and \eqref{cond:s-ram}
if $S$ is the set of places dividing $\fs$.
We then let
\[
    \euF_{\fs}\in 
    S^{\ord}(U^p_{\fs},\I_\fs),\quad
    L_\fs=\Omega_p^{-2\Sigma}
    \mathbf{B}(\euF_{\fs}, U_\fs^{-1}\euF_{\fs})
    \in \I_\fs
\]
be as given in Proposition \ref{prop:family} and
Proposition \ref{prop:L-function_at_s}.
We will write $\euF_\fs=\euF_\id$ and $L_\fs=L_\id$
when $\fs=\oo_\K$.
\end{defn}

\begin{prop}\label{prop:eigensystem}
The Hida familiy $\euF_\fs$ induces an eigensystem
$\lambda_{\fs}^{\ord}\colon \TT^{\ord}(U^p_\fs,\eo)\to\I_\fs$ with
\begin{align*}
&T_w^{(1)}\mapsto
q_w\chi_\circ(\varpi_w)(1+\Psi_\fs(\varpi_w))&
&T_w^{(2)}\mapsto q_w\chi_\circ^2\Psi_\fs(\varpi_w)\\
&U_{w}^{(1)}\mapsto
\epsilon^{-1}\chi_\circ\Psi_\fs(\varpi_w)&
&U_{w}^{(2)}\mapsto
\epsilon^{-1}\chi_\circ^2\Psi_\fs(\varpi_w)&
&\langle u\rangle\mapsto
\epsilon^{-1}\chi_\circ^2\Psi_\fs(u_{11})\\
&U_{w}^{(1)}\mapsto q_w\chi_\circ\Psi_\fs(\varpi_w)&
&U_{w}^{(2)}\mapsto q_w\chi_\circ^2\Psi_\fs(\varpi_w)&
&\langle u\rangle\mapsto \Psi_\fs(u_{11})
\end{align*}
Here we identify $\Psi_\fs$ with a character of 
$\A_{\K}^\times/\K^\times$ by composition with the 
reciprocity map.
\end{prop}
\begin{proof}
Since $\widehat{\alpha}\circ\Psi_\fs(\Art(\varpi))=
\psi^{-1}\tilde{\alpha}^\wedge(\varpi_w)$,
the proposition is a consequence of Proposition \ref{prop:single}
and that 
$\tilde{\eta}=\chi^2\psi^{-1}\tilde{\alpha}
=\chi_\circ^2|\cdot|^{-1}\psi^{-1}\tilde{\alpha}$.
\end{proof}

\begin{cor}\label{cor:reducible}
Let $T^{\ord}=T^{\ord}(U^p_\fs)
\colon\Gal_\K\to \TT^{\ord}(U^p_\fs,\eo)$
be the big ordinary Galois pseudo-representation
as in Definition \ref{def:big_Gal}, then we have
\begin{equation*}
    \lambda_{\fs}^{\ord}\circ T^{\ord}=
    \epsilon^{-1}\hat{\chi}_\circ+
    \epsilon^{-1}\hat{\chi}_\circ\Psi_\fs,\quad
    \lambda_{\fs}^{\ord}\circ (\epsilon \det T^{\ord})=
    \epsilon^{-1}\hat{\chi}^2_\circ\Psi_\fs.
\end{equation*}
\end{cor}
\begin{proof}
The claim is immediate from Chebotarev's density
and the formulae in the previous proposition.
\end{proof}

\subsection{Main construction}

We now apply the formulation in \S\ref{sub:fund_exact_sequence}
to the eigensystem $\lambda^{\ord}_\fs$. Let
\[
    \lambda_\fs\colon \TT(U^p,\eo)\to\TT^{\ord}(U^p,\eo)\to\I_\fs
\]
denote the composition with $\lambda^{\ord}_\fs$,
$\fm\subset \TT(U^p_\fs,\eo)$ and
$\fm_1\subset \TT^\ord(U^p_\fs,\eo)$
denote the pre-images of the maximal ideal $\fm_{\I_\fs}$ of $\I_\fs$,
and let $\bar{\delta_1}, \bar{\delta_2}$
denote $\epsilon^{-1}\hat{\chi}_\circ\bmod \fm_{\I_\fs}$ 
and $\epsilon^{-1}\hat{\chi}_\circ\Psi_\fs\bmod \fm_{\I_\fs}$ 
respectively. We then have
\begin{itemize}
\item $T(U^p_\fs)\equiv\bar{\delta}_1+\bar{\delta}_2\bmod \fm$.
\item The maximal ideal $\fm$ satisfies \eqref{cond:red_gen}
because of \eqref{cond:gen_psi}.
\item The map $\lambda_\fs\colon \TT(U^p_\fs,\eo)\to \I_\fs$
is surjective by Chebotarev's density
since $\lambda_\fs(T_w^{(2)})=q_w\chi_0^2\Psi_\fs(\varpi_w)$
for $v=w\bw$ in \eqref{def:hecke_away_p},
which form a set of density one.
\end{itemize}

\begin{defn}
Let $M^{\ord}(U^p_\fs,\I_\fs)\cong\Hom_{\I_\fs}
(S^{\ord}(U^p_\fs,\I_\fs),\I_\fs)$
be the isomorphism from Proposition \ref{prop:ord_to_dual}.
We let $F_\fs\in M^{\ord}(U^p_\fs,\I_\fs)$
denote the pre-image of
$\Omega_p^{-2\Sigma}\mathbf{B}(*,U_\fs^{-1}\euF_{\fs})$
under the isomorphism and define
\[
    \Theta_\fs=ev(\euF_\fs)\colon 
    M^{\ord}(U^p_\fs,\I_\fs)\cong\Hom_{\I_\fs}
    (S^{\ord}(U^p_\fs,\I_\fs),\I_\fs)\to\I_\fs.
\]
In particular we have $\Theta_\fs(F_\fs)=
\Omega_p^{-2\Sigma}\mathbf{B}(\euF_\fs,U_\fs^{-1}\euF_\fs)=L_\fs$.
\end{defn}

\begin{lem}
The choice of $\Theta_\fs$ and $F_\fs$ above
satisfies \ref{cond:C1} and \ref{cond:C3}.
\end{lem}
\begin{proof}
That $\Theta_\fs$ satisfies \ref{cond:C1} 
follows from that $\euF_\fs$ is an eigenform
with eigensystem $\lambda_\fs^{\ord}$
and that the pairing $\mathbf{B}$ is Hecke-equivariant
by \cite[Rmk 5.13]{lee}.
Decompose $\fG_\fs^a=\Delta_\fs^a\times W$
into the finite part and the free part.
The image of $L_\fs=\Theta_\fs(F_\fs)$
under $\alpha\colon \I_\fs\to \eo\llbracket W\rrbracket$,
for $\alpha\in\hat{\Delta}_\fs$,
is a $p$-adic $L$-function 
for the character $\epsilon\psi^{-1}\alpha$.
It is shown in \cite{Hida10}
that the $p$-adic $L$-function for such characters,
with the conductor only divisible by split primes,
is nonzero.
From which we conclude that \ref{cond:C3}
is also satisfied.
\end{proof}

\subsubsection{Residually reducible pseudo-representations}

We briefly recall the process of constructing
cohomology classes from a generically irreducible
pseudo-representation form \cite[\S 2.1]{urban}.
Suppose $T\colon \mathcal{G}\to R$
is a two-dimensional pseudo-representation
into a Henselian local ring $R$
with maximal ideal $\fm_R$
and of odd residual characteristic.
We assume there exists distinct characters
$\bar{\delta}_i\colon \mathcal{G}\to R/\fm_R$ of $\mathcal{G}$
for $i=1,2$ such that  
$T\equiv \bar{\delta}_1+\bar{\delta}_2\bmod \fm_R$.
After fixing $z\in \mathcal{G}$
with $\bar{\delta}_1(z)\neq \bar{\delta}_2(z)$
we apply the Henselian property on
\begin{equation*}
    P(z,X)=
    X^2-T(z)X+\det(T)(z) \equiv 
    (X-\bar{\delta}_1(z))(X-\bar{\delta}_2(z))
    \mod \fm_R
\end{equation*}
and lift the distinct roots $\bar{\delta}_i(z)$
to roots $\alpha,\beta$ of $P(z,X)$
with $\alpha-\beta\in R^\times$.
We then define the functions
\begin{equation}\label{eq:presentation}
   A(\sigma)=
   \frac{T(\sigma z)-\beta T(\sigma)}{\alpha-\beta}\quad
   D(\sigma)=
   \frac{T(\sigma z)-\alpha T(\sigma)}{\beta-\alpha}\quad
   x(\sigma,\tau)=a(\sigma\tau)-a(\sigma)a(\tau).
\end{equation}
Then 
$A(\sigma)\equiv \bar{\delta}_1(\sigma)$,
$D(\sigma)\equiv \bar{\delta}_2(\sigma)\bmod\fm_R$,
and $x(\sigma,\tau)$
generate the reducibility ideal of $T$.

Moreover when there exists $\sigma_0,\tau_0$ such that
$x=x(\sigma_0,\tau_0)\in R$ is not a zero divisor
\begin{equation}\label{eq:present_repn}
    \sigma\mapsto 
    \begin{pmatrix}
        A(\sigma)& B(\sigma)\\
        C(\sigma) & D(\sigma)
    \end{pmatrix}\coloneqq
    \begin{pmatrix}
        A(\sigma)& x(\sigma,\tau_0)/x(\sigma_0,\tau_0)\\
        x(\sigma_0,\sigma) & D(\sigma)
    \end{pmatrix}
        \in \GL_2(R[1/x])
\end{equation}
defines a group representation 
and satisfies $x(\sigma,\tau)=B(\sigma)C(\tau)$.
Let $B\subset R[1/x]$ and $C\subset R$
be the $R$-submodules generated respectively by
$B(\sigma)$ and $C(\sigma)$ for $\sigma\in\mathcal{G}$.
The following proposition is a restatement of 
\cite[Thm 1.5.5]{BC} in this case.
\begin{prop}\label{prop:BC}
Let $S$ be an $R$-algebra with structure 
homomorphism $\phi\colon R\to S$.
Suppose the reducibility ideal 
$I$ of $T$ is contained $\ker(\phi)$,
so $\delta_1\coloneqq \phi\circ A(\sigma)$
and $\delta_2\coloneqq \phi\circ D(\sigma)$
become characters.
Write $\delta=\delta_1\delta_2^{-1}$, then
there exists injections 
\begin{align*}
    \Hom_R(B,S)&\hookrightarrow H^1(\mathcal{G},S(\delta))\quad
    f\mapsto \delta_2(\sigma)^{-1}f(B(\sigma))\\
    \Hom_R(C,S)&\hookrightarrow H^1(\mathcal{G},S(\delta^{-1}))\quad
    f\mapsto \delta_1(\tau)^{-1}f(C(\tau))
\end{align*}
\end{prop}

\begin{rem}
The idea of constructing nontrivial cohomology classes
from generically irreducible deformations
dates back to the proof of the converse of
Herbrand-Ribet theorem in \cite{Ribet1976}
and is further investigated in \cite{Urban1999}.
It is now commonly refered as Urban's lattice construction
in applications to Iwasawa theory.
\end{rem}

\subsubsection{Condition C3}

Write $T_\fm=T(U^p_\fs)_\fm$ and put
$T_\fm'\coloneqq \xi_\fm^{1/2}\epsilon T_\fm$, so that
by Corollary \ref{cor:reducible}
\[
    \lambda_\fs\circ T'_\fm=
    \delta_1'+\delta_2'\quad
    \delta_1'\coloneqq 
    (\lambda_\fs\circ\xi_\fm^{1/2})\hat{\chi}_0,
    \delta_2'\coloneqq 
    (\lambda_\fs\circ\xi_\fm^{1/2})\hat{\chi}_0\Psi_\fs
\]
Note that $\bar{\delta}'_i\coloneqq\delta'_i\bmod \fm_{\I_\fs}
=\omega\delta_i$ and $\chi_i=\delta'_i\vert_{\Gp}$ for $i=1,2$.
Then the homomorphism 
$R^{\zeta\epsilon}\to \TT(U^p_\fs,\eo)_\fm$
from Proposition \ref{lem:twist}
push $T^{\zeta\epsilon}$, the universal deformation 
of $\chi_1+\chi_2$ with determinant $\zeta\epsilon$,
to $T'_\fm\vert_{\Gp}$.

Since $\fm$ satisfies \eqref{cond:red_gen},
there exists $z\in\Gp$ such that 
${\chi}_1(z)\neq{\chi}_2(z)$.
Let $\alpha_0,\beta_0$ the roots of $P(z,X)$ on $T^{\epsilon\zeta}$.
We then construct the $R^{\zeta\epsilon}$-valued functions
$A_0(\sigma),D_0(\sigma)$, and $x_0(\sigma,\tau)$ 
on $\Gp$ by \eqref{eq:presentation}.
Moreover, since $R^{\zeta\epsilon}$
is isomorphic to a power series ring of 3 variables
and the reducibility ideal is generated by a regular element,
there exists $\sigma_0,\tau_0\in\Gp$ such that
$\xx=x(\sigma_0,\tau_0)$ is a 
generator of the reducibility ideal.
Thus if we define the functions $B_0(\sigma)$ and $C_0(\sigma)$ 
by \eqref{eq:present_repn},
then the module $B_0$ generated by all $B_0(\sigma)$
is $R^{\zeta\epsilon}$.

Let $\alpha,\beta\in \TT(U^p_\fs,\eo)_\fm$
denote the images of the roots $\alpha_0,\beta_0$
under the homomorphism $R^{\zeta\epsilon}\to \TT(U^p_\fs,\eo)_\fm$.
We can define the $\TT(U^p_\fs,\eo)_\fm$-valued functions
$A(\sigma), D(\sigma)$, and $x(\sigma,\tau)$ on $\Gal_\K$
by \eqref{eq:presentation} so that $T'_\fm=A+D$.
Note that the restrictions of which to 
$\Gp$ coincides with the images of 
$A_0(\sigma), D_0(\sigma)$, and $x_0(\sigma,\tau)$.
Since $\xx=x(\sigma_0,\tau_0)$ is not a zero divisor
by Corollary \ref{cor:Hecke_ff}, we can define
the functions $B$ and $C$ by \eqref{eq:present_repn}.
Again their restrictions to $\Gp$ coincides with 
the images of $B_0(\sigma)$ and $C_0(\sigma)$.

\begin{defn}\label{def:fq_ideal}
Let $B\subset\TT(U^p_\fs,\eo)_\fm[1/\xx]$
and $C\subset\TT(U^p_\fs,\eo)_\fm$
be the $\TT(U^p_\fs,\eo)_\fm$-submodules
generated by $B(\sigma)$ and $C(\sigma)$ 
for $\sigma\in\Gal_\K$. We define the ideal
\[
    \fq_\fs=(x(\sigma,\tau_0)=B(\sigma)C(\tau_0)=B(\sigma)\xx
    \mid\sigma\in\Gal_\K)\subset \TT(U^p_\fs,\eo)_\fm.
\]
\end{defn}

\begin{lem}
For $w=w_0$ and $i=1,2$ let
$x_i(\sigma,\tau)$ be the images of $x(\sigma,\tau)$ in
$\TT^{\ord}(U^p_\fs,\eo)_{\fm_i}$, then 
\begin{itemize}
    \item $x_1(\sigma,\tau)=0=x_2(\tau, \sigma)$ 
    for $\sigma\in D_{w}$ and $\tau\in\Gal_\K$.
    \item $x_1(\sigma,\tau)=0=x_2(\tau,\sigma)$ 
    for $\sigma\in \Gal_\K$ and $\tau\in D_{\bw}$.
\end{itemize}
\end{lem}

\begin{proof}

By the density result from Corollary \ref{cor:density},
it suffices to show the equalities in $E_\fp$
for minimal primes
$\fp\subset \TT_{\wt{k}}(U^p_\fs\Iw(p^{0,1}),E)_{\fm_i}$.
Let $A_\fp(\sigma), D_\fp(\sigma)$, and $x_\fp(\sigma,\tau)$
denote the images of the function in $E_\fp$,
then we have $A_\fp+D_\fp=(\xi_\fp^{1/2}\epsilon)\mtr(r_\fp)$.
where $\xi_\fp^{1/2}$ denote the image 
of the character $\xi_\fm^{1/2}$ in $E_\fp$.

On the other hand, 
let $\Psi_{w,i}\colon D_{w}\to \TT(U^p_\fs,\eo)_{\fm_i}$,
$i=1,2$, be the characters defined in Proposition 
\ref{prop:big_char_at_p}
and let $\psi_{1},\psi_2$ be the respective images in $E_\fp$,
so that $\mtr(r_\fp)\vert_{D_{w}}=\psi_1+\epsilon^{-1}\psi_2$.
Then Proposition \ref{prop:eigensystem}
and the characterization of $\fm_i$ in 
Corollary \ref{cor:two_max} together implies that
\begin{align}
\xi_\fp^{1/2}\psi_2\equiv\psi_2
\equiv\Psi_{w,2}\equiv 
\chi_1\coloneqq 
\omega\bar{\delta}_1\vert_{D_{w}}\equiv A_\fp\vert_{D_{w}}
\text{ if }
\fp\subset \TT(U^p_\fs\Iw(p^{0,1}),\eo)_{\fm_1}
\label{eq:con1}\\
\xi_\fp^{1/2}\psi_2\equiv\psi_2
\equiv\Psi_{w,2}\equiv 
\chi_2\coloneqq 
\omega\bar{\delta}_1\vert_{D_{w}}\equiv D_\fp\vert_{D_{w}}
\text{ if }
\fp\subset \TT(U^p_\fs\Iw(p^{0,1}),\eo)_{\fm_2}
\end{align}
Focus on the case when
$\fp\subset \TT(U^p_\fs\Iw(p^{0,1}),\eo)_{\fm_1}$,
then \eqref{cond:gen_psi} and \eqref{eq:con1}
implies that 
$A_\fp\vert_{D_{w}}=\xi_\fp^{1/2}\psi_2$ and
$D_\fp\vert_{D_{w}}=\xi_\fp^{1/2}\epsilon\psi_1$.
The choice of $z\in\Gp=D_w$
determines a basis for $r_\fp$ such that
\[
    r'_\fp(\sigma)\coloneqq
    (\xi_\fp^{1/2}\epsilon)r_\fp(\sigma)=
    \begin{pmatrix}
        A_\fp(\sigma) & B_\fp(\sigma)\\
        C_\fp(\sigma) & D_\fp(\sigma)
    \end{pmatrix}
    \in \GL_2(E_\fp) \text{ and }
    r'_\fp(z)=
    \begin{pmatrix}
        \alpha_\fp & \\
        & \beta_\fp
    \end{pmatrix}
\]
which satisfies $x_\fp(\sigma,\tau)=B_\fp(\sigma)C_\fp(\tau)$.
Here $\alpha_\fp$ and $\beta_\fp$
are the images of $\alpha$ and $\beta$ in $E_\fp$.
But Lemma \ref{lem:galois_at_p}
implies that $r'_\fp\vert_{D_{w}}$
has the character $D_\fp\vert_{D_{w}}$
as a subrepresentation.
Since $\beta_\fp=D_\fp(z)$,
the shape of $r'_\fp(z)$ implies that $(0,1)^\intercal$ is
the $D_\fp\vert_{D_w}$-eigenvector.
Therefore $B_\fp(\sigma)=0$
and consequently $x_\fp(\sigma,\tau)=0$ if $\sigma\in D_w$.

The case when $\tau\in D_{\bw}$
can be proved similarly,
using that $r_\fp^c\cong r_\fp^\vee\epsilon^{-1}$
has $\psi_2^{-1}$ has a subrepresentation.
And the case when 
$\fp\subset \TT(U^p_\fs\Iw(p^{0,1}),\eo)_{\fm_2}$
is the same except that the roles
of $A_\fp$ and $D_\fp$ are switched.

\end{proof}

\begin{prop}
The ideal  $q_{\fs}$
satisfies the condition \ref{cond:C2}.
\end{prop}

\begin{proof}

By definition we have $\xx=x(\sigma_0,\tau_0)\in \fq_\fs$.
And since the action of $\TT(U^p_\fs,\eo)_\fm$
on $F_\fs$ factors through $\lambda_\fs$
and $\lambda_\fs\circ T'_\fm$ is reducible 
by Corollay \ref{cor:reducible},
we have $\fq_\fs\subset \ker(\lambda_\fs)$ and 
therefore $\fq_\fs F_\fs=0$.
At last, the lemma above
implies that $x_2(\sigma, \tau_0)=0$
for any $\sigma\in\Gal_\K$
and consequently
$\fq_\fs M^{\ord}(U^p_\fs)_{\fm_2}=0$.
\end{proof}

We can now apply Corollary \ref{cor:fund}
and obtain the commutative diagram
\begin{equation}\label{eq:comm_s}
    \begin{tikzcd}
    & (M_\fs/S_\fs)
    \arrow[r]\arrow[d,equal] &
    (M_\fs/S_\fs)\otimes_{\TT}\fq_\fs
    \arrow[r,"q"]\arrow[d,"p"] &
    (M_\fs/S_\fs)\otimes_{\TT}\fq_\fs^\red
    \arrow[r]\arrow[d,"s"] &0\\
    0\arrow[r] &
    (M_\fs/S_\fs)
    \arrow[r] &
    \fq_\fs M_\fs/\fq_\fs S_\fs
    \arrow[r,"r"] &
    \fq_\fs M^{\ord}_{\fs}/\fq_\fs S^{\ord}_{\fs}
    \arrow[r] &0
    \end{tikzcd}
\end{equation}
Here
$M^{\ord}_{\fs}=M^{\ord}(U^p_\fs,\I_\fs)_{\fm_1}$,
$S^{\ord}_{\fs}=\ker(\Theta_\fs)$,
$M_\fs=\Hom_{Q'}(\tilde{P}_{2,\fm},M(U^p_\fs)_{\fm}')
\otimes_{\Lambda_{\fs}^+}\I_\fs$,
$S_\fs$ the pre-image of $S^{\ord}_\fs$
under $M_\fs\to M^{\ord}_\fs$,
$\TT=\TT(U^p_\fs,\eo)_\fm$, and
$\fq_\fs^\red\coloneqq\fq_\fs\TT^{\red}$.

\begin{defn}\label{def:cong_map}
Let $\tilde{F}_\fs$ be the pre-image of
$(F_\fs\bmod S_\fs^{\ord})$ under 
$M_\fs/S_\fs\cong M^{\ord}_{\fs}/S^{\ord}_{\fs}$.
We define a map  $f_\fs\in \Hom_{\TT}(B,\I_\fs)$ 
as follows:
for $b\in B$ write $q(b)=b\xx=bC(\tau_0)\in\fq_\fs$.
We define
\begin{align*}
   &y_\fs'(b)=
(\tilde{F}_\fs\bmod S_\fs)\otimes q(b)\in 
(M_\fs/S_\fs)\otimes_\TT\fq_\fs\\
   &y_\fs(b)=p(y_\fs'(b))=
q(b)\tilde{F}_\fs\bmod \fq_\fs S_\fs\in 
\fq_\fs M_\fs/\fq_\fs S_\fs
\end{align*}
On the other hand, observe that 
$(M_\fs/S_\fs)\otimes_{\TT}\fq_\fs^\red=
(M_\fs^{\red}/S_\fs^{\red})\otimes_{\TT}\fq_\fs^\red=
(M_\fs^{\ord}/S_\fs^{\ord})\otimes_{\TT}\fq_\fs^\red$
and therefore
\begin{align*}
   q(y_\fs'(b))&=
(\tilde{F}_\fs\bmod S_\fs)\otimes q(b)\in 
(M^{\ord}_\fs/S^{\ord}_\fs)\otimes_\TT\fq_\fs^\red\\
   r(y_\fs(b))&=(s\circ q)(y_\fs'(b))=
q(b)F_\fs\bmod \fq_\fs S^{\ord}_\fs\in 
\fq_\fs M^{\ord}_\fs/\fq_\fs S^{\ord}_\fs
\end{align*}
But $q(b){F}_\fs\in \fq_\fs {F}_\fs=0$ 
because of \ref{cond:C2}.
By the exactness of the sequence
there exists $\tilde{F}_\fs(b)\in M_\fs/ S_\fs$
such that 
$\xx \tilde{F}_\fs(b)\equiv q(b)\tilde{F}_\fs
\bmod \fq_\fs S_\fs$.
We now define 
$f_\fs(b)=\Theta_\fs(\tilde{F}_\fs(b))\in \I_\fs$.

\end{defn}

\begin{prop}\label{prop:res}
For $w'\in\Sigma_p\setminus\{w_0\}$
and $\sigma_1,\sigma_2\in D_{w'}$ we have
\[
(A(\sigma_1)-D(\sigma_1))B(\sigma_2)=
(A(\sigma_2)-D(\sigma_2))B(\sigma_1)
\]
\end{prop}
\begin{proof}

Follow the formulation in \S\ref{sub:fund_exact_sequence},
we decompose $\fG_\fs^a=W\times\Delta_\fs^a$.
And for each $\alpha\in\widehat{\Delta}_\fs^a$ 
we have an induced homomorphism
$\alpha\colon \I_\fs\to \Lambda_\alpha$,
where $\Lambda_\alpha$ is $\eo\llbracket W\rrbracket$
with an $I_\fs$-algebra structure on which $\Delta$ acts by $\alpha$.
Let $K_\alpha$ be the field of fraction of $\Lambda_\alpha$.
Since $\I_\fs\hookrightarrow 
\prod_{\alpha\in\hat{\Delta}}K_\alpha$,
it suffices to verify the equality for each 
$\alpha\circ f\in \Hom_{\TT}(B,K_\alpha)$.

Write $\lambda_\alpha=\alpha\circ\lambda_\fs
\colon\TT\to K_\alpha$
and define the prime ideals $\wp_1=\ker(\lambda_\alpha)$,
$\fp_1=\wp_1\cap\Lambda\llbracket\xx\rrbracket$,
and $\fp=\fp_1\cap \Lambda$.
As in the proof of \ref{lem:smallprime},
we write $\Lambda\llbracket\xx\rrbracket_{\fp_1}$
for the completion of the localization at $\fp_1$
and refer localizations at $\fp_1$ as
$\otimes_{\Lambda\llbracket\xx\rrbracket}
\Lambda\llbracket\xx\rrbracket_{\fp_1}$.
Since $(\TT_{\fm})_{\fp_1}$ is finite free
over $\Lambda\llbracket\xx\rrbracket_{\fp_1}$
by Corollary \ref{cor:Hecke_ff},
it is complete and therefore isomorphic
to the direct product of its localizations 
at prime ideals.
In particular the component that corresponds to 
$\lambda_\alpha\colon (\TT)_{\fp_1}\to K_\alpha$
is none other than $\TT_{\wp_1}$,
the completion of the localizations at $\wp_1$.
From this we conclude that
$\Hom_{\TT}(B,K_\alpha)=
\Hom_{(\TT)_{\fp_1}}(B_{\fp_1},K_\alpha)=
\Hom_{\TT_{\wp_1}}(B_{\wp_1},K_\alpha)$.

Observe that $\TT_{\wp_1}$ is reduced by Lemma \ref{lem:big_red}
and also finite free over over 
$\Lambda\llbracket\xx\rrbracket_{\fp_1}$.
Therefore $\TT_{\wp_1}[1/\xx]$ is reduced and Noetherian.
So when $\wp$ ranges over minimal primes of which we have the diagram
\[
\begin{tikzcd}
\TT_{\wp_1}[1/\xx] \arrow[r,hookrightarrow] &
\prod_{\wp} \TT_{\wp_1}[1/\xx]/\wp \\
\TT_{\wp_1} \arrow[u,hookrightarrow] \arrow[r,hookrightarrow] &
\prod_{\wp} \TT_{\wp_1}/\wp\cap \TT_{\wp_1}\arrow[u,hookrightarrow] 
\end{tikzcd}
\]
On the other hand recall $B\subset \TT[1/\xx]$,
thus $B_{\fp_1}\subset (\TT)_{\fp_1}[1/\xx]$ and 
consequently $B_{\wp_1}\subset \TT_{\wp_1}[1/\xx]$.
Therefore we can further reduce to the cases
in which we replace $B_{\wp_1}$
by its images in $\TT_{\wp_1}[1/\xx]/\wp$.

Replace $T_\fm', A(\sigma),B(\sigma),C(\sigma),D(\sigma)$
by their images in $\TT_{\wp_1}[1/\xx]/\wp$.
By Proposition \ref{prop:big_char_at_p},
$T_\fm'\vert_{D_w'}=\Psi_1'+\epsilon^{-1}\Psi_2'$ is reducible
and $\Psi_1'$-ordinary, where
$\Psi_i'$ denotes the image of 
$\xi_\fm^{1/2}\epsilon\Psi_{w',i}$ in $\TT_{\wp_1}[1/\xx]/\wp$ 
for $i=1,2$.
Observe that every functions in 
$A\vert_{D_{w'}}+D\vert_{D_{w'}}=T'_\fm\vert_{D_{w'}}=
\Psi_1'+\epsilon^{-1}\Psi_2'$
takes value in the integral domain 
$\TT_{\wp_1}/\wp\cap \TT_{\fp_1}$.
And since Proposition \ref{prop:eigensystem} implies that
\begin{align*}
    &\lambda_\alpha\circ A\vert_{D_{w'}}=
    \delta_1'\vert_{D_{w'}}=
    (\lambda_\alpha\circ\xi_\fm^{1/2})\hat{\chi}_\circ=
    \lambda_\alpha\circ (\xi_\fm^{1/2}\epsilon\Psi_{w',2})\\
    &\lambda_\alpha\circ D\vert_{D_{w'}}=
    \delta_2'\vert_{D_{w'}}=
    (\lambda_\alpha\circ\xi_\fm^{1/2})\hat{\chi}_\circ
    (\alpha\circ\Psi_\fs)=
    \lambda_\alpha\circ (\xi_\fm^{1/2}\epsilon\Psi_{w',1})
\end{align*}
for $\lambda_\alpha\colon \TT_{\wp_1}\to K_\alpha$,
we consequently have
$A(\sigma)=\epsilon^{-1}\Psi_2'(\sigma)$ and
$D(\sigma)=\Psi_1'(\sigma)$ for $\sigma\in D_{w'}$.
We can now apply \cite[Lem 5.3.3]{pan} to the representation below,
which is irreducible since $\xx$ is nonzero.
\[
    \rho(\sigma)\coloneqq
    \begin{pmatrix}
        A(\sigma)& B(\sigma)\\
        C(\sigma) & D(\sigma)
    \end{pmatrix}
    \in \GL_2(\TT_{\wp_1}[1/\xx]/\wp)
\]
Then there exists a nonzero vector $(X,Y)$
over some extension of 
the field $\textnormal{Frac}(\TT_{\wp_1}[1/\xx]/\wp)$
which is an $\Psi_1'(\sigma)$-eigenvector 
under $\rho(\sigma)$ for all $\sigma\in D_{w'}$.
In other word,
\[
    \begin{pmatrix}
        A(\sigma)-D(\sigma) & B(\sigma)\\
        C(\sigma) &0
    \end{pmatrix}
    \begin{pmatrix}
        X\\Y
    \end{pmatrix}=0
    \quad \text{ for all } \sigma\in D_{w'}
\]
There can be two possibilities,
and each gives the desired equality.
\begin{itemize}
    \item If $X=0$, then $Y\neq 0$ and consequently
    $B(\sigma)=0$ for all $\sigma\in D_{w'}$.
    \item If $X\neq 0$, then $XC(\sigma)=0$
    implies $C(\sigma)=0$ for all $\sigma\in D_{w'}$.
    Then the proof of \textit{loc.cit.} shows that 
    \[
        (A(\sigma_1)-D(\sigma_1))B(\sigma_2)=
        (A(\sigma_2)-D(\sigma_2))B(\sigma_1)\quad
        \text{ for all  }\sigma_1,\sigma_2\in D_{w'}.
    \]
\end{itemize}

\end{proof}

\begin{prop}\label{prop:res_0}

Let $f_0=R^{\zeta\epsilon}\to \I_\fs$ 
be the composition of $\lambda_\fs$ with 
$R^{\zeta\epsilon}\to \TT$
and recall that $B_0=R^{\zeta\epsilon}$.
Then the composition of $f_\fs$ with
$B_0\to B$
is $L_\fs\cdot f_0$.
\end{prop}
\begin{proof}
Observe that if $b\in\TT$, then
$\tilde{F}_\fs(b)=b\tilde{F}_\fs$ by definition
and consequently 
$f_\fs(b)=\Theta_\fs(b\tilde{F}_\fs)=\lambda_{\fs}(b)\cdot L_\fs$.
The proposition then follows since the image of $B_0=R^{\zeta\epsilon}$
belongs to $\TT$.

\end{proof}

\subsection{Compatibility}

Let $\ff$ and $\fs=\ff\ff^c$ be as in 
Definition \ref{def:family_at_s}
and let $\fl$ be a prime such that $\fl\neq\fl^c$ and
$\iota_\fl^{-1}(\GL_2(\oo_\fl))\subset U_\fs^p$.
Write $\ell=\fl\fl^c$, then both $\ff$ and $\fl\ff$
satisfies the conditions in Definition \ref{def:family_at_s}.
Let $\phi^{\ell\fs}_\fs\colon \I_{\ell\fs}\to \I_\fs$
be the ring homomorphism induced by $\fG_{\ell\fs}^a\to \fG_\fs^a$,
which satisfies $\phi^{\ell\fs}_\fs\circ \Psi_{\ell\fs}=\Psi_{\fs}$.

Since $U^p_{\ell\fs}\subset U^p_{\fs}$,
we have the natural map
$\varphi^{\ell\fs}_{\fs}\colon 
\TT(U^p_{\ell\fs},\eo)\to \TT(U^p_{\fs},\eo)$
that pushes $T(U^p_{\ell\fs})$ to $T(U^p_{\fs})$.
Then by Proposition \ref{prop:eigensystem}
we have the commutative diagram
\[
\begin{tikzcd}
	& \TT(U^p_{\fs\ell},\eo)_\fm
    \arrow[d,"\varphi^{\ell\fs}_\fs"] \arrow[r,"\lambda_{\fs\ell}"]
    &\I_{\ell\fs} \arrow[d,"\phi^{\ell\fs}_\fs"]\\
	R^{\zeta\epsilon}\arrow[r]\arrow[ur]
	& \TT(U^p_\fs,\eo)_\fm \arrow[r,"\lambda_\fs"] &
    \I_\fs
\end{tikzcd}
\]
Here we abuse the notation
and let $\fm\subset \TT(U^p_\bullet,\eo)$
denote the respective maximal ideals 
given by $\lambda_\bullet^{-1}(\fm_{\I_{\bullet}})$.
Therefore the choice of $z$ and the roots of $P(z,T)$ 
on $T^{\epsilon\zeta}$ defines the functions
$A_\bullet(\sigma),B_\bullet(\sigma),
C_\bullet(\sigma)$ and $D_\bullet(\sigma)$
valued in $\TT(U^p_\bullet,\eo)[1/\xx]$ compatibly.
In other word we have
$\varphi^{\ell\fs}_\fs\circ A_{\ell\fs}(\sigma)=
A_\fs(\sigma)$ and similarly for the rest of the functions.

On the other hand, recall that 
$F_{\fs}\in M^{\ord}(U^p_\fs,\I_\fs)\cong
\Hom_{\I_\fs}(S^{\ord}(U^p_\fs,\I_\fs),\I_\fs)$
is the homomorphism that sends
$\EuScript{G}_\fs\in S^{\ord}(U^p_\fs,\I_\fs)$
to $\Omega_p^{-2\Sigma}\mathbf{B}(\EuScript{G}_\fs, U_\fs^{-1}\euF_{\fs})$.
Let $\overline{F}_{\ell\fs}\in M^{\ord}(U^p_{\ell\fs},\I_\fs)$
denote the image of $F_{\ell\fs}$ under $\phi^{\ell\fs}_\fs$.
We consider its images in $M^{\ord}(U^p_\fs,\I_\fs)$
under the maps in the diagram below.
Here we define the map
$V_\fl\colon S(U^p_\fs,A)\to S(U^p_{\fs\ell},A)$
between algebraic modular forms by
$V_\fl f(g)=f(g\iota_\fl^{-1}(\smat{1&\\&\varpi_\fl}))$
and we use the same notation to denote the induced map on
$S^{\ord}(U^p_\fs,\I_\fs)$ and $S(U^p_\fs)$.

\[
\begin{tikzcd}
M^{\ord}(U^p_{\ell\fs},\I_\fs) 
\arrow[d, shift right=6.5ex, "V_\fl^*",swap]
\arrow[d, shift right=5.5ex, "\id^*"]
\times  S^{\ord}(U^p_{\ell\fs},\I_\fs) 
\arrow[r] &
\I_\fs \arrow[d,equal]\\
M^{\ord}(U^p_{\ell\fs},\I_\fs) \times
S^{\ord}(U^p_{\ell\fs},\I_\fs)
\arrow[u, shift right=7.5ex, "V_\fl"]
\arrow[u, shift right=8.5ex, "\id",swap]
\arrow[r] & \I_\fs
\end{tikzcd}
\]

\begin{lem}
Suppose $\EuScript{G}_{\fs}\in 
S^{\ord}(U^p_{\fs},\I_{\fs})$,
then the images 
$\id^*\overline{F}_{\ell\fs}$ and
$V_\fl^*\overline{F}_{\ell\fs}$ satisfies
\begin{align*}
\id^*\overline{F}_{\ell\fs}(\EuScript{G}_{\fs})&
\coloneqq \overline{F}_{\ell\fs}(\EuScript{G}_{\fs})=
(1-\epsilon\Psi_{\fs}(\varpi_{\bar{\fl}}))
F_{\fs}(\EuScript{G}_{\fs}) \\
V_\fl^*\overline{F}_{\ell\fs}(\EuScript{G}_{\fs})&
\coloneqq \overline{F}_{\ell\fs}(V_\fl\EuScript{G}_{\fs})=
(\epsilon^{-1}\chi_\circ)(\varpi_\fl)\cdot 
(1-\epsilon\Psi_{\fs}(\varpi_{\bar{\fl}}))
F_{\fs}(\EuScript{G}_{\fs})
\end{align*}
\end{lem}

\begin{proof}

To distinguish the level, 
we let $\mathbf{B}_\fs=\mathbf{B}$ denote 
the usual pairing on $S^{\ord}(U^p_\fs, \I_\fs)$
while $\mathbf{B}_{\ell\fs}=\phi^{\ell\fs}_\fs\circ \mathbf{B}$
denote the images of the pairing on 
$S^{\ord}(U^p_{\ell\fs}, \I_{\ell\fs})$ under $\phi^{\ell\fs}_\fs$,
so both pairings are valued in $\I_\fs$.
Then by Proposition \cite[Prop 5.12]{lee},
the pairings satisfies the relation
\begin{align*}
	&\mathbf{B}_{\ell\fs}(\euF_1,\euF_2)=
	\mathbf{B}_{\fs}(\euF_1,T_{\fl}^{(1)}\cdot \euF_2)\\
	&\mathbf{B}_{\ell\fs}(\euF_1,V_{\fl}\cdot\euF_2)=
	(q_\fl+1) \mathbf{B}_{\fs}(\euF_1, T_{\fl}^{(2)}\cdot\euF_2)\\
	&\mathbf{B}_{\ell\fs}
	(V_{\fl}\cdot \euF_1,V_{\fl}\cdot\euF_2)=
	\mathbf{B}_{\fs} (T_{\fl}^{(1)}\cdot\euF_1,
	T_{\fl}^{(2)}\cdot \euF_2)
\end{align*} 
for $\euF_1,\euF_2\in S^{\ord}(U^p_\fs,\I_\fs)$.
Now let $u=\epsilon^{-1}\chi_\circ(\varpi_\fl)$
and $v=\epsilon^{-1}\chi_\circ\Psi_\fs(\varpi_\fl)$.
By Proposition \ref{prop:eigensystem}
we have $T_\fl^{(1)}\euF_\fs=(u+v)\euF_\fs$,
$T_\fl^{(2)}\euF_\fs=q_\fl^{-1}(uv)\euF_\fs$, and
$U_\fl\euF_{\ell\fs}=v\euF_{\ell\fs}$.
Moreover, by the construction in \cite{lee}, the Hida families
$\overline{\euF}_{\ell\fs}=\phi^{\ell\fs}_\fs\circ \euF_{\ell\fs}\in S^{\ord}(U^p_{\ell\fs}, \I_\fs)$ and
$\euF_\fs\in S^{\ord}(U^p_{\fs}, \I_\fs)$ are related by
\begin{equation}\label{eq:compare_level}
    \overline{\euF}_{\ell\fs}=
	\big(1-v^{-1}V_\fl)\euF_{\fs}
    \in S^{\ord}(U^p_{\ell\fs},\I_\fs).
\end{equation}
Since $U_{\ell\fs}^{-1}\euF_{\ell\fs}=v^{-1}U_{\fs}^{-1}\euF_{\ell\fs}$,
we can then compute that for $\EuScript{G}_{\fs}\in S^{\ord},(U^p_{\fs},\I_\fs)$, 
\begin{align*}
\overline{F}_{\ell\fs}(\EuScript{G}_{\fs})&=
\mathbf{B}_{\ell\fs}(\EuScript{G}_{\fs},
v^{-1}U_\fs^{-1}(1-v^{-1}V_\fl)\euF_\fs)\\&=
(u+v)v^{-1}
\mathbf{B}_\fs(\EuScript{G}_{\fs},U_\fs^{-1}\euF_\fs)-
(q_\fl^{-1}+1)(u/v)
\mathbf{B}_\fs(\EuScript{G}_{\fs}, U_\fs^{-1}\euF_\fs)=
(1-(u/v)q_\fl^{-1})F_\fs(\EuScript{G}_\fs).\\
\overline{F}_{\ell\fs}(V_\fl\EuScript{G}_{\fs})&=
\mathbf{B}_{\ell\fs}(V_\fl\EuScript{G}_{\fs},
v^{-1}U_\fs^{-1}(1-v^{-1}V_\fl)\euF_\fs)\\&=
u(q_\fl^{-1}+1)
\mathbf{B}_\fs(\EuScript{G}_{\fs},U_\fs^{-1}\euF_\fs)-
(u+v)(u/v)q_\fl^{-1}
\mathbf{B}_\fs(\EuScript{G}_{\fs}, U_\fs^{-1}\euF_\fs)=
u(1-(u/v)q_\fl^{-1})F_\fs(\EuScript{G}_\fs).
\end{align*}
Here we have used that $\mathbf{B}_\fs$
is equivariant with respect to the Hecke operators
by \cite[Rmk 5.13]{lee}.

\end{proof}

Let $u, v$ be as above. 
Then for $F\in M_{\ell\fs}^{\ord}$ we have
$\phi^{\ell\fs}_\fs(F(\euF_{\ell\fs}))=(\psi^*\overline{F})(\euF_\fs)$.
by \eqref{eq:compare_level}. Thus the morphism
\[
    \varphi\coloneqq (1-v^{-1}V_\fl)^*\otimes \phi^{\ell\fs}_\fs\colon 
    M_{\ell\fs}^{\ord}\to M_\fs^{\ord},
\]
which is equivariant with respect to the ring homomorphism
$\varphi^{\ell\fs}_\ell$ between Hecke algebra,
maps $S_{\ell\fs}^{\ord}\coloneqq\ker(\Theta_{\ell\fs})=
\{F\in M^{\ord}_{\ell\fs}\mid F(\euF_{\ell\fs})=0\}$ into 
$S_{\fs}^{\ord}$ as well. Then the morphism
$\varphi\colon M_{\ell\fs}\to M_\fs$ defined similarly also maps
$S_{\ell\fs}$ into $S_\fs$.

\begin{prop}\label{prop:cores}

Write $\bar{b}=\varphi^{\ell\fs}_{\fs}(b)\in B_\fs$
for $b\in B_{\ell\fs}$.
Then for the maps $f_\fs$ and $f_{\ell\fs}$
constructed in Definition \ref{def:cong_map}
we have the equality
\[
\overline{f_{\ell\fs}(b)}=
(1-\epsilon\Psi_{\fs}(\varpi_{\bar{\fl}}))
(1-\Psi_\fs(\varpi_{\bar{\fl}}))
\cdot f_\fs(\bar{b})\in \I_\fs.
\]
\end{prop}

\begin{proof}

Let $u,v$ be as in the previous lemma 
and consider the following cube,
where vertical maps are induced from 
the morphism $\varphi$ defined above
and the horizontal ones are from \eqref{eq:comm_s}.
\begin{equation*}
\begin{tikzcd}[column sep=tiny]
M_{\ell\fs}/S_{\ell\fs}\arrow{dd} \arrow[rd,equal]\arrow[rr]
&& (M_{\ell\fs}/S_{\ell\fs})\otimes_{\TT}\fq_{\ell\fs}
	\arrow[rd]\arrow[dd]\\
& M_{\ell\fs}/S_{\ell\fs}
	\arrow[crossing over]{dd} \arrow[crossing over]{rr} 
&& \fq_{\ell\fs}M_{\ell\fs}/\fq_{\ell\fs}S_{\ell\fs}
	\arrow[crossing over]{dd}\\
M_{\fs}/S_{\fs}\arrow{rr}\arrow[rd,equal]
&& (M_{\fs}/S_{\fs})\otimes_{\TT} \fq_{\fs}
	\arrow[rd]\\
& M_{\fs}/S_{\fs} \arrow[crossing over]{rr} 
&& \fq_{\fs}M_{\fs}/\fq_{\fs}S_{\fs}
\arrow[from=2-2,to=4-2,crossing over]
\arrow[from=2-4,to=4-4,crossing over]
\end{tikzcd}
\end{equation*}
Write $\bar{b}=\varphi^{\ell\fs}_{\fs}(b)$
for $b\in B_{\ell_\fs}$.
By Definition \ref{def:cong_map} we have
$f_\fs(\bar{b})=\Theta_{\fs}(\tilde{F}_\fs(\bar{b}))$ and 
$f_{\ell\fs}(b)=\Theta_{\ell\fs}(\tilde{F}_{\ell\fs}(b))$,
where $\tilde{F}_\bullet(b)\in (M_\bullet/S_\bullet)$
is the unique element such that 
$\xx \tilde{F}_\bullet(b)\equiv 
q(b)\tilde{F}_\bullet\mod q_\bullet S_\bullet$
and $\tilde{F}_\bullet$ is the preimage 
of $F_\bullet$ under 
$M_\bullet/S_\bullet\cong M_\bullet^{\ord}/S_\bullet^{\ord}$.
By the previous lemma we have
$\varphi F_{\ell\fs}=
(1-(u/v))(1-(u/v)q_\fl^{-1})F_{\fs}$. Therefore
$\varphi\tilde{F}_{\ell\fs}\equiv
(1-(\alpha/\beta))(1-(\alpha/\beta)q_\fl^{-1})\tilde{F}_{\fs}
\bmod S_\fs$ and consequently 
\begin{multline*}
\xx\varphi(\tilde{F}_{\ell\fs}(b))=
\varphi(\xx \tilde{F}_{\ell\fs}(b))\equiv 
\varphi(q(b)\tilde{F}_{\ell\fs})=
q(\bar{b})\varphi(\tilde{F}_{\ell\fs}) \equiv 
q(\bar{b})(1-(u/v))(1-(u/v)q_\fl^{-1})
\tilde{F}_{\fs}\\\equiv
\xx(1-(u/v))(1-(u/v)q_\fl^{-1})
\tilde{F}_{\fs}(\bar{b})\mod \fq_\fs S_\fs
\end{multline*}
Then by the injectivity in \eqref{eq:comm_s} we conclude that
$\varphi F_{\ell\fs}(b)\equiv
(1-(u/v))(1-(u/v)q_\fl^{-1})
F_\fs(\bar{b})\bmod S_\fs$.
And since $\Theta_\bullet$ is the evaluation
at $\euF_\bullet$ of the images in $M_\bullet^{\ord}$,
we have
\[
    \overline{\Theta_{\ell\fs}(\tilde{F}_{\ell\fs}(b))}
    =\Theta_\fs(\varphi\tilde{F}_{\ell\fs}(b))=
    (1-(u/v))(1-(u/v)q_\fl^{-1})
    \Theta_\fs(\tilde{F}_\fs(\bar{b}))
\]
and the proposition follows from 
$u/v=\Psi_\fs^{-1}(\varpi_\fl)=\Psi_\fs(\varpi_{\bar{\fl}})$.
	
\end{proof}

\begin{defn}\label{def:bigEuler}

Let $\mathcal{R}$ be the set of ideals
$\fs=\ff\ff^c$ as in Definition \ref{def:family_at_s}.
For $\fs\in\mathcal{R}$ we define
$\mathcal{Z}_\fs\in H^1(\K,\I_\fs(\Psi_\fs^{-1}))$
as the cohomology class
associated to $f_\fs\in \Hom_{\TT_\fm}(B,\I_\fs)$
by Proposition \ref{prop:BC}.
Similarly, we define 
$\mathcal{Z}_{p,\fs}\in H^1(\Qp,\I_\fs(\Psi_\fs^{-1}))$
as the class associated to 
$f_0\in \Hom_{R^{\zeta\epsilon}}(B_0,\I_\fs)$
from Proposition \ref{prop:BC}.

\end{defn}

\begin{thm}\label{thm:Bigeu}
The cohomology classes
$\{\mathcal{Z}_\fs\in H^1(\K, \I_\fs(\Psi_\fs^{-1}))\}_{\fs\in\mathcal{R}}$
has the following properties.
\begin{enumerate}
    \item There exists a finite set $S$ of primes of $\K$
    that contains all the places above $p$ and all the 
    places where $\psi$ is ramified, such that 
    $\mathcal{Z}_{\fs}$ is unramified away 
    $S\cup \{w\colon w\mid \fs\}$.
    \item For $w'\in \Sigma_p\setminus\{w_0\}$
    and $\sigma\in D_{w'}$, the restriction
    $\loc_{w'}(\mathcal{Z}_\fs)\in H^1(\K_{w'}, \I_\fs(\Psi_\fs^{-1}))$
    satisfies 
    \[
        (\Psi_{\fs}^{-1}(\sigma)-1)\cdot \loc_{w'}(\mathcal{Z}_\fs)=0
    \]
    \item For $w=w_0$, the restriction 
    $\loc_{w}(\mathcal{Z}_\fs)\in H^1(\Qp, \I_\fs(\Psi_\fs^{-1}))$
    is $L_\fs\cdot \mathcal{Z}_{p,\fs}$.
    \item If $\ell=\fl\fl^c$ and $\ell\fs,\fs\in \mathcal{R}$, then
    \[
        \phi^{\ell\fs}_\fs(\mathcal{Z}_{\ell\fs})=
        (1-\epsilon\Psi_{\fs}(\varpi_{\bar{\fl}}))
        (1-\Psi_\fs(\varpi_{\bar{\fl}}))\cdot 
        \mathcal{Z}_{\fs}
    \]
\end{enumerate}
\end{thm}
\begin{proof}
Recall that each $U^p_\fs$ is obtained 
from replacing the component at $\fs$
of a fixed compact open subgroup $K^p$.
Let $S$ be the union of 
\begin{itemize}
    \item places above $p$.
    \item places where $\psi$ is ramified.
    \item places where the component of $K^p$ at which
    is not hyperspecial.
\end{itemize}
It then follows from Proposition \ref{prop:gal_ger}
that each $r_\fp$ and therefore the big Galois pseudo-representations
$T(U^p_\fs)$ factors through the Galois group over $\K$
of the maximal extension that is unramified away 
$S\cup \{w\colon w\mid \fs\}$ and the first part of the claim follows.

The second claim follows from that for any $\sigma\in D_{w'}$, 
where $w'\in\Sigma_p\setminus w_0$, the class 
\[
   \sigma'\mapsto 
   f_\fs\big((A(\sigma)-D(\sigma))B(\sigma')\big)=
   \epsilon^{-1}\hat{\chi}_\circ(\sigma)
   (1+\Psi_\fs(\sigma))\mathcal{Z}_\fs(\sigma')
\]
is a coboundary by Proposition \ref{prop:res}.
Similarly, by construction
the last two claims follow respectively from
Proposition \ref{prop:res_0} and
Proposition \ref{prop:cores}.
\end{proof}

\subsection{Specializations}

By Shapiro lemma, the compatible system constructed in 
Theorem \ref{thm:Bigeu}
can also be viewed as an infinite collection of systems
of cohomology classes,
indexed by characters of $\fG_\fs^a$,
that are compatible under the corestriction maps.
Here we explicate this transition.

\begin{defn}\label{def:anticyc_extn}

Let $\tilde{\K}_\id/\K$ be the extension 
that corresponds to 
$\fG^a_\id=\Delta^a_\id\times W$ and fix an isomorphism
$W\cong \Zp^d$.
Then $\tilde{K}\coloneqq \tilde{K}_\id^{\Delta^a_\id}$
is the maximal anticyclotomic $\Zp^d$-extension
with $\Gal(\tilde{K}/\K)\cong W$.
Define
$U\coloneqq \prod_{w\in\Sigma_p}(1+\varpi_w\oo_w)^\times$
and let $H=Gal(\K(1)/\K)$ be the Galois group of 
the maximal anticyclotomic unramified $p$-extension.
Since $\K$ contains no $p$-power roots of unity,
by class field theory we have an exact sequences
\[
    1\to U\to \fG_\id^a\to H\to 1
\]
If $V\subset U$ is an open subgroup
we let $\K(V)/\K(\id)$ denote the extension such that
$Gal(\K(V)/\K(\id))\cong U/V$.

More generally,
for $\fs=\ff\ff^c\in\mathcal{R}$
defined as in Definition \ref{def:bigEuler}
let $U=\prod_{w\in\ff}U_w$,
where $\Delta_w$ is the maximal pro-$p$ quotient of 
$(\oo_w/(\varpi_w))^\times$ when $w\mid \ff$.
Similarly we have the exact sequence
\[
    1\to \Delta_\ff\times U\to \fG_\fs^a\to H\to 1.
\]
For each $\fl\mid\ff$ and $\ell\coloneqq\fl\fl^c$
we let $\K(\ell)/\K(1)$ be the extension 
such that $H(\ell)\coloneqq Gal(\K(\ell)/\K(\id))\cong \Delta_{\fl}$
and let $\K(\fs)$ be the product of all $\K(\ell_i)$
if $\ff=\fl_1\cdots\fl_r$ and $\ell_i=\fl_i\fl_i^c$. Note that 
\[
    Gal(\K(\fs)/\K(\id))\cong 
    H(\ell_1)\times\cdots\times H(\ell_r)\cong \Delta_\ff
\]
For $V\subset U$ as above, 
we also define $\K(\fs V)=\K(\fs)\K(V)$,
which satisfies that 
$Gal(\K(\fs V)/\K(\id))\cong \Delta_{\fs}\times U/V$.
Then $\tilde{K}_\fs=\cup_V \K(\fs V)=\K(\fs)\tilde{\K}_\id$
is the extension that corresponds to $\fG_\fs^a$,
where $V$ ranges through open subgroups of $U$.

\end{defn}

Recall that $\Psi_\fs\colon \Gal_\K\to \I_\fs^\times$
is given by $\Psi(\gamma)=\psi^{-1}(\gamma)\langle \gamma)$,
where $\langle*\rangle$ is the tautological character.
Given a character $\alpha\colon W\to \eo^\times$,
we define the homomorphism 
\[
    \alpha^\iota\colon \I_\fs\to \I_\fs\quad
    \langle\gamma\rangle\mapsto \alpha^{-1}(\gamma)
    \langle \gamma^{-1}\rangle
\]
which are compatible between different tame levels 
$\ell\fs, \fs\in \mathcal{R}$ under $\phi^{\ell\fs}_\fs$.
Then $\alpha^\iota$ defines an isomorphism between Galois modules
$\alpha^\iota\colon \I_\fs(\Psi_\fs^{-1})=
\I_\fs(\psi\langle*\rangle^{-1})
\to \I_\fs(\psi\alpha\langle*\rangle)$.
By Shapiro lemma
(cf \cite[Lem 5.8]{Schneider2016} for example), 
we then have an isomorphism between Galois modules
\begin{equation}\label{eq:Shapiro}
    \alpha^\iota(\mathcal{Z}_\fs) \in 
    H^1(\K, \I_\fs(\psi\alpha\langle*\rangle))\cong 
    \varprojlim_{L\subset \tilde{\K}_\fs}
    H^1(L, \eo(\psi\alpha))
\end{equation}
where $\Gal_\K$ acts on the left by $\langle*\rangle$
and on the right by the usual 
Galois actions on cohomology groups.

\begin{defn}

For $\alpha$ as above we let  
$(z_L)_L\in \varprojlim_{L\subset \tilde{\K}_\fs}
H^1(L, \eo(\psi\alpha))$
be the system of cohmology classes that corresponds to 
$\alpha^\iota(\mathcal{Z}_\fs)$
under the isomorphism \eqref{eq:Shapiro}.
We write $z_\fs=z_{\K(\fs)}$ and
$z_{\fs V}=z_{\K(\fs V)}$ 
for the fields $\K(\fs V)$ defined above.
In particular we have  
$z_{\id}\in H^1(\K(\id), \eo(\psi\alpha))$.
We also put 
$z=\Cor_{\K(\id)}^{\K}(z_\id)\in H^1(\K, \eo(\psi\alpha))$,
which is also the bottom class associated to 
$\alpha^\iota(\mathcal{Z}_\id)$.

\end{defn}

\begin{lem}\label{lem:classp}

Let $z_p\in H^1(\Qp, \eo(\psi\alpha))$
be the class associated to 
$\alpha^{-1}\circ f_\circ\in \Hom_{R^{\zeta\epsilon}}(B_0, \eo)$,
where the homomorphism 
$\alpha^{-1}\colon \I_\fs\to \eo$ is induced by 
$\alpha^{-1}\colon \fG_\fs^a\to \eo^\times$, 
then $H^1(\Qp, \eo(\psi\alpha))=\eo z_p$
is free of rank one.
\end{lem}
\begin{proof}

Let $\fF$ denote the residue field of $\eo$.
The assumption \eqref{cond:gen_psi}
implies that 
$H^1(\Qp, \eo(\psi\alpha))$ is torsion-free and that
$H^1(\Qp, \fF(\psi\alpha))=H^1(\Qp, \fF(\psi))$
is the one-dimensional vector space generated by 
the image of $z_p$ under the exact sequence
\[
    H^1(\Qp, \eo(\psi\alpha))\xrightarrow{\varpi}
    H^1(\Qp, \eo(\psi\alpha))\to
    H^1(\Qp, \fF(\psi\alpha))
\]
Therefore $H^1(\Qp, \eo(\psi\alpha))$
is generated by $z_p$ by Nakayama's lemma
and free over $\eo$.

\end{proof}

When $L/\K$ is a finite extension and $w$ is a place of $\K$,
we let $\loc_w$ denote
$\prod_{\tilde{w}\mid w}\loc_{\tilde{w}}\colon 
H^1(L, \eo(\psi\alpha))\to 
\prod_{\tilde{w}\mid w}H^1(L_{\tilde{w}}, \eo(\psi\alpha))$.
The following proposition then follows immediately
from Theorem \ref{thm:Bigeu}
and that $\phi^{\ell\fs}_\fs$
corresponds to the corestriction maps
between cohomology groups of different tame levels.

\begin{thm}\label{thm:eu}

Given a character $\alpha\colon W\to \eo^\times$,
the cohomology classes
$\{ z_{\fs V}\in H^1(\K(\fs V), \eo(\psi\alpha))\}
_{\fs\in \mathcal{R},V\subset U}$
and $z=\Cor_{\K(\id)}^{\K}(z_\id)\in H^1(\K, \eo(\psi\alpha))$
satisfies the following properties.

\begin{enumerate}
    \item There exists a finite set $S$ of primes of $\K$
    that contains all the places above $p$ and all the 
    places where $\psi$ is ramified, such that 
    $\mathcal{Z}_{\fs}$ is unramified away primes above
    $S\cup \{w\colon w\mid \fs\}$.
    \item For $w'\in \Sigma_p\setminus\{w_0\}$
    and $\sigma\in D_{w'}$, the restriction
    $\loc_{w'}(z_\fs)\in H^1(\K(\fs)_{w'}, \eo(\psi\alpha))$
    satisfies 
    \[
        ((\psi\alpha)(\sigma)-1)\cdot \loc_{w'}(z_\fs)=0.
    \]
    \item For $w=w_0$, the restriction 
    $\loc_{w}(z)\in H^1(\Qp, \eo(\psi\alpha))$
    is $\alpha^{-1}(L_\id)\cdot z_{p}$.
    \item
    For $\ell=\fl\fl^c$ such that $\ell\fs,\fs\in \mathcal{R}$ define
    $P_\fl(X)=(1-\epsilon(\psi\alpha)^{-1}(\varpi_\bl)X)
    (1-(\psi\alpha)^{-1}(\varpi_\bl)X)$, then
    \begin{equation}\label{eq:Pl}
        \Cor_{\K(\fs V)}^{\K(\ell\fs V)}(z_{\ell\fs V})=
        P_\fl(\Fr_\fl)\cdot z_{\fs V}.
    \end{equation}
\end{enumerate}

\end{thm}

\section{Application to the main conjecture}

Let $\tilde{\K}$ be the maximal anticylotomic $\Z_p^d$-extension
over $\K$ with Galois group $W$ 
and let $\Lambda_W\coloneqq\eo\llbracket W\rrbracket$
for the ring of integers $\eo$ of a sufficiently large
finite extension over $\bar{\Q}^{un}$.
When $\psi$ is a finite order character of $\Gal_\K$ we can define
\[
\Psi\colon \Gal_\K\to \Lambda_W\quad \gamma\to \psi(\gamma)
\langle \gamma\rangle
\]
where $\langle*\rangle$ is the tautological character.
Let $\Gal_\K$ acts on $\Lambda_W^*\coloneqq
\Hom^{\cts}(\Lambda_W,\Qp/\Zp)$ by
$(\gamma\cdot f)(\lambda)=f(\Psi(\gamma)\lambda)$
for $f\in \Lambda_W^*$.
The Selmer groups defined by
\begin{equation}\label{def:Selmers}
\begin{aligned}
    &\Sel(\Psi,\Sigma_p^c)=
    \ker\big\{
    H^1(\K,\Lambda_W^*(\Psi))\to 
    \prod_{w\notin \Sigma_p^c}
    H^1(I_w,\Lambda_W^*(\Psi))
    \big\}\\
    &\Sel^{str}(\Psi,\Sigma_p^c)=
    \ker\big\{
    \Sel(\Psi,\Sigma_p^c)\to
    \prod_{w\in \Sigma_p}
    H^1(\K_w,\Lambda_W^*(\Psi))
    \big\}\\
\end{aligned}
\end{equation}
are co-finitely generated over $\Lambda_W$
in the sense that the Pontryagin duals 
$X(\Psi,\Sigma_p^c)\coloneqq  \Sel(\Psi,\Sigma_p^c)^\vee$ and
$X^{str}(\Psi,\Sigma_p^c)\coloneqq  \Sel^{str}(\Psi,\Sigma_p^c)^\vee$ 
are finitely generated over $\Lambda_W$.
Then the anticyclotomic main conjecture
states that $X(\Psi,\Sigma^c_p)$ is $\Lambda_W$-torsion and
that the characteristic power series
of $X(\Psi,\Sigma^c_p)$ 
is equal to a $p$-adic $L$-function up to units.

Our goal in this section is to show that 
$X(\Psi,\Sigma_p^c)$ is indeed $\Lambda_W$-torsion
and that the characteristic polynomial divides
a $p$-adic $L$-function in 
$\Lambda_W\hat{\otimes}_\eo E$,
where $E$ is the field of fraction of $\eo$,
using the collection of Euler systems 
$\{\mathcal{Z}_\fs\}_{\fs\in\mathcal{R}}$
constructed in Definition \ref{def:bigEuler}.
Recall that for the construction of which 
we assume  \ref{cond:K2} and \ref{cond:K3}
in Lemma \ref{lem:good_chi}
and that $\psi$ satisfies \eqref{cond:gen_psi}
and \eqref{cond:can}.

Following the notations in \cite{Hsieh2010},
we also consider the dual Selmer groups 
$\Sel(\Psi^D,\Sigma_p)$ and
$\Sel^{str}(\Psi^D,\Sigma_p)$ defined as above
for the character $\Psi^D=\epsilon\Psi^{-1}$,
but with the role of $\Sigma_p^c$ and $\Sigma_p$ switched.
By the main result from \textit{loc.cit.}
we will work with $\Sel^{str}(\Psi^D,\Sigma_p)$
instead of $\Sel(\Psi,\Sigma_p^c)$.

\subsection{Specialization principles}

We recall Ochiai's technique
developed in \cite{Och05}
that reduce the determination of the characteristic power series
to the zero-dimensional case.
For an integer $n\geq 0$ we define
$\Lda{n}_{\eo}=\eo\llbracket X_1,\cdots, X_n\rrbracket$.
In particular, 
after fixing topological generators 
$\{\gamma_1,\cdots,\gamma_d\}$ of $W$
we may identify $\Lambda_W$
and $\Lda{d}_{\eo}$ by $\langle\gamma_i\rangle=1+X_i$.
Since $\eo$ is a complete discrete valuation ring,
we can pick a uniformizer $\varpi\in\eo$ such that $\fp=(\varpi)$
is the maximal ideal.

\begin{defn}\label{def:lin_elt}
A polynomial 
$\ell=a_0+a_1X_1+\cdots a_nX_n\in\Lda{n}_{\eo}$
of degree at most 1
is called a linear element 
if it is neither a unit
nor an element in $\fp\Lda{n}_{\eo}$.
In other word $a_0\in\fp$
and $a_i\in\eo^\times$
for some $1\leq i\leq n$.
Then 
\[
    \lin{n}_{\eo}\coloneqq\{(\ell)\subset \Lda{n}_{\eo}\mid\ell
    \text{ is a linear element in }\Lda{n}_{\eo}
\]
Moreover, if $n\geq 2$
and $M$ is a finitely generated torsion 
$\Lda{n}_{\eo}$-module, we define 
$\lin{n}_{\eo}(M)$ to be the subset of
$(\ell)\in \lin{n}_{\eo}$ satisfying the following conditions.
\begin{enumerate}
    \item The quotient $M/\ell M$ is 
    torsion over $\Lambda^{(n)}/(\ell)$.
    \item The image of the characteristic ideal
    $\car_{\Lda{n}_{\eo}}(M)\subset \Lda{n}_{\eo}$ in $\Lambda^{(n)}/(\ell)$
    is equal to $\car_{\Lda{n}_{\eo}/(\ell)}(M/\ell M)$.
\end{enumerate}
\end{defn}

Let $S_\psi$ be the set of primes on which $\psi$ is ramified
and let $\Gal$ be the Galois group of the maximal extension
over $\K$ that is unramified away $S_\psi$ and primes above $p$.
For any quotient $\mathcal{A}$ of $\Lambda_W$, we define 
\begin{equation}
    \Sel(\mathcal{A})=
    \ker\big\{
    H^1(\Gal, \mathcal{A}^*(\Psi^D))\to
    \prod_{w\in S_\psi}
    H^1(I_w,\mathcal{A}^*(\Psi^D))\times 
    \prod_{w\in \Sigma_p^c}
    H^1(\K_w,\mathcal{A}^*(\Psi^D))
    \big\}
\end{equation}
and $X(\mathcal{A})=\Sel(\mathcal{A})^\vee$.
Note that by definition we have
$\Sel^{str}(\Psi^D,\Sigma_p)=\Sel(\Lambda_W)$.
As in \cite[Prop 2.6]{Hsieh2010},
when $\mathcal{A}\cong \Lda{n}_{\eo}$ for some $n\geq 1$
and $\ell\in\Lda{n}_{\eo}$ is a linear element,
a standard control theorem argument gives an exact sequence
$0\to \coker(\loc^0)^\vee\to X(\mathcal{A})/\ell X(\mathcal{A})
\to X(\mathcal{A}/(\ell))\to \ker(\loc^0)^\vee\to 0$
by applying the snake lemma to the diagram below
then take the Pontryagin duals.
\[
\begin{tikzcd}
0 & 0\\
H^1(\Gal,\mathcal{A}^*(\Psi^\circ))[\ell] \arrow[u]\arrow[r] &
\prod_{w\in S_\psi}
H^1(I_w,\mathcal{A}^*(\Psi^\circ))[\ell]\times
\prod_{w\in\Sigma_p^c}
H^1(D_w,\mathcal{A}^*(\Psi^\circ))[\ell] \arrow[u]\\
H^1(\Gal,\mathcal{A}^*(\Psi^\circ)[\ell]) \arrow[u]\arrow[r] &
\prod_{w\in S_\psi}
H^1(I_w,\mathcal{A}^*(\Psi^\circ)[\ell])\times
\prod_{w\in\Sigma_p^c}
H^1(D_w,\mathcal{A}^*(\Psi^\circ)[\ell]) \arrow[u]\\
\frac{H^0(\Gal,\mathcal{A}^*(\Psi^\circ))}
{(\ell)H^0(\Gal,\mathcal{A}^*(\Psi^\circ))}
\arrow[u]\arrow[r,"\loc^0"] &
\prod_{w\in S_\psi}
\frac{H^0(I_w,\mathcal{A}^*(\Psi^\circ))}
{(\ell)H^0(I_w,\mathcal{A}^*(\Psi^\circ))}\times
\prod_{w\in\Sigma_p^c}
\frac{H^0(D_w,\mathcal{A}^*(\Psi^\circ))}
{(\ell)H^0(D_w,\mathcal{A}^*(\Psi^\circ))} \arrow[u]\\
0\arrow[u] & 
0\arrow[u]
\end{tikzcd}
\]

\begin{lem}
For $\mathcal{A}\cong \Lda{n}_{\eo}$ as above
and suppose $\Psi^D\vert_{I_w}\neq \id$ on $\mathcal{A}$
for $w\in\Sigma_p^c$,  then
for all but finitely many linear element $\ell\in\Lda{n}_{\eo}$
the followings hold:
\begin{itemize}
\item 
For $w\in\Sigma_p^c$
the character $\Psi^D\vert_{I_w}$ is nontrivial on $\mathcal{A}/(\ell)$
\item 
The modules $\coker(\loc^0)^\vee$ and $\ker(\loc^0)^\vee$
are torsion over $\mathcal{A}/(\ell)$.
\end{itemize}
\end{lem}

\begin{proof}
Note that $H^0(\Gal,\mathcal{A}^*(\Psi^D))^\vee$
is isomorphic to $\mathcal{A}/(\Psi^D(\Gal)-1)$.
Therefore $\ker(\loc^0)^\vee$ is a quotient of
$\big(\mathcal{A}/(\Psi^D(\Gal)-1)\big)[\ell]$.
Similarly, $\coker(\loc^0)$ is a submodule of 
the product of 
$\big(\mathcal{A}/(\Psi^D(I_w)-1)\big)[\ell]$
for $w\in S_\psi$ and
$\big(\mathcal{A}/(\Psi^D(D_w)-1)\big)[\ell]$
for $w\in \Sigma_p^c$.
It is clear that 
$\big(\mathcal{A}/(\Psi^D(I_w)-1)\big)[\ell]=
\big(\mathcal{A}/(\psi(I_w)-1)\big)[\ell]$
is always torsion over $\mathcal{A}/(\ell)\cong \Lda{n-1}_{\eo}$.
Therefore $\big(\mathcal{A}/(\Psi^D(\Gal)-1)\big)[\ell]$
is also torsion over $\mathcal{A}/(\ell)\cong \Lda{n-1}_{\eo}$.
For $w\in\Sigma_p^c$,
$(\Psi^D(I_w)-1)$ is generated by
$(\Psi^D(\Art_w(1+\varpi_w))-1$,
thus the desired properties are satisfies as long as 
$\ell$ is distinct from the image of which in $\mathcal{A}$.
    
\end{proof}

\begin{prop}\label{prop:tor_crit}
Let $\alpha\colon \Lambda_W\to \eo$
be the homomorphism induced by a
character $\alpha\colon W\to \eo^\times$
and let $\eo(\alpha)$ denote 
the resulting $\Gal_\K$-module structure
on $\eo(\alpha)=\eo$. 
Write $\mu=\psi\alpha$ and 
$\mu^*=\alpha\circ\Psi^D=\epsilon(\psi\alpha)^{-1}$ so that
\begin{equation}\label{eq:Sel_base}
\Sel(\eo(\alpha))=
\ker\big\{
H^1(\Gal, (E/\eo)(\mu^*))\to
\prod_{w\in S_\psi}
H^1(I_w,(E/\eo)(\mu^*))\times 
\prod_{w\in \Sigma_p^c}
H^1(\K_w,(E/\eo)(\mu^*))
\big\}
\end{equation}
Then $X(\Lambda_W)$
is torsion over $\Lambda_W$
if $\Sel(\eo(\alpha))^\vee$ is finite over $\eo$
for all but finitely many characters $\alpha$.
\end{prop}

\begin{proof}
By the previous lemma, there exists a character $\alpha$
such that  $\Sel(\eo(\alpha))^\vee$ is finite over $\eo$
and that the kernel and cokernel of the homomorphism
\[
    X(\mathcal{A}_i)/(\ell_{i+1})X(\mathcal{A}_{i})
    \to X(\mathcal{A}_i/(\ell_{i+1}))
\]
are torsion over $\mathcal{A}_{i+1}\coloneqq \mathcal{A}_i/(\ell_{i+1})$.
Here $\mathcal{A}_0=\Lambda_W$
and $\ell_i=(X_i-\alpha(\gamma_i))$ for $1\leq i\leq d$.
Then $X(\mathcal{A}_{d-1})$ is torsion over 
$\mathcal{A}_{d-1}\cong \Lda{1}_{\eo}$ since
$X(\mathcal{A}_d)=\Sel(\eo(\alpha))^\vee$,
and thus $X(\mathcal{A}_{d-1})/(\ell_{d})X(\mathcal{A}_{d-1})$,
torsion over $\eo$.
We can then iterate the argument and prove that 
$X(\Lambda_W)$ is torsion over $\Lambda_W$.

\end{proof}

\begin{prop}\label{prop:specialize}

Let $\eo'$ be an arbitrary extension over $\eo$
which is the ring of integers 
of a finite extension over $\bar{\Q}^{un}$
and let $L\in \Lambda_W$.
If for all but finitely many characters
$\alpha\colon W\to (\eo')^\times$
the Pontryagin dual $X(\eo'(\alpha))$ defined 
over $\eo'$ satisfies that
\[
    \length_{\eo'}(\eo'/\alpha(L))-
    \length_{\eo'}(X(\eo'(\alpha))
\]
is bounded independent of $\alpha$,
then there exists an integer $h\geq 0$
such that 
$\car_{\Lambda_W}(X(\Lambda_W))\supset (\varpi^hL)$,
where $\varpi$ is a uniformizer 
of the complete discrete valuation ring $\eo$.

\end{prop}

\begin{proof}

Let $\mathcal{A}$
be an arbitrary quotient of $\Lambda_W$
that is isomorphic to $\Lda{n}_{\eo}$
for some integer $n\geq 1$.
The previous proposition implies that
$X(\mathcal{A})$ is a finite and torsion 
$\mathcal{A}$-module
and therefore $\car_{\mathcal{A}}(X(\mathcal{A}))$ exists.
We will prove inductively that
$\car_{\mathcal{A}}(X(\mathcal{A}))\supset 
(\varpi^{h'} \bar{L})$
for some integer $h'\geq 0$,
where $\bar{L}$ is the image of $L$ in $\mathcal{A}$.

When $n=1$, 
the proposition follows directly
from \cite[Prop 3.11]{Och05}.

When $n\geq 2$, suppose inductively that
$\car_{\mathcal{A}/(\ell)}(X(\mathcal{A}/(\ell))\supset 
(\varpi^{h'} \bar{L})$ for all but finitely many 
$\ell\in \lin{n}_{\eo'}$.
After inverting $p$ and take the completion,
this implies that 
$\car_{\mathcal{A}_{E'}/(\ell)}(X(\mathcal{A}_{E'}/(\ell))
\supset (\bar{L})$,
where $E'$ is the field of fraction of $\eo'$
and  $\mathcal{A}_{E'}$ is isomorphic to
$E'\llbracket X_1,\cdots,X_n\rrbracket$.
Let 
\[
    E'\llbracket X_1,\cdots, X_{n-1}\rrbracket 
    \hookrightarrow
    \prod_{1\leq i\leq j}
    E'\llbracket X_1,\cdots, X_{n-1}\rrbracket/(\ell_i)
\]
be the homomorphism considered in  \cite[p.134]{Och05}
for the choice of the set
$\{(\ell_i)\in \lin{n}_{\eo''}\}_{1\leq i<\infty}$
made in \cite[Claim 3.10]{Och05}
Then the same argument in the proof of 
\cite[Prop 3.6]{Och05}
gives the desired inclusion.

\end{proof}

\begin{rem}
It is assumed in \cite{Och05} that
$\eo$ is the ring of integers
of a finite field extension over $\Qp$.
But the proofs also works when
$\eo$ is a complete discrete valuation ring 
with the residue field $\bar{\fF}_p$.
In particular the propositions cited above
still hold in our case,
when $\eo$ is the ring of integers
of a finite field extension over $\widehat{\Q}_p^{un}$.
Note that it is then necessary to state
\cite[Prop 3.11]{Och05}
in terms of lengths of $\eo'$-modules
instead of using the cardinalities,
which will always be infinite in our case
\end{rem}

\begin{cor}\label{cor:main}
Under the same assumptions and notations
in the previous proposition, 
the $\Lambda_W$-module $X(\Psi, \Sigma_p^c)$
is also torsion and 
$\car_{\Lambda_W}(X(\Psi,\Sigma_p^c))=\car_{\Lambda_W}(X(\Lambda_W))$.
\end{cor}
\begin{proof}
    We first note that 
    the Pontryain dual of the right hand side of the exact sequence
    \[
    0\to \Sel^{str}(\Psi^D,\Sigma_p)\to 
    \Sel(\Psi^D,\Sigma_p)\to 
    \prod_{w\in \Sigma_p^c}
    H^1(D_w/I_w, \big(\Lambda_W^*(\Psi^D)\big)^{I_w})
    \]
    is a pseudo-null $\Lambda_W$-module.
    This follows from that $D_w$ has $\Zp$-rank at least two
    for $w\in\Sigma_p$ if $d>1$.
    If $d=1$, then the condition \eqref{cond:gen_psi} implies
    that $\big(\Lambda_W^*(\Psi^D)\big)^{I_w}=0$
    at the unique place $w_0$ above $p$.
    Therefore 
    $X(\Psi^D,\Sigma_p)$ is also $\Lambda_W$-torsion
    and the characteristic ideal of which agrees with 
    that of $X(\Lambda_W)=X^{str}(\Psi^D, \Sigma_p)$.

    Apply Poitou-Tate duality as in the proof of
    \cite[Prop 2.6]{Hsieh2010} then implies that 
    $\Sel(\Psi,\Sigma_p^c)$ satisfies a similar 
    result as Proposition \ref{prop:tor_crit}
    and is therefore $\Lambda_W$-torsion.
    The corollary then follows from 
    \cite[Thm 2.8]{Hsieh2010}
    since the proof of which also applies in our case.

\end{proof}

\subsection{Base case}

Given a character $\alpha\colon W\to \eo^\times$,
we write $\mu=\psi\alpha$ and let 
$T=\eo(\mu), V=E(\mu), W=(E/\eo)(\mu)$,
where $E$ is the field of fraction of $\eo$ and
$\Gal_\K$ acts on each of the underlying spaces by $\mu$.
If $M\in\eo$, we define $W_M=W[M]$ and let $\K(W_M)/\K$ 
denote the minimal extension such that
the $\Gal_\K$-representation on the finite group $W_M$
factors through $Gal(\K(W_M)/\K)$.
Let $\mu^*\coloneqq \epsilon\mu^{-1}$ denote the Tate dual.
We similarly define the spaces
$T^*=\eo(\mu^*), V^*=E(\mu^*), W^*=(E/\eo)(\mu^*)$ and 
$W_M^*=W[M]^*$.

Let $\{z_{\fs V}\in H^1(\K(\fs V), T))\}_{\fs\in\mathcal{R}, V\subset U}$
be the system of cohomology classes
given in Theorem \ref{thm:eu}.
We show that a slight modification of 
the argument in \cite{Rubin} 
gives a uniform bound of
$\length_{\eo}(X(\eo(\alpha))$
in terms of $\alpha^{-1}(L_\id)$ for almost all $\alpha$.
We start with the following lemma 
modified from \cite[Lem 4.1.3]{Rubin}.

\begin{lem}\label{lem:estimate}
Suppose $\tau\in \Gal_{\K(\id)}$ and
$M\colon=\mu(\tau)-1$ is a nonzero element in $\varpi\eo$.
Let $\bar{M}\in\Z^+$
denote the smallest power of $p$
which is divisible by $M$ in $\eo$.
Then there exists infinitely many 
prime-to-$p$ split primes $\fl\neq\fl^c$
at which $\mu$ is unramified satisfies the conditions below.
\begin{itemize}
\item $\bar{M}^2\mid n_\ell\coloneqq \#H(\ell)$ for $\ell=\fl\fl^c$,
\item $M\equiv (\mu(\Fr_\fl)-1)\mod M^2$.
\item $\fl$ splits completely in $\K(\id)$.
\end{itemize}
\end{lem}
\begin{proof}

We first observe that if $\fl$ splits completely in 
$\K(\mu_{\bar{M}^2})$, then
$\bar{M}^2\mid (q_\fl-1)=\#(\oo_\K/\fl)^\times$,
and therefore $\bar{M}^2\mid \#\Delta_\fl=\#H(\ell)$
is satisfied if $\ell=\fl\fl^c$.

Then suppose $\Fr_\fl$ is conjugate to $\tau$
in $Gal(\K(\id)\K(W_{M^2})/\K)$.
This automatically implies that $\Fr_\fl$ splits completely 
in $\K(\id)$ since $\tau\in\Gal_{\K(\id)}$;
moreover we have
$(\mu(\Fr_\fl)-1)\equiv (\mu(\tau)-1)=M\bmod M^2$
and thus the last two conditions are satisfied.

Since $\mu$ is anticyclotomic,
the extension $\K(\id)\K(W_{M^2})$
is linearly disjoint from 
$\K(\mu_{\bar{M}^2})$
and there exists some $\bar{\tau}\in 
Gal(\K(\id)\K(W_{M^2})\K(\mu_{\bar{M}^2})/\K)$
that is conjugate to $\tau$ in 
$Gal(\K(\id)\K(W_{M^2})/\K)$
and trivial in 
$Gal(\K(\mu_{\bar{M}^2})/\K)$.
By the discussion above, any $\Fr_\fl$ 
that is conjugate to $\bar{\tau}$ in
$Gal(\K(\id)\K(W_{M^2})\K(\mu_{\bar{M}^2})/\K)$
would satisfy all three conditions,
and by Chebotarev's density,
there exists infinitely such primes that split.
    
\end{proof}

\begin{defn}\label{def:RM}
For a nonzero $M\in \eo$ as above,
we let $\mathcal{R}_M\subset \mathcal{R}$
be the subset of ideals $\fs=\ff\ff^c$
that are coprime to the conductor of $\psi$
and such that each prime divisor $\fl\mid \ff$
satisfies the conditions in Lemma \ref{lem:estimate}.
\end{defn}

\begin{defn}

Suppose $\fl$ is a split prime such that 
$\ell=\fl\fl^c\in \mathcal{R}$.
Let $\sigma_\ell$
be a generator of the cyclic group
$H(\ell)\cong \Delta_\fl$.
We then define
$D_\ell=\sum_{i=0}^{n_\ell-1}i\cdot \sigma_\ell^i
\in \Z[H(\ell)]$
where $n_\ell\coloneqq\#H(\ell)$.
In general, if $\fs=\ff\ff^c\in\mathcal{R}$
where $\ff=\prod_{i=1}^r\fl_i$, we put
\[
    D_\fs=\prod_{i=1}^r D_{\ell_i}\in 
    \Z[Gal(\K(\fs)/\K(\id))]\cong 
    \Z[H(\ell_1)\times\cdots\times H(\ell_r)].
\]
\end{defn}

Recall that $\eo$ is a complete discrete valuation ring
with field of fraction $E$ which is finite over $\widehat{\Q}_p^{un}$.
We fix an uniformizer $\varpi\in\eo$
and let $e$ be the ramification index,
so that $(p)=(\varpi)^e\subset \eo$.
We normalize the valuation function
$\val\colon E^\times/\eo^\times\to \Q$
so that $\val(p)=1$.
We note that $\val(E)\subset e^{-1}\Z$.

\begin{lem}

Suppose $\fs=\ff\ff^c\in\mathcal{R}_M$
for $M$ and $\mathcal{R}_M$ as in Defintion \ref{def:RM}.
We have $D_{\fs}z_{\fs}\bmod M^2\in 
H^1(\K(\fs),T/{M^2}T)^{\K(\id)}$
and consequently
$\Nr_{\K(\id)/\K}D_\fs z_\fs\bmod M^2\in
H^1(\K(\fs),T/{M^2}T)^{\K}$.
\end{lem}

\begin{proof}

We prove the lemma by induction on $r$,
the number of prime divisors in  $\ff=\fl_1\cdots\fl_r$.
The claim is trivial when $r=0$.
Suppose inductively the claim holds for $\fs$.
Then if $\fl$ is a prime such that 
$\ell\fs\in\mathcal{R}_M$ for $\ell=\fl\fl^c$,
it follows from 
$(\sigma_\ell-1)D_\ell=n_\ell-\Nr_{\K(\ell)/\K(\id)}$
that
\[
	 (\sigma_\ell-1)D_{\fs\ell}z_{\fs\ell}\equiv
	 -\Cor^{\fs\ell}_{\fs}D_{\fs}z_{\fs\ell}=
	 -P_\fl(\Fr_{\fl})D_{\fs}z_{\fs} \equiv
	 -P_\fl(1)D_{\fs}z_{\fs}\equiv 0\mod M^2.
\]
Here the first equality follows from $M^2\mid n_\ell$,
the second from \eqref{eq:Pl},
the third from that
$\Fr_\fl\in \Gal_{\K(\id)}$ acts trivially on $D_\fs z_\fs$
by induction,
and the last follows from that $\ell\in\mathcal{R}_M$, as then
\[P_\fl(1)=(1-q_\bl^{-1}\mu^{-1}(\Fr_\bl))(1-\mu^{-1}(\Fr_\bl))
\equiv (\mu(\Fr_\fl)-1)^2 \equiv 0\mod M^2.\]
\end{proof}

Observe that $(T/{M^2}T)^{K(\fs)}=0$ by \eqref{cond:gen_psi}
since $K(\fs)$ is unramified at all primes above $p$. 
Thereforl the exact sequence below degenerates to 
$\Res\colon H^1(\K, T/{M^2}T)\cong H^1(\K(\fs), T/{M^2}T)^{\K}$.
\[
	0\to H^1(\K(\fs)/\K, (T/{M^2}T)^{\K(\fs)})\to
	H^1(\K, T/{M^2}T)\to
	H^1(\K(\fs), T/{M^2}T)^{\K}\to
	H^2(\K(\fs)/\K, (T/{M^2}T)^{\K(\fs)})
\]

\begin{defn}
For $M\in \eo$ and $\fs\in\mathcal{R}_M$ as above
we define the Kolyvagin derivative 
$\kappa_{\fs,M}\in H^1(\K,W_{M^2})\cong H^1(\K,T/{M^2}T)$
by the property that 
\[
    \Res(z_{\fs,M})=\Nr_{\K(\id)/\K}D_{\fs}z_{\fs} \bmod M^2\in 
    H^1(\K(\fs), T/{M^2}T)^{\K},
\]
which is well-defined as
$\Res\colon H^1(\K, T/{M^2}T)\cong H^1(\K(\fs), T/{M^2}T)^{\K}$.
Here we identify $W_{M^2}$ and $T/M^2T$ 
with the homomorphism $M^2\colon W_{M^2}\cong T/M^2T$
coming from multiplication by $M^2$.
\end{defn}

As in \cite[Def 1.3.4]{Rubin},
if $w$ is a prime-to-$p$ place of $\K$
we let $H^1_f(\K_w,W)$
denote the image of 
\[
    H^1_f(\K_w,V)\coloneqq
    \ker\big(H^1(\K_w,V)\to H^1(I_w, V)\big)
\]
under $H^1(\K_w,V)\to H^1(\K_w, W)$
and let $H^1_f(\K_w,W_{M^2})$ denote the preimage of which
under $H^1(\K_w,W_{M^2})\to H^1(\K_w, W)$. 
Then by \cite[Lem 1.3.8]{Rubin},
when $\psi$ is unramified at $w$ we have
\begin{equation}\label{eq:f_un}
    H^1_f(\K_w,W_{M^2})=
    \ker\big(H^1(\K_w,W_{M^2})\to H^1(I_w, W_{M^2})\big).
\end{equation}
We also let $H^1_s(\K_w,W_{M^2})$
denote the quotient of $H^1(\K_w,W_{M^2})$ by $H^1_f(\K_w,W_{M^2})$,
so we have 
\[
    0\to H^1_f(\K_w,W)\to H^1(\K_w,W)\to H^1_s(\K_w,W)\to 0.
\]

\begin{prop}\label{prop:goodm}
There exists $m\in \eo$
that depends only on $\psi$ in $\mu=\psi\alpha$, such that 
for $\fs=\ff\ff^c\in \mathcal{R}_M$ as above we have
$\loc_w(m\cdot \kappa_{\fs,M})\in H^1_f(\K_w,W_{M^2})$
for any finite place $w$ that is prime to $p\fs$.
\end{prop}

\begin{proof}

Let $S$ be the finite set of places in Theorem \ref{thm:eu},
then $z_\fs$, and consequently $\kappa_{\fs,M}$
is unramified at $w\notin S\cup\{w\colon w\mid\fs\}$.
Therefore $\loc_w(\kappa_{\fs, M})\in H^1_f(\K_w,W_{M^2})$
as $\mu=\psi\alpha$ is also unramified at 
$w\notin S\cup\{w\colon w\mid\fs\}$.

If $w\in S$ is a split place and prime to $p$,
then the image of $D_w$ in $\fG_{\fs}^a$
is infinite and $\kappa_{\fs, M}$
belongs to the right hand side of \eqref{eq:f_un}
by \cite[Cor 4.6.2]{Rubin}.
Then $\loc_w(m\cdot \kappa_{\fs, M})\in H^1_f(\K_w, W_{M^2})$
by \cite[Cor 4.6.5]{Rubin}
if we pick an $m\in \eo$ that annihilates
$(W^{I_w}/W^{I_w})_{\textnormal{div}}$,
which depends only on $\psi$
as $\alpha$ is ramified only at $p$.
If $w\in S$ is non-split,
by assumption $\psi$ is unramified at $w$.
It then follows from \cite[Lem 1.3.2(ii)]{Rubin} that
\[
    H^1_s(\K_w,W_{M^2})\cong \Hom(I_w, W_{M^2})^{\Fr_w=1}.
\]
Since any homomorphism from $I_w$ to $W$
factors through the maximal pro-p
quotient of $(\oo_w/(\varpi_w))^\times$ by class field theory,
we have $\loc_w(m\cdot \kappa_{\fs,M})\in
H^1_f(\K_w,W_{M^2})$
if $m\in \eo$ annihilates the quotient,
which again is independent of $\alpha$.
We can now pick $m\in \eo$
that satisfies the conditions above
at the finitely many places in $S$
and the proposition follows from the above discussions.

\end{proof}

Let $\fl$ be a prime such that $\ell=\fl\fl^c\in \mathcal{R}_M$.
Then $T$ is unramified at $\fl$ by definition, and we let
\begin{align}
\alpha_\fl&\colon
H^1_s(\K_\fl,W_{M^2})\to W_{M^2}^{\Fr_\fl=1}\quad
[\kappa]\mapsto \kappa(\sigma_\fl)
\label{eq:singular} \\
\beta_\fl&\colon 
H^1_f(\K_\fl,W_{M^2})\to W_{M^2}/(\Fr_\fl-1)W_{M^2}\quad
[\kappa]\mapsto \kappa(\Fr_\fl)
\label{eq:finite}
\end{align}
be the isomorphisms from \cite[Lem 1.4.7]{Rubin}, which applies 
as $\bar{M}^2\mid n_\ell$ implies that
$\mu_{\bar{M}^2}\subset \K_\fl^\times$.

On the other hand,
if we define $Q_\fl(X)$ by
$P_\fl(X)=(1-X)Q_\fl(X)+P_\fl(1)$, then
\[
Q_\fl(X)=\big(\mu(\Fr_\fl)+q_\fl^{-1}\mu(\Fr_\fl)-q_\fl^{-1}\mu^2(\Fr_\fl)\big)-
q_\fl^{-1}\mu^2(\Fr_\fl)X.
\]
Therefore $Q_\fl(1)\equiv 2\mu(\Fr_\fl)(1-\mu(\Fr_\fl))\bmod M^2$
and $Q_\fl(1)\eo=M\eo$.
Since $W_{M^2}^{\Fr_\fl=1}$ and $W_{M^2}/(\Fr_\fl-1)W_{M^2}$ 
are isomorphic to $W_M$ and $W_{M^2}/W_M$
respectively, the map
$Q_\fl(1)\colon W_{M^2}/(\Fr_\fl-1)W_{M^2}\to W_{M^2}^{\Fr_\fl=1}$
is an isomorphism. 
We now define the isomorphism
$\phi_\fl^{\textnormal{fs}}$ by the commutative diagram
\begin{equation}\label{def:finite_singular}
	\begin{tikzcd}
		H^1_s(\K_\fl, W_{M^2}) \arrow[r,"\alpha_\fl"]&
		W_{M^2}^{\Fr_\fl=1} \\
		H^1_{f}(\K_\fl, W_{M^2}) \arrow[r,"\beta_\fl"]
		\arrow[u,"\phi_\fl^{\textnormal{fs}}"]&
		W_{M^2}/(\Fr_\fl-1)W_{M^2}
		\arrow[u,"Q_\fl(1)",swap]
	\end{tikzcd}
\end{equation}

\begin{prop}\label{prop:fscomparison}

Given $\fs\in\mathcal{R}_M$
and suppose $\fl$ is a prime such that 
$\ell\fs\in\mathcal{R}_M$ for $\ell=\fl\fl^c$.
Let $m\in \eo$ be as in Proposition \ref{prop:goodm}
so that $\loc_\fl(m\cdot \kappa_{\fs,M})\in H^1_f(\K_\fl, W_{M^2})$.
Then we have  
\[
\phi_\fl^{\textnormal{fs}}(\loc_{\fl}(m\cdot \kappa_{\fs,M}))=
\loc^s_\fl(m\cdot \kappa_{\ell\fs,M})
\]
where $\loc_\fl^s$ denote the composition of $\loc_\fl$
with the quotient map to $H^1_s(\K_\fl,W_{M^2})$.

\end{prop}

\begin{proof}

The proposition is a modified version of 
\cite[Thm 4.5.4]{Rubin} and can be proved in the same way.
For the convenience of the reader,
we reproduce below an alternative proof,
which the author learned from Gyujin Oh's notes on 
Skinner's lecture on Euler systems in 2022.

Let $\Gal_\K$ acts on $\tilde{T}=\text{Map}^{\cts}(\Gal_\K,T)$
by $(g\cdot f)(x)=f(xg)$ for $f\in \tilde{T}$
and consider the sequence of $\Gal_\K$-modules
$0\to T\to \tilde{T}\to \tilde{T}/T\to 0$,
where $T\to \tilde{T}$ is the map
$t\mapsto [\gamma\mapsto \gamma\cdot t]$.
Since $T^{\K(\fs)}=0$ by \eqref{cond:gen_psi}
and $H^1(\K(\fs),\tilde{T})=0$ by
Shapiro's lemma, we have the exact sequence of $Gal(\K(\fs)/\K)$-modules
\[
	0\to\tilde{T}^{\K(\fs)}\to (\tilde{T}/T)^{\K(\fs)}\to 
	H^1(\K(\fs),T)\to 0
\]
Then we can pick
$\hat{z}_\fs\in \tilde{T}$ such that
$\hat{z}\bmod T\in (\tilde{T}/T)^{\K(\fs)}$ and the cocycle
\[
    \gamma\mapsto (\gamma-1)\cdot \hat{z}_\fs\in T
\]
defines the same class as $z_\fs\in H^1(\K(\fs),T)$.
We claim that for $\ell\fs, \fs\in\mathcal{R}_M$ as above,
we can pick $\hat{z}_\fs$ and $\hat{z}_{\fs\ell}$ such that 
$\Nr_{\K(\ell)/\K(\id)}\hat{z}_{\fs\ell}-P_\fl(\Fr_\fl)
\hat{z}_{\fs}\in M^2\tilde{T}$.
Then for $w=\fl$ the proposition follows from that
\begin{align*}
	D_{\ell\fs}z_{\ell\fs}(\sigma_\fl)=
	(\sigma_\fl-1)D_{\ell\fs}\hat{z}_{\ell\fs}
	\equiv-\Nr_{\K(\ell)/\K(\id)}D_\fs \hat{z}_{\ell\fs}
	&\equiv-P_\fl(\Fr_{\fl}) D_\fs \hat{z}_\fs \mod M^2\tilde{T}\\
	Q_\fl(\Fr_\fl)\cdot
	D_\fs z_\fs(\Fr_\fl)=
	Q_\fl(\Fr_{\fl})(\Fr_{\fl}-1)
	\cdot D_\fs \hat{z}_\fs
	&\equiv-P_\fl(\Fr_{\fl}) D_\fs \hat{z}_\fs \mod M^2\tilde{T}
\end{align*}

To prove the claim,
pick $k>0$ such that 
$\Fr_\fl^k$ acts trivially on $T/{M^2}T$.
Since $\Fr_\fl$ has infinite order in $\fG_{\fs}^a$,
we can pick $V\subset U$ such that $\bar{M}^2$ divides
the order $N$ of $\Fr_\fl^k$ in $Gal(\K(V)/\K(\id))$.
Pick any lift 
$\hat{z}_{\ell\fs V}$ and  $\hat{z}_{\fs V}$  for
$z_{\ell\fs V}$ and  $z_{\fs V}$.
It then follows from 
$\Cor_{\K(\fs V)}^{\K(\fs V)}(z_{\ell\fs V})=
P_\fl(\Fr_\fl)\cdot z_{\fs V}$ that
\[
	\Nr_{\K(\ell)/\K(\id)}
	\hat{z}_{\ell\fs V}-P_\fl(\Fr_\fl)
	\hat{z}_{\fs V}\in T+\tilde{T}^{\K(\fs w^a)}.
\]
Moreover, since
$\Nr_{\K(\ell)/\K(\id)}\colon\tilde{T}^{\K(\fs\ell V)}\to
\tilde{T}^{\fs V}$ is surjective,
we can modify $\hat{z}_{\ell\fs V}$ so that the difference
above lies in $T$.
Take $\hat{z}_{\ell\fs}=\Nr_{\K(V)/\K(\id)}
\hat{z}_{\ell\fs V}$ and
$\hat{z}_{\fs}=\Nr_{\K(V)/\K(\id)}\hat{z}_{\fs V}$
as lifts of $z_{\fs\ell}$ and $z_{\fs}$.
Since
\[
	\Nr_{\K(V)/\K(\id)}t=
	\sum_{Gal(\K(V)/K(\id))/\langle \Fr_\fl^k\rangle}\gamma\,
	\sum_{i=1}^{N}(\Fr_\fl^k)^i\cdot  t \equiv
	\sum_{Gal(\K(V)/\K(\id))/\langle \Fr_\fl^k\rangle}\gamma\cdot
    Nt\equiv 0
	\mod M^2T\quad
	\text{ for } t\in T
\]
we have 
$\Nr_{\K(\ell)/\K(\id)}\hat{z}_{\ell\fs}-P_\fl(\Fr_\fl)
\hat{z}_{\fs}\in \Nr_{\K(V)/\K(\id)}T\subset M^2\tilde{T}$.
\end{proof}

\begin{lem}\label{lem:vanish}
    Let $\Omega/\K$ be the infinite abelian extension
    which is the product of the fields
    $\K(\id), \K(W)$, and $\K(\mu_{p^\infty})$.
    Then under the assumption \eqref{cond:gen_psi} both 
    $H^1(\Omega/\K, W)$ and $H^1(\Omega/\K, W^*)$ 
	are trivial.
\end{lem}
\begin{proof}
    Write $L=\K(W_{\varpi})$ and let $\Delta=Gal(L/\K)$.
    Then $\Delta$ acts on $W_{\varpi}$
    by the units in the residue field of $\eo$
    and hence $p\neq \#\Delta$.
    Thus $H^i(\Delta, W_\varpi^L)=0$ and we have
	\[
		H^1(\Omega/\K,W_{\varpi})\cong
		H^1(\Omega/L,W_{\varpi})=
		\Hom(Gal(\Omega/L),W_{\varpi})^\Delta=
		\Hom(Gal(\Omega/L),W_{\varpi}^\Delta)=0
	\]
    where the second equality follows from that
    the conjugation by $\Delta$ 
    is trivial on the abelian group $Gal(\Omega/L)$
    and the last comes from 
    $W_{\varpi}^\K=0$ by \eqref{cond:gen_psi}.
    This implies that the cohomology group 
    $H^1(\Omega,W)$
	has no nontrivial $\varpi$-torsion
    by considering the long exact sequence associated to
	$0\to W_\varpi\to W\xrightarrow{\varpi}W\to 0$,
    which implies that $H^1(\Omega,W)=0$.
	The same argument also works for $H^1(\Omega,W^*)$.
\end{proof}

The next lemma is a modification of \cite[Lem 5.2.3]{Rubin}.
Let $\fp=(\varpi)$ be the prime ideal of $\eo$
and let $\ord_\fp\colon E^\times\to \Z$ be the valuation 
such that $\ord_\fp(\varpi)=1$.
If $B$ is an $\eo$ module and $b\in B$, we write
\[
    \ord(b,B)=\inf\{n\geq0\mid \fp^nb=0\}\leq\infty.
\]

\begin{lem}\label{lem:dualC}
Fix $M$ and $\mathcal{R}_M$ as in Definition \ref{def:RM}.
Let  $C=\{\eta_1,\cdots,\eta_k\}$
be a finite subset of elements in $H^1(\K,W_{M}^*)$.
Write $\kappa_{\fs,M}'=m\cdot \kappa_{\fs,M}$
for $\fs\in \mathcal{M}$ and $m\in \eo$
as in Proposition \ref{prop:goodm}.
Then there exists a finite set of primes
$\{\fl_1,\cdots,\fl_k\}$  such that 
$\fs_i=\prod_{j=1}^i\ell_i\in \mathcal{R}_M$
for each $0\leq i\leq k$ and $\ell_i=\fl_i\fl_i^c$
(with $\fs_0=\oo_\K$) and satisfies the following conditions.
\begin{itemize}
\item For $1\leq i\leq k$,
$\textnormal{ord}( \loc_{\fl_i}(\kappa_{\fs_{i-1},M}'),
H^1_f(\K_{\fl_i}, W_{M^2}))\geq
\textnormal{ord}( M\cdot \kappa_{\fs_{i-1},M}', H^1(\Omega,W_{M^2}))$.
\item For any $\eta\in C$ and $\fl\in \{\fl_1,\cdots,\fl_k\}$,
we have $\loc_\fl(\eta)\in H^1_f(\K_{\fl},W_{M}^*)$.
\item For any $\eta\in C$,
if $\loc_{\fl_i}(\eta)=0$ for all $1\leq i\leq k$ then $\eta=0$.
\end{itemize}

\end{lem}

\begin{proof}
Let $L=\K(\id)\K(W_{M^2})\K(\mu_{\bar{M}^2})$
and $\bar{\tau}\in Gal(L/\K)$ be as in the proof of Lemma \ref{lem:estimate}.
Then $\Gal_L$ acts trivially on $W_{M^2}$ with
$(\bar{\tau}-1)W_{M^2}=MW_{M^2}=W_M$ and similarly for $W^*$, hence
\[
    W_{M^2}/(\bar{\tau}-1)W_{M^2}\cong
    \eo/M\eo\cong
    W_{M}^*/(\bar{\tau}-1)W_{M}^*.
\]
To find $\fl_1$,
let $\kappa_1\coloneqq \kappa_{\id,M}'\in H^1(\K,W_{M^2})$ and
$\eta_{1}\in H^1(\K,W_{M}^*)$, we claim there exists
$\gamma\in \Gal_L$ such that 
\begin{align*}
	\ord(\kappa_1(\gamma\bar{\tau}),W_{M^2}/(\bar{\tau}-1)W_{M^2})&\geq
	\ord( M\cdot (\kappa_1)_L, H^1(L,W_{M^2}) )\\
	\ord(\eta_1(\gamma\bar{\tau}),W_{M}^*/(\bar{\tau}-1)W_{M}^*)&\geq
	\ord( M\cdot (\eta_1)_L, H^1(L,W_{M}^*) )
\end{align*}
where $(*)_L$ denote the restriction of the cohomology classes to $L$.
Indeed, the argument in \cite[Lem 5.2.1]{Rubin}
shows that $\kappa_1(\Gal_L)=W_{\fp^d}$ if 
$d=\ord( (\kappa_1)_L, H^1(L,W_{M^2}) )$,
therefore there exists $\gamma'\in \Gal_L$ such that 
\[
	\ord( M\cdot (\kappa_1)_L, H^1(L,W_{M^2}) )=
    \ord(M\cdot \kappa_1(\gamma'),W_{M^2})=
    \ord(\kappa_1(\gamma'),W_{M^2}/(\bar{\tau}-1)W_{M^2})
\]
and $J=\{\gamma\in\Gal_L\mid 
\ord(\kappa_1(\gamma),W_{M^2}/(\bar{\tau}-1)W_{M^2})<
\ord(M\cdot  (\kappa_1)_L, H^1(L,W_{M^2}) )\}$
is a proper subgroup.
After a similar analysis for $\eta_1$,
the existence of $\gamma$
follows from the proof of \textit{loc.cit}.

Let $L'=\ker(\kappa_1)_L\cap \ker(\eta_1)_L$
and pick $\fl_1$ such that 
$\eta_1\in H^1_f(\K_{\fl_1}, W_{M}^*)$ and
$\Fr_{\fl_i}$ is conjugate to 
$\gamma\bar{\tau}$ in $Gal(L'/\K)$.
Then $\ell_1\coloneqq \fl_1\fl_1^c\in \mathcal{R}_M$ 
by Lemma \ref{lem:estimate} since
$\Fr_{\fl_i}$ is conjugate to  $\bar{\tau}$ in $Gal(L/\K)$.
Moreover, apply $\beta_{\fl_1}$ from \eqref{eq:finite} gives
\begin{align*}
	\ord( \loc_{\fl_1}(\kappa_1), H^1_f(\K_{\fl_1},W_{M^2}))=
	\ord(\kappa_1(\gamma\bar{\tau}),W_{M^2}/(\bar{\tau}-1)W_{M^2})&\geq
	\ord( (\kappa_1)_L, H^1(L,W_{M^2}) )\\
	\ord( \loc_{\fl_1}(\eta_1), H^1_f(\K_{\fl_1},W_{M}^*))=
	\ord(\eta_1(\gamma\bar{\tau}),W_{M}^*/(\bar{\tau}-1)W_{M}^*)&\geq
	\ord( (\eta_1)_L, H^1(L,W_{M}^*) ) 
\end{align*}
We now set $\kappa_2=\kappa_{\fs_1,M}'$
and apply the same procedure on $\kappa_2$ and $\eta_2$
to find $\fl_2$ with
$\ell_2\coloneqq \fl_2\fl_2^c\in \mathcal{R}_M$ 
such that similar inequalities as above hold
and iterate the process to find 
the set $\{\fl_1,\cdots,\fl_k\}$.

Now the first claim is clear since $L\subset \Omega$,
the second follows by construction,
and suppose $\eta=\eta_i$ satisfies
the assumption in the last claim,
then in particular $\loc_{\fl_i}(\eta_i)=0$.
By the above inequalities this implies that
$\eta\in H^1(L/\K, W_{M}^*)\subset H^1(\Omega/\K, W_{M}^*)$,
which vanishes be Lemma \ref{lem:vanish}.

\end{proof}

Now, for $w\in \Sigma_p$ 
we put $H^1_f(\K_w, V)=0$
and define $H^1_f(\K_w, W)$ $H^1_f(\K_w, T)$
as the image and pre-image of which in the sequence
$H^1_f(\K_w, T)\to H^1_f(\K_w, V)\to H^1_f(\K_w, W)$.
By \cite[Lem 1.3.8]{Rubin}
we can then define $H^1_f(\K_w, W_{M^2})$
from either $H^1_f(\K_w, T)$ or $H^1_f(\K_w, W)$.
We make similar definitions for the cohomology of the Tate duals
using $H^1_f(\K_w, V^*)=H^1(\K_w, V^*)$.
Then for any finite set $\Pi$ of finite primes of $\K$ we define
\begin{align*}
    &S^{\Sigma_p^c\cup \Pi}(\K, \square)=
    \ker\big\{
    H^1(\K, \square)\to \prod_{w\notin \Sigma_p^c\cup \Pi}H^1_s(\K_w, \square)
    \big\}\\
    &S_{\Sigma_p^c\cup \Pi}(\K, \square^*)=
    \ker\big\{
    H^1(\K, \square^*)\to \prod_{w\notin \Sigma_p^c\cup \Pi}H^1_s(\K_w, \square^*)
    \times \prod_{w\in\Sigma_p^c\cup \Pi}H^1(\K_w,\square^*)
    \big\}
\end{align*}
for $\square\in \{T,V,W,W_{M^2}\}$, where $H^1_s(\K_w,\square)$
denote the quotient of $H^1(\K_w,\square)$ by $H^1_f(\K_w,\square)$.
Note that by definition $H^1_f(K_w,T)=H^1(K_w,T)_{\textnormal{tor}}$
if $w\in\Sigma_p$, 
therefore $\loc_w(\kappa_{\fs_i,M})\in H^1_f(\K, W_{M^2})=H^1(\K_w, T)_{\textnormal{tor}}$
for $w\in\Sigma_p\setminus{w_0}$
by Theorem \ref{thm:eu} if $\mu\vert_{D_w}\neq 1$.
Then for $\fs\in \mathcal{R}_M$ and $m$ as in 
Proposition \ref{prop:goodm} we have
\[
    \kappa_{\fs,M}'=m\cdot \kappa_{\fs,M}
    \in S^{\Sigma_p^c\cup \Pi}(\K, W_{M^2})
    \text{ for }\Pi=\{w_0, \fl_1,\fl_1^c,\cdots, \fl_k,\fl_k^c\}.
\]

We give a modification of Lemma \cite[Lem 5.2.5]{Rubin},
which provide an upper bound for the cokernel of
\[
    \loc_{\Pi\setminus\{w_0\},\square}^s\colon
    S^{\Sigma_p^c\cup \Pi}(\K, \square)
    \xrightarrow{}
    \bigoplus_{w\in\Pi\setminus\{w_0\}}H^1_s(\K_w, \square).
\]

\begin{lem}
Fix a nonzero $\fm\in \eo$ with $n=\ord_\fp(\fm)$ 
and suppose $M\in \eo$ satisfies
\[
    2\ord_\fp(M)\geq 2n+\ind_\eo(z')
\]
where 
$\ind_\eo(z')=\max\{k\geq 0\mid z'\coloneqq m\cdot z\in 
\fp^k H^1(\K, T)\bmod H^1(\K,T)_{\textnormal{tor}}\}$.
Suppose further that $\{\fl_1,\cdots,\fl_k\}$ 
is a finite set of primes such that
$\fs_i=\prod_{j=1}^i\ell_i\in \mathcal{R}_M$
for each $0\leq i\leq k$ and $\ell_i=\fl_i\fl_i^c$ and satisfies 
\[
\textnormal{ord}( \loc_{\fl_i}(\kappa_{\fs_{i-1},M}'),
H^1_f(\K_{\fl_i}, W_{M^2}))\geq
\textnormal{ord}( M\cdot \kappa_{\fs_{i-1},M}', H^1(\Omega,W_{M^2}))
\]
for each $1\leq i\leq k$.
Then we have 
$\length_{\eo}(\coker(\loc_{\Sigma\setminus\{w_0\},W_{\fm}}^s))
\leq \ind_{\eo}(z)$
for $\Pi=\{w_0, \fl_1,\cdots, \fl_k\}$.
\end{lem}
\begin{proof}
Thanks to Lemma \ref{lem:vanish} and \eqref{cond:gen_psi},
the hypotheses \cite[(5.5) p.109]{Rubin} is satisfied 
and we can use the simplified argument on \cite[p.110]{Rubin}.
In particular, for $0\leq i\leq k$ we can define
\[
    \dd_i=\ord(M\cdot \kappa_{\fs_i,M}', H^1(K,W_{M^2}))=
    \ord(M\cdot \kappa_{\fs_i,M}', H^1(\Omega,W_{M^2}))=
    \ord(\kappa_{\fs_i,M}', H^1(\Omega,W_{M^2})/MH^1(\Omega,W_{M^2})),
\]
where the last equality follows from that 
$\kappa_{\fs_i,M}'(\Gal_\Omega)=W_{\fp^d}$ if 
$d=\ord(\kappa_{\fs_i,M}', H^1(\Omega,W_{M^2}))$.

By the argument in \textit{loc.cit}, we have
\[
    \dd_0=\ord_\fp(M)-\ind_\eo(z')\geq 2n.
\]
And since $\alpha_{\fl_i}\circ \loc_{\fl_i}^s$
factors through $H^1(\Omega,W_{M^2})/MH^1(\Omega,W_{M^2})$, we have
\[
    \dd_i\geq 
    \ord(\loc^s_{\fl_i}(\kappa_{\fs_i,M}'),H^1_s(\K_{\fl_i}, W_{M^2}))=
    \ord(\loc_{\fl_i}(\kappa_{\fs_{i-1},M}'),H^1_f(\K_{\fl_i}, W_{M^2}))
    \geq \dd_{i-1}
\]
where the middle equality comes from Proposition \ref{prop:fscomparison}.
We then let $A^{(i)}\subset S^{\Sigma_p^c\cup \Pi}(\K, W_{\fm^2})$
be generated by $\bar{\kappa}_i\coloneqq \fp^{\dd_i-2n}\kappa_{\fs_i,M}'$,
the same argument in \textit{loc.cit} gives
\[
    \ell_\eo(\loc^s_{\Pi\setminus\{w_0\}}
    (S^{\Sigma_p^c\cup \Pi}(\K, W_{\fm^2})))\geq 2kn-\ind_\eo(z').
\]
We can then conclude the lemma since
$\ell_\eo(\bigoplus_{w\in\Pi\setminus\{w_0\}}H^1_s(\K_w, W_{\fm^2}))=2kn$.

\end{proof}

\begin{prop}
    Assume $\mu$ has infinite order and $\mu\vert_{D_w}\neq 1$
    for $w\in\Sigma_p\setminus\{w_0\}$, then
    \[
    \length_\eo(S_{\Sigma_p^c}(\K,W^*))\leq \ind_{\eo}(m\cdot \alpha^{-1}(L_\id)).
    \]
\end{prop}
\begin{proof}
    When $\mu$ has infinite order we can pick $\tau\in \Gal_{\K(1)}$
    so that $\ord_\fp(M)$ is arbitrary large for $M=\psi(\tau)-1$.
    In particular, given $\fm=\fp^n$ we may pick $M$ so that 
    \[
         2\ord_\fp(M)\geq 2n+\ind_\eo(z')\text{ and }
         \ord_\fp(M)\geq 2n.
    \]
    Let $C\subset H^1(\K,W_M^*) $
    be the image of $S_{\Sigma_p^c\cup\{w_0\}}(\K,W_{\fm^2}^*)$
    and apply Lemma \ref{lem:dualC}.
    We can then argue as in \cite[p.114]{Rubin},
    using the previous Lemma, and obtain
    $\length_\eo(S_{\Sigma_p^c\cup \{w_0\}}(\K,W^*))\leq \ind_{\eo}(z')$.
    We then consider the exact sequence
    \[
    0\to S^{\Sigma_p^c}(\K, T)\to 
    S^{\Sigma_p^c\cup\{w_0\}}(\K, T)\xrightarrow{\loc_{w_0}}  
    H^1_s(\Qp,T)
    \]
    Observe that $z'=m\cdot z\in S^{\Sigma_p^c\cup\{w_0\}}(\K,T)$
    and $\loc_{w_0}(z')=m\alpha^{-1}(L_\id)\cdot z_p$
    by Theorem \ref{thm:eu}.
    Since Lemma \ref{lem:classp} implies that
    $H^1(\Qp,T)=\eo z_p$ is torsion free,
    we have $H^1(\Qp,T)=H^1_s(\Qp,T)$ and
    the proposition follows from the same argument in \cite[Thm 2.2.10]{Rubin}.
\end{proof}

Note that however $S_{\Sigma_p^c}(\K,W^*)$ is different
from the Selmer group $\Sel(\eo(\alpha))$ in \eqref{eq:Sel_base}.
Precisely there the local conditions for $\Sel(\eo(\alpha))$
are relaxed at $\Sigma_p$ and we have the exact sequence
\[
    0\to S_{\Sigma_p^c}(\K,W^*)\to 
    \Sel(\eo(\alpha))\to \prod_{w\in\Sigma_p}H^1_s(\K_w,W^*)
\]
where $H^1_f(\K_w,W^*)=H^1(\K_w,W^*)_{\textnormal{div}}$ by definition.
Thus $H^1_s(\K_w,W^*)$ is isomorphic to
$H^2(\K_w, T^*)_{\textnormal{tor}}$, which is isomorphic to
$W^{D_w}/W^{D_w}_{\textnormal{div}}$ by Tate duality.
In particular, $H^1_s(\K_{w_0},W^*)=0$ thanks to 
\eqref{cond:gen_psi}.

On the other hand, if the classes $\{z_{\fs V}\in H^1(\K(\fs V), T)\}$
from Theorem \ref{thm:eu}
were trivial at $w\in \Sigma_p\setminus\{w_0\}$
instead of just possibly torsion 
(which is the main reason that we require 
$\mu\vert_{D_w}\neq 1$ at $w\in \Sigma_p\setminus\{w_0\}$
in the previous proposition),
then the Kolyvagin system $\{\kappa_{\fs,M}'\}$
would lie inside the subspace of
$S^{\Sigma_p\cup\Pi}(\K, W_{M^2})$
that are strict at $w\in \Sigma_p\setminus\{w_0\}$.
The same argument then yield that 
\begin{equation}\label{eq:boundL}
    \length_\eo(\Sel(\eo(\alpha)))\leq \ind_{\eo}(m\cdot \alpha^{-1}(L_\id)).
\end{equation}
This is indeed the case if we pick an $\sigma_w\in D_w$
for each $w\in \Sigma_p\setminus\{w_0\}$
and replace the classes 
$\{\mathcal{Z}_\fs\in H^1(\K, \I_\fs(\Psi_\fs^{-1}))\}$ 
from Theorem \ref{thm:Bigeu} by
\[
    \prod_{w\in\Sigma_p\setminus\{w_0\}}
    (\Psi_\fs^{-1}(\sigma_w)-1)\cdot \mathcal{Z}_\fs.
\]
Then \eqref{eq:boundL} holds if $\mu=\psi\alpha$ has infinite order
and $L_\id$ is replaced with 
$\prod_{w\in\Sigma_p\setminus\{w_0\}}(\Psi^{-1}(\sigma_w)-1)L_\id$.

\begin{thm}

The Pontryagin dual
$X(\Psi,\Sigma_p^c)$ is a torsion $\Lambda_W$-module if 
we assume \ref{cond:K2}-\ref{cond:K3} and
that $\psi$ satisfies \eqref{cond:gen_psi}.
If we further assume that
$\psi\vert_{D_w}\not\equiv\id$  for all $w\in\Sigma_p$,
let $L$ be the image of $L_\id$ in $\Lambda_W$
and let $\iota\colon \Lambda_W\to\Lambda_W$
be the map induced by 
$\langle\gamma\rangle^\iota=\langle\gamma^{-1}\rangle$, then
\[
\car_{\Lambda_W}(X(\Psi,\Sigma_p^c))\supset(L^\iota)
\]
in $\Lambda_W\hat{\otimes}_\eo E$,
where $E$ is the field of fraction of $\eo$.
\end{thm}
\begin{proof}

Let $\mu=\psi\alpha$ be as above.
We note that $\mu$ has finite order if and only if 
each $\alpha(\gamma_i)$,
for a fixed choice of topological generators
$\{\gamma_1,\cdots,\gamma_d\}$ of $W$,
is a $p$-power roots of unity.
Therefore there are only finitely many characters $\mu$
that are of finite order 
since there are only finitely many 
$p$-power roots of unity in a fixed 
finite extension over $\bar{\Q}_p^{un}$.
Therefore the conditions in Proposition \ref{prop:tor_crit} and
Proposition \ref{prop:specialize}
are verified and by Corollary \ref{cor:main} we have
\[
\car_{\Lambda_W}(X(\Psi,\Sigma_p^c))\supset
\big(\prod_{w\in\Sigma_p\setminus\{w_0\}}
(\Psi^{-1}(\sigma_w)-1)L^\iota\big).
\]
Since the choice of $\sigma_w$ is arbitrary, 
we can pick different sets of 
$\{\sigma_w\}_{w\in \Sigma_p\setminus\{w_0\}}$
so that the collection of 
$\big(\prod_{w\in\Sigma_p\setminus\{w_0\}}
(\Psi^{-1}(\sigma_w)-1)$
has no common divisors, we can then conclude
\[
\car_{\Lambda_W}(X(\Psi,\Sigma_p^c))\supset(L^\iota).
\]

\end{proof}

\begin{rem}
Compare with the interpolation formula 
from Proposition \ref{prop:L-function_at_s}
and apply the functional equation we see that
$\alpha(L^\iota)$ is the algebraic part of 
\[
L(0,\psi\alpha)\prod_{w\in\Sigma_p^c}
(1-(\psi\alpha)(\varpi_\bw)q_\bw^{-1})
(1-(\psi\alpha)(\varpi_\bw)).
\]
In other word $L^\iota$ is a multiple
of $L_p(\Psi,\Sigma_p^c)$ in the notation
of \cite{Hsieh2010}
and our result agrees with the main conjecture
stated in \textit{loc.cit}.
\end{rem}

\bibliographystyle{amsalpha}
\bibliography{biblio}
\end{document}